\documentclass{amsart}

\usepackage{xcolor}
\definecolor{nicered}{rgb}{0.6, 0, 0.1}
\definecolor{niceblue}{rgb}{0.06, 0.3, 0.57}
\definecolor{nicegreen}{rgb}{0.0, 0.51, 0.5}
\usepackage{color}

\usepackage[colorlinks=true,pagebackref,hyperindex,citecolor=nicegreen,linkcolor=niceblue,urlcolor=nicered]{hyperref}
\usepackage{amsmath,amsfonts,amssymb}
\usepackage{mathrsfs}
\usepackage{mathtools}
\usepackage{microtype}
\usepackage[all]{xy}
\usepackage[T1]{fontenc}
\usepackage{lmodern}
\usepackage{xparse}
\usepackage{enumitem}
\usepackage{comment}
\setlist[enumerate]{leftmargin=2em,label=\textup{(\roman*)}}
\usepackage{accents}
\newcommand\ubar[1]{%
  \underaccent{\bar}{#1}}
  
\usepackage{mathbbol}
\usepackage{stmaryrd}

\usepackage{cleveref}
\crefname{equation}{Eq.}{Eqs.}
\crefname{theorem}{Theorem}{Theorems} 
\crefname{lemma}{Lemma}{Lemmas}
\crefname{corollary}{Corollary}{Corollaries}
\crefname{proposition}{Proposition}{Propositions}
\crefname{definition}{Definition}{Definitions}
\crefname{remark}{Remark}{Remarks}
\crefname{example}{Example}{Examples}
\crefname{notation}{Notation}{Notations}
\crefname{setup}{Setup}{Setup}
\crefname{question}{Question}{Question}
\crefname{convention}{Convention}{Conventions}

\allowdisplaybreaks

\newtheorem{theorem}{Theorem}[section]
\newtheorem{corollary}[theorem]{Corollary}
\newtheorem{lemma}[theorem]{Lemma}
\newtheorem{proposition}[theorem]{Proposition}

\theoremstyle{definition}
\newtheorem{definition}[theorem]{Definition}

\newtheorem{example}[theorem]{Example}
\newtheorem{remark}[theorem]{Remark}

\newtheorem{conjecture}[theorem]{Conjecture}

\numberwithin{equation}{section}

\newcommand{\lct}{\operatorname{lct}}	

\newcommand{\m}{\mathfrak{m}}
\newcommand{\p}{\mathfrak{p}}
\newcommand{\fa}{\mathfrak{a}}
\newcommand{\fb}{\mathfrak{b}}
\newcommand{\n}{\mathfrak{n}}

\newcommand{\FF}{\mathbb{F}}
\newcommand{\KK}{\mathbb{K}}
\newcommand{\CC}{\mathbb{C}}
\newcommand{\QQ}{\mathbb{Q}}
\newcommand{\RR}{\mathbb{R}}
\newcommand{\cJ}{\mathcal{J}}

\newcommand{\cB}{\mathcal{B}}
\newcommand{\cO}{\mathcal{O}}
\newcommand{\cF}{\mathcal{F}}
\newcommand{\Cech}{ \check{\rm{C}}}

\newcommand{\ZZ}{\mathbb{Z}}
\newcommand{\NN}{\mathbb{N}}

\newcommand{\act}{\mathbin{\vcenter{\hbox{\scalebox{0.7}{$\bullet$}}}}}

\newcommand{\Hom}{\operatorname{Hom}}
\newcommand{\Ext}{\operatorname{Ext}}

\newcommand{\Rank}{\operatorname{rank}}

\newcommand{\End}{\operatorname{End}}
\newcommand{\Der}{\operatorname{Der}}
\newcommand{\Grade}{\operatorname{grade}}
\newcommand{\gd}{\operatorname{gl.dim}}
\newcommand{\Frac}{\operatorname{Frac}}

\newcommand{\Ass}{\operatorname{Ass}}

\newcommand{\Len}{\operatorname{length}}

\newcommand{\gr}{\operatorname{gr}}	
\newcommand{\Ann}{\operatorname{Ann}}	
\newcommand{\e}{\operatorname{e}}

\newcommand{\cC}{\mathcal{C}}

\NewDocumentCommand \pd { m o }
{
\IfNoValueTF {#2}
{ \partial_{#1} }
{ \partial_{#1} \act #2 }
}
\NewDocumentCommand \fpd { m o }
{
\IfNoValueTF {#2}
{ \frac{\partial\ }{\partial #1 } }
{ \frac{\partial #2 }{\partial #1 } }
}

\newcommand{\fs}{\boldsymbol{f^s}}

\newcommand{\R}{R}
\renewcommand{\S}{S}

\NewDocumentCommand \seq { o m }
{
\IfNoValueTF {#1}
{ \ubar{#2} }
{ {#2}_1,\ldots,{#2}_{#1} }
}

\NewDocumentCommand \pt { o m }
{
\IfNoValueTF {#1}
{ #2 }
{ ({#2}_1,\ldots,{#2}_{#1}) }
}

\usepackage[textwidth=3.3 cm,textsize=small,shadow
]{todonotes}

\newcommand{\details}[2][]{} 
 
\title{Bernstein-Sato polynomials in commutative algebra}

\dedicatory{Dedicated to Professor David Eisenbud on the occasion of his seventy-fifth birthday.}

\author[\`Alvarez Montaner et al.]{Josep \`Alvarez Montaner{$^1$}}
\address{Departament de Matem\`atiques  and  Institut de Matem\`atiques de la UPC-BarcelonaTech (IMTech)\\  Universitat Polit\`ecnica de Catalunya }
\email{josep.alvarez@upc.edu}

\author[]{Jack Jeffries${^2}$}
\address{University of Nebraska-Lincoln}
\email{jack.jeffries@unl.edu}

\author[]{Luis N\'u\~nez-Betancourt${^3}$}
\address{Centro de Investigaci\'on en Matem\'aticas, Guanajuato, Gto., M\'exico}
\email{luisnub@cimat.mx}

\thanks{{$^1$}Partially supported by grants  2017SGR-932 (AGAUR) and PID2019-103849GB-I00 (AEI/10.13039/501100011033).}

\thanks{{$^2$}Partially supported by National Science Foundation grant DMS-1606353 and CAREER Award DMS-2044833.}

\thanks{{$^3$}Partially supported by the  CONACYT Grant 284598 and C\'atedras Marcos Moshinsky.}

\subjclass[2010]{Primary: 14F10, 13N10, 13A35, 16S32; Secondary: 13D45, 14B05, 14M25, 13A50.}
\keywords{Bernstein--Sato polynomial, $D$-module, singularities, multiplier ideals.}

\begin{document}
\maketitle

\begin{abstract}
This is an expository survey on the theory of Bernstein-Sato polynomials with special emphasis in its recent developments  and its importance in commutative algebra.
\end{abstract}

\setcounter{tocdepth}{1}

\tableofcontents

\section{Introduction}

The origin of the theory of $D$-modules can be found in the works of Kashiwara \cite{Kas70} and Bernstein  \cite{Ber71, Ber72}.  The motivation behind Bernstein's approach was to give a solution to a question posed by I. M. Gel'fand  \cite{Gelf54} at the 1954 edition of the International Congress of Mathematicians
regarding the analytic continuation of the complex zeta function.  The solution is based on the existence of a polynomial in a single variable satisfying a certain functional equation. This polynomial coincides with the $b$-function developed by Sato in the context of prehomogeneous vector spaces and it is known as the \emph{Bernstein-Sato polynomial}.

The theory of $D$-modules grew up immensely in the $1970$'s and $1980$'s and fundamental results regarding Bernstein-Sato polynomials were obtained by Malgrange \cite{Mal74, BernsteinRationalMalgrange, MalgrangeVfil} and Kashiwara \cite{KashiwaraRationality, KashiwaraVfil}. For instance, they proved the rationality of the roots of the Bernstein-Sato polynomial and related the roots to the eigenvalues of the monodromy  of the Milnor fiber associated to the singularity. Indeed this link is made through the concept of $V$-filtrations and the Hilbert-Riemann correspondence.

The theory of $D$-modules burst into  commutative algebra through the seminal work of Lyubeznik \cite{Lyubeznik93} where he proved some finiteness properties of local cohomology modules. Nowadays, the theory of $D$-modules is an essential tool used in the area and has a prominent role.
For example,  the smallest integer root of the Bernstein-Sato polynomial determines the structure of the localization \cite{WaltherBS},  and thus, using the \v{C}ech complex, it is a key ingredient in the computation of local cohomology modules \cite{Oaku97,OakuLC,OakuLC2,OakuLC3}.  
In addition, several results regarding finiteness aspects of local cohomology were obtained via the existence of the Bernstein-Sato polynomial and related techniques \cite{NBDT,AMHNB}. Finally, there are several invariants that measure singularity that are related to the Bernstein-Sato polynomial 
\cite{ELSV2004,MTW2005,BudurSaito05,BMS2006a}.

In this expository paper we survey several features of the theory of Bernstein-Sato polynomials relating to commutative algebra that have been developed over the last fifteen years or so.  For instance, we discuss a version of Bernstein-Sato polynomial associated to ideals was introduced by Budur, Musta\c{t}\u{a}, and Saito \cite{BMS2006a}. We also  present  a version of the theory for rings of positive characteristic developed by Musta\c{t}\u{a} \cite{MustataBSprime} and furthered by Bitoun \cite{BitounBSpos} and Quinlan-Gallego \cite{EamonBSpos}. Finally,  we treat  a recent extension to certain singular rings  \cite{HsiaoMatusevich,AMHNB,Vfilt}. In addition, we discuss relations between the roots of the Bernstein-Sato polynomial and the poles of the complex zeta function  \cite{Ber71, Ber72} and also  the relation with multiplier ideals and jumping numbers \cite{ELSV2004,BudurSaito05,BMS2006a}.

In this survey we have extended a few results to greater generality than previously in the literature. For instance, we prove the existence of Bernstein-Sato polynomials of nonprincipal ideals for differentiably admissible algebras in   Theorem~\ref{ThmExistinceNonPrincipal}. In  Proposition~\ref{uli},
we show that Walther's proof \cite{WaltherBS} about generation of the localization as a $D$-module also holds for nonregular rings. In Theorem~\ref{thm:finite} we observe conditions sufficient for the finiteness of the associated primes of local cohomology in terms of the existence of the Bernstein-Sato polynomial; this covers several cases where this finiteness result is known. We point out that these results are likely expected by the experts and the proofs are along the lines of previous results. They are in this survey to expand the literature on this subject.

 We have attempted to collect as many examples as possible.
 In particular, Section \ref{SecExamples} is devoted to discuss several examples for classical Bernstein-Sato polynomials.
 In Section \ref{SecNonPrincipalRelative}, we also provide several examples for nonprincipal ideals. In addition, we tried to collect many examples in 
other sections. We also  attempted to present this material in an accessible way for people with no previous experience in the subject.
 
 The theory surrounding the Bernstein-Sato polynomial is vast, and only a portion of it is discussed here. Our most blatant omission is the relation of the roots of Bernstein-Sato polynomials with the eigenvalues of the monodromy of the Milnor fiber \cite{Mal74b}. Another crucial aspect of the theory that is not touched upon here is mixed Hodge modules \cite{Saitomhm}. We also do not discuss the different variants of the Strong Monodromy conjecture which relate the poles of the $p$-adic Igusa zeta function or the topological zeta function with the roots of the Bernstein-Sato polynomial \cite{Igusa_book, Denef_Loeser, Nicaise}. We also omitted computational aspects of this subject \cite{Oaku97,BLAlg}.  
We do not discuss in depth several recent results obtained via representation theory \cite{LRWW,LA}.
We hope the reader of this survey is inspired to learn more and we enthusiastically recommend the surveys of Budur \cite{SurveyBudur, BudurNotes}, Granger \cite{Granger2010}, Saito \cite{Saito_survey}, and Walther \cite{WSurvey,UliAlg} for further insight.

\section{Preliminaries}

\subsection{Differential operators}

\begin{definition}\label{def:WEyl}
	Let $\KK$ be a field of characteristic zero, and let $A$ be either 
	\begin{itemize}
		\item $A=\KK[x_1,\dots,x_d]$, a polynomial ring over $\KK$,
		\item $A=\KK\llbracket x_1,\dots,x_d\rrbracket$, a power series ring over $\KK$, or
		\item $A=\CC\{ x_1,\dots,x_d\}$, the ring of convergent power series in a neighborhood of the origin over $\CC$.
	\end{itemize} The ring of differential operators $D_{A|\KK}$ is the $\KK$-subalgebra of $\End_{\KK}(A)$ generated by $A$ and $\partial_1,\dots,\partial_d$, where $\partial_i$ is the derivation $\frac{\partial}{\partial x_i}$.
\end{definition}

In the polynomial ring case, $D_{A|\KK}$ is the Weyl algebra. We refer the reader to books on this subject \cite{Coutinho},  \cite[Chapter~15]{McC_Rob} for  a basic introduction to this ring and its modules. The Weyl algebra can be described in terms of generators and relations as
 
 \[D_{A|\KK} = \frac{\KK\langle x_1,\dots,x_d,\partial_1,\dots,\partial_d \rangle}{(\partial_i x_j - x_j \partial_i - \delta_{ij} \ | \ i,j=1,\dots,d)},\]
 where $\delta_{ij}$ is the Kronecker delta. As $D_{A|\KK}$ is a subalgebra of $\End_{\KK}(A)$, $x_i\in D_{A|\KK}$ is the operator of multiplication by $x_i$. The ring $D_{A|\KK}$ has an order filtration
  \[ D^i_{A|\KK} = \bigoplus_{\substack{a_1,\dots,a_d\in \NN \\  b_1+\cdots+b_d \leq i }}\KK \cdot  x_1^{a_1} \cdots x_d^{a_d} \partial_1^{b_1} \cdots \partial^{b_d}_d .\]
The associated graded ring of $D_{A|\KK}$ with respect to the order filtration is a polynomial ring in $2d$ variables. Many good properties follow from this, for instance,  the Weyl algebra is left-Noetherian,  is right-Noetherian, and has finite global dimension.

In the generality of Definition~\ref{def:WEyl}, the associated graded ring of $D_{A|\KK}$ with respect to the order filtration is a polynomial ring over $A$.

Rings of differential operators are defined much more generally as follows.

\begin{definition}
	Let $\KK$ be a field, and $R$ be a $\KK$-algebra.
\begin{itemize}
	\item $D^0_{R|\KK} = \Hom_R(R,R) \subseteq \End_{\KK}(R)$.
	\item Inductively, we define $D^i_{R|\KK}$  as 
	\[\{ \delta \in \End_{\KK}(R) \ | \ \delta \circ \mu - \mu \circ \delta \in D^{i-1}_{R|\KK} \ \text{for all} \ \mu\in D^0_{R|\KK}\}.\]
	\item $D_{R|\KK} =\bigcup_{i\in \NN} D^i_{R|\KK}$.
\end{itemize}
We call $D_{R|\KK}$ the ring of ($\KK$-linear) differential operators on $R$, and 
\[D^0_{R|\KK} \subseteq D^1_{R|\KK} \subseteq D^2_{R|\KK}  \subseteq \cdots\] the order filtration on $D_{R|\KK}$.
\end{definition}

We refer the interested reader to classic literature on this subject, e.g., \cite[\S 16.8]{EGA}, \cite{Bjork79}, \cite{Nakai}, and \cite[Chapter 15]{McC_Rob}.
We now present a few examples of rings of differential operators.

\begin{enumerate}
	\item If $A$ is a polynomial ring over a field $\KK$, then

		\[ D^i_{A|\KK} = \bigoplus_{a_1 + \cdots +a_d \leq i} A \cdot \frac{\partial^{a_1}_{1}}{a_1!} \cdots \frac{\partial^{a_d}_{d}}{a_d!},\]
	where $\displaystyle \frac{\partial^{a_i}_{i}}{a_i!}$ is the $\KK$-linear operator given by 
	\[\frac{\partial^{a_i}_{i}}{a_i!}(x_1^{b_1} \cdots x_d^{b_d}) =  \binom{b_i}{a_i}  x_1^{b_1} \cdots x_i^{b_i-a_i} \cdots x_d^{b_d}.\]
	Here, we identify an element $a\in A$ with the operator of multiplication by~$a$.
	In particular, when $\KK$ has characteristic zero, this definition agrees with Definition~\ref{def:WEyl}.
	\item If $R$ is essentially of finite type over $\KK$, and $W\subseteq R$ is multiplicatively closed, then $D^i_{W^{-1}R|\KK}=W^{-1}D^i_{R|\KK}$. In particular, for $R=\KK[x_1,\dots,x_d]_f$, 
	\[D^i_{R|\KK}=\bigoplus_{a_1 + \cdots +a_d \leq i} K[x_1,\dots,x_d]_f \cdot \frac{\partial^{a_1}_{1}}{a_1!} \cdots \frac{\partial^{a_d}_{d}}{a_d!} .\]
	\item If $A$ is a polynomial ring over $\KK$, and $R=A/\fa$ for some ideal $\fa$, then
	\[ D^i_{R|\KK} = \frac{ \{ \delta\in D^i_{A|\KK} \ | \ \delta(\fa)\subseteq \fa \} } {\fa D^i_{A|\KK}}. \]
	\end{enumerate}

In general, rings of differential operators need not be left-Noetherian or right-Noetherian, nor have finite global dimension \cite{DiffNonNoeth}.

We note that if $R$ is an $\NN$-graded $\KK$-algebra, then $D_{R|\KK}$ admits a compatible $\ZZ$-grading via $\deg(\delta)=\deg(\delta(f))-\deg(f)$ for all homogeneous $f\in R$.

\begin{remark}\label{rem:loc} The ring $R$ is tautologically a left $D_{R|\KK}$-module. Every localization of $R$ is a $D_{R|\KK}$-module as well. For $\delta\in D_{R|\KK}$, and $f\in R$, we define $\delta^{(j),f}$ inductively as $\delta^{(0),f}=\delta$, and $\delta^{(j),f} =\delta^{(j-1),f} \circ f - f \circ \delta^{(j-1),f}$. The action of $D_{R|\KK}$ on $W^{-1}R$ is then given by
\[ \delta \cdot \frac{r}{f} = \sum_{j=0}^{t} \frac{\delta^{(j),f}(r)}{f^{j+1}} \]
for $\delta\in D^t_{R|\KK}$, $r\in R$, $f\in W$.
\end{remark}

\begin{definition}
	Let $\fa\subseteq R$ be an ideal and   $F=f_1,\ldots, f_\ell\in R$ be a set of generators for $\fa$. 
	Let $M$ be any
	$R$-module. The $\check{\mbox{C}}$ech complex of $M$ with respect to
	$F$ is defined by
	$$
	\Cech^\bullet(F;M): \hskip 3mm 0\to M\to \bigoplus_i
	M_{f_i}\to\bigoplus_{i,j} M_{f_i f_j}\to \cdots \to M_{f_1 \cdots
		f_\ell} \to 0,
	$$
	where the maps on every summand are localization maps up to a sign.
	The local cohomology of $M$ with support on $\fa$ is defined by
	$$
	H^i_{\fa}(M)=H^i(\Cech^\bullet(F;M)).
	$$ 
	This module is independent of the set of generators of $\fa$.
\end{definition}

As a special case, $\displaystyle H^1_{(f)}(R)=\frac{R_f}{R}$.

The \v{C}ech complex of any left $D_{R|\KK}$-module with respect to any sequence of elements is a complex of $D_{R|\KK}$-modules, and hence the local cohomology of any $D_{R|\KK}$-module with respect to any ideal is  a left $D_{R|\KK}$-module.

\subsection{Differentiably admissible $\KK$-algebras} \label{SubSecPreDiffAdm}
In this subsection we introduce what is called now differentiably admissible algebras.  To the best of our knowledge, this is the more general class of ring where the existence of the Bernstein-Sato polynomial is known. 
We follow the extension  done for Tate and  Dwork-Monsky-Washnitzer  $\KK$-algebras by  Mebkhout and Narv\'{a}ez-Macarro \cite{MNM},  which was extended by the third-named author to differentiably admissible algebras \cite{NBDT}. We assume that $\KK$ is a field of characteristic zero.

\begin{definition}
Let $A$ be a Noetherian regular $\KK$-algebra  of dimension $d$. We say that $A$ is  differentiably admissible  if 
\begin{enumerate}
\item $\dim (A_\m)=d$ for every maximal ideal $\m\subseteq A$,
\item $A/\m$ is an algebraic extension of $\KK$  for every maximal ideal $\m\subseteq A$, and
\item $\Der_{A|\KK}$ is a projective $A$-module of rank $d$ such that the natural map 
$$
A_\m \otimes_A \Der_{A | \KK} \to \Der_{A_\m|\KK}
$$ 
 is an isomorphism. 
\end{enumerate}
\end{definition}

\begin{example} \label{ExDifAd}
The following are examples of differentiably admissible algebras:
\begin{enumerate}
	\item Polynomial rings over $\KK$.
	\item Power series rings over $\KK$.
	\item The ring of convergent power series in a neighborhood of the origin over $\CC$.
	\item Tate and  Dwork-Monsky-Washnitzer  $\KK$-algebras \cite{MNM}.
	\item The localization of a complete regular rings of mixed characteristic at the uniformizer \cite{NBDT,LyuMixChar}.
	\item Localization of complete local domains of equal-characteristic zero at certain elements \cite{Put}.
\end{enumerate}
\end{example}

We note that in the Examples \ref{ExDifAd}(i)-(iv), we have that $\Der_{A|\KK}$  is free, because there exists 
$x_1,\ldots,x_d\in R$ and $\partial_1,\ldots, \partial_d\in \Der_{A|\KK}$  such that $\partial_i (x_j)=\delta_{i,j}$ \cite[Theorem 99]{Matsumura}.

\begin{theorem}[{\cite[Theorem 2.7]{NBDT}}]
Let $A$ be a differentiably admissible $\KK$-algebra.
 If there is an element $f\in A$
such that $R=A/fA$ is a regular ring, then $R$ is a differentiably admissible $\KK$-algebra.
\end{theorem}

\begin{remark}[{\cite[Proposition 2.10]{NBDT}}]\label{RmkDiffAdGenDer}
Let $A$ be a differentiably admissible $\KK$-algebra.
Then,
\begin{enumerate}
\item $D^n_{A|\KK}=(\Der_{A|\KK}+A)^n$, and
\item $D_{A|\KK}\cong  A\langle \Der_{A|\KK}\rangle $.
\end{enumerate}
\end{remark}

\begin{theorem}[{\cite[Section 2]{NBDT}}]\label{ThmDifType}
Let  $A$  be  a differentiably admissible $\KK$-algebra.
Then,
\begin{enumerate}
\item $D_{A|\KK}$ is left and right Noetherian;
\item $\gr_{D^\bullet_{A|\KK}}(D_{A|\KK})$ is a regular ring of pure graded dimension $2d$;
\item $\gd(D_{A|\KK})=d$.
\end{enumerate}
\end{theorem}

We recall that for Noetherian rings the  left and right global dimension are equal. In fact, this number is also equal to the weak global dimension \cite[Theorem~8.27]{Rot}.

\begin{definition}[{\cite{MNM}}]
We say that $D_{A|\KK}$ is a  \emph{ring of differentiable type}
if
\begin{enumerate}
	\item $D_{A|\KK}$ is left and right Noetherian,
	\item $\gr_{D^\bullet_{A|\KK}}(D_{A|\KK})$ is a regular ring of pure graded dimension $2d$, and
	\item $\gd(D_{A|\KK})=d$.
\end{enumerate}
\end{definition}

By Theorem~\ref{ThmDifType}, the ring of differential operators of any differentiably admissible algebra is a ring of differentiable type.

\subsection{Log-resolutions} \label{Log-resolutions}

Let $A=\CC[x_1,\dots,x_d]$ be the polynomial ring over the complex numbers and set $X=\CC^d$. 
A  \emph{log-resolution} of an ideal $\mathfrak{a}\subseteq A$ is a proper birational morphism
  $\pi: X' \rightarrow X$  such that $X'$ is smooth, $\mathfrak{a}\cdot\cO_{X'} = \cO_{X'}\left(-F_\pi\right)$ for some effective Cartier divisor $F_\pi$ and $F_\pi +E$ is a simple normal crossing divisor where $E={\rm Exc}(\pi)=\sum_{i=1}^r E_i$ denotes the exceptional divisor. We have a decomposition $F_\pi = F_{exc}+F_{aff}$ into its \emph{exceptional} and \emph{affine} parts which we denote \[ F_\pi := \sum_{i=1}^r N_i E_i + \sum_{j=1}^s N'_j S_j \] with $N_i, N'_j$ being nonnegative integers. For a principal ideal  $\mathfrak{a}=(f)$  we have that $F_\pi =\pi^\ast f$ is the total transform divisor and  $S_j$ are the irreducible components of the \emph{strict transform} of \(f\).  In particular $N'_j=1$ for all $j$ when \(f\) is reduced.

The \emph{relative canonical divisor} \[ K_\pi := \sum_{i=1}^r k_i E_i \]  is the effective divisor with exceptional support defined by the Jacobian determinant of the morphism $\pi$.

There are many invariants of singularities that are defined using log-resolutions but for now we only focus on \emph{multiplier ideals}.  We introduce the basics on these invariants and we refer the interested reader to Lazarsfeld's book \cite{Laz2004}. We also want to point out that there is an analytical definition of these ideals that we consider in \Cref{Multipliers}. 

\begin{definition}\label{Mult_ideal}
The multiplier ideal  associated to an ideal $\fa \subseteq A$ and $\lambda \in \RR_{\geq 0}$ is defined as
\[\mathcal{J}(\mathfrak{a}^\lambda) = \pi_*\mathcal{O}_{X'}\left(\left\lceil K_{\pi} - \lambda F_\pi \right\rceil\right)= \{g\in A \hskip 2mm | \hskip 2mm  {\rm ord}_{E_i}(\pi^\ast g) \geq \lfloor  \lambda e_i - k_i \rfloor \hskip 2mm \forall i \}.\]
\end{definition}

An important feature is that
 \(\mathcal{J}(\mathfrak{a}^\lambda)\) does not depend on the log-resolution ${\pi: X' \rightarrow X}$. Moreover we have
\(R^i\pi_*\mathcal{O}_{X'}\left(\left\lceil K_{\pi} - \lambda F_\pi \right\rceil\right)= 0\) for all $i>0$.

From its definition we deduce that multiplier ideals satisfy the following properties:

\begin{proposition}\label{PropBasicsTestIDeals}
Let  $\fa,\fb\subseteq A$ be ideals, and $\lambda,\lambda'\in\RR_{\geq 0}.$
Then,
\begin{enumerate}
\item If $\fa\subseteq \fb,$ then $\mathcal{J}(\fa^\lambda)\subseteq \mathcal{J}(\fb^{\lambda})$.
\item If $\lambda<\lambda',$ then $\mathcal{J}(\fa^{\lambda'})\subseteq \mathcal{J}(\fa^{\lambda})$.
\item There exists $\epsilon>0$ such that $\mathcal{J}(\fa^{\lambda})= \mathcal{J}(\fa^{\lambda'})$,
if $\lambda'\in [\lambda,\lambda+\epsilon)$.
\end{enumerate}
\end{proposition}

\begin{definition}
We say that  $\lambda$ is a \textit{jumping number}
of $\fa$ if
$$
\mathcal{J}(\fa^{\lambda})\neq \mathcal{J}(\fa^{\lambda-\epsilon})
$$
for every $\epsilon>0.$
\end{definition}

Notice that jumping numbers have to be rational 
and we have a nested filtration
$$A\supsetneqq \mathcal{J}(\mathfrak{a}^{\lambda_1})\supsetneqq \mathcal{J}(\mathfrak{a}^{\lambda_2})\supsetneqq \cdots \supsetneqq \mathcal{J}(\mathfrak{a}^{\lambda_i})\supsetneqq \cdots$$ 
where the jumping numbers  are the $\lambda_i$ where  we have a strict inclusion and $\lambda_1={\rm lct}(\mathfrak{a})$ is the so-called \emph{log-canonical threshold}. Skoda's theorem  states that $\mathcal{J}(\mathfrak{a}^{\lambda}) =\mathfrak{a} \cdot \mathcal{J}(\mathfrak{a}^{\lambda -1}) $ for all $\lambda > \dim A$.

Multiplier ideals can be generalized without much effort  to the case where $X$ is a normal $\QQ$-Gorenstein variety over a field $\KK$ of characteristic zero; one needs to consider $\mathbb{Q}$-divisors. Fix a log-resolution $\pi: X' \rightarrow X$ and let $K_X$  be  a canonical divisor on $X$ which is $\QQ$-Cartier with index $m$ large enough. Pick a canonical divisor $K_{X'}$ in $X'$ such that $\pi_{\ast} K_{X'}=K_X$. Then, the relative canonical divisor is 
\[K_\pi = K_{X'} - \frac{1}{m}\pi^{\ast}(m K_X)\]  and the multiplier ideal of an ideal $\fa \subseteq \mathcal{O}_X$ is \(\mathcal{J}(\mathfrak{a}^\lambda) = \pi_*\mathcal{O}_{X'}\left(\left\lceil K_{\pi} - \lambda F_\pi \right\rceil\right)\).

A version of multiplier ideals for normal varieties has been given by de Fernex and Hacon \cite{dFH}. In this generality  we ensure the existence of canonical divisors that are not necessarily $\QQ$-Cartier. Then we may find some effective boundary divisor $\Delta$ such that $K_X + \Delta $ is $\QQ$-Cartier with index $m$ large enough. Then we consider $$K_\pi= K_{X'} - \frac{1}{m}\pi^{\ast}(m(K_X+\Delta)) $$ and the multiplier ideal 
$\cJ(\fa^{\lambda}, \Delta) = \pi_*\mathcal{O}_{X'}\left(\left\lceil K_\pi - \lambda F \right\rceil\right)$ which depends on $\Delta$. This construction allowed de Fernex and Hacon to  define the multiplier ideal $\cJ(\fa^{\lambda})$  associated to $\fa$ and $\lambda$ as the unique maximal element of the set of multiplier ideals $\cJ(\fa^{\lambda}, \Delta)$ where $\Delta$ varies among all the effective divisors such that $K_X + \Delta $ is $\QQ$-Cartier. A key point proved in \cite{dFH} is the existence of such a divisor $\Delta$ that realizes the multiplier ideal as $\cJ(\fa^{\lambda})=\cJ(\fa^{\lambda}, \Delta) $.

\subsection{Methods in prime characteristic}
In this  section we recall definitions and results in prime characteristic that are used in Section \ref{sec:positivechar}.
We focus on Cartier operators, differential operators, and test ideals.

Let $R$ be a ring of prime characteristic $p$.
The Frobenius map $F:R\to R$ is defined by $r\mapsto r^p$.
We denote by $F^e_* R$ the $R$-module that is isomorphic to $R$ as an Abelian group with the sum and the scalar multiplication is given by the $e$-th iteration of Frobenius. To distinguish  the elements  of  $F^e_* R$ from $R$ we write them as $F^e_* f$. In particular, $r\cdot F^e_* f=F^e_* (r^{p^e}f)$. Throughout this subsection we assume that 
$F^e_* R$ is a finitely generated  $R$-module: that is, $R$ is \emph{$F$-finite}.

\begin{definition} \label{FrobeniusCartier}
Let $R$ be an $F$-finite ring.
\begin{enumerate}
\item  An additive map $\psi:R\to R$ is a \textit{$p^{e}$-linear map} if
$\psi(r f)=r^{p^e}\psi(f).$ Let $\mathcal{F}^e_R$ be the set of all
the $p^{e}$-linear maps. 
\item  An additive map $\phi:R\to R$ is a \textit{$p^{-e}$-linear ma}p if
$\phi(r^{p^e} f)=r\phi(f).$ Let $\cC^e_R$ be the set of all the
$p^{-e}$-linear maps. 
\item An additive map $\delta:R\to R$ is a \textit{differential operator of level $e$} if it is $R^{p^e}$-linear.
Let $D^{(e)}_R$ be the set of all 
differential operator of level $e$.
\end{enumerate}
\end{definition}

Differential operators relate to the Frobenius map in the following important way.
This alternative characterization of the ring of differential operators  is used in Section \ref{sec:positivechar}.

\begin{theorem}[{\cite[Theorem~2.7]{SmithSP}, \cite[Theorem~1.4.9]{yekutieli.explicit_construction}}]\label{thm:opscharp} Let $R$ be a finitely generated algebra over a perfect field $\KK$.
Then
	\[D_{R|\KK} = \bigcup_{e\in \NN} D^{(e)}_R =\bigcup_{e\in \NN} \Hom_{R^{p^e}}(R,R).\]
	In particular, any operator of degree $\leq p$ is $R^{p}$-linear.
	\end{theorem}

\begin{remark}\label{RemCorresCartier}
Suppose that $R$ is a reduced ring. Then, we may identify $F^e_* R = R^{1/p^e}$. We have that
\begin{enumerate}
\item $\cF^e_R\cong \Hom_R(R,F^e_* R)$,
\item $\cC^e_R\cong \Hom_R(F^e_* R,R)$, and 
\item $D^{(e)}_R\cong \Hom_R(F^e_* R,F^e_* R)$.
\end{enumerate}
\end{remark}

\begin{remark}
Let $A$ be a regular $F$-finite ring. Then,
$$\cC^e_A \otimes_A \mathcal{F}^e_A
\cong D^{(e)}_A.$$
This can be reduced to the case of a  complete regular local ring. In this case, one can construct explicitly a free basis for 
$F^e_* A$ as $A$ is a power series over an $F$-finite  field.
Then, it follows that 
$\cC^e_A$, $ \mathcal{F}^e_A$, and $D^{(e)}_A$ are free $A$-modules.
From this it follows  that
$\cC^e_A \fa=\cC^e_A\fb$ if and only $D^{(e)}_{A} \fa=D^{(e)}_{A} \fb$
for any two ideals $\fa,\fb	\subseteq A$.
\end{remark}

We now focus on test ideals.
These ideals have been a fundamental tool
to study singularities in prime characteristic. 
They were first introduced by means of  tight closure
developed by Hochster and Huneke \cite{HoHuStrong,HoHu1,HoHu2,HoHu3}.
Hara and Yoshida \cite{HY2003} extended the theory to include
test ideals  of pairs.
An  approach to test ideals by means of Cartier operators was given by
Blickle, Musta\c{t}\u{a}, and Smith  \cite{BMS2008,BMS2009}  in the case that
$A$ is a regular ring.  Test ideals have also been studied in singular rings via Cartier maps  \cite{TestQGor,BB-CartierMod,BlickleP-1maps}.

\begin{definition}
Let $A$ be an $F$-finite regular ring. The test ideal  associated to an ideal $\fa \subseteq A$ and $\lambda \in \RR_{\geq 0}$ is
defined by
$$
\tau_A(\fa^{\lambda})=\bigcup_{e\in\NN}\cC^e_A \fa^{\lceil p^e \lambda \rceil}.
$$
\end{definition}

We note that the chain of ideals $\{ \cC^e_A I^{\lceil p^e \lambda
\rceil}\}$   is increasing \cite{BMS2008}, and so,
$\tau_A(\fa^{\lambda})=\cC^e_A \fa^{\lceil p^e \lambda \rceil}$ for
$e\gg 0$.

We now summarize basic well-known properties of test ideals.

\begin{proposition}[{\cite{BMS2008}}]\label{PropBasicsTestIDeals}
Let $A$ be an $F$-finite regular ring, $\fa,\fb\subseteq A$ ideals, and $\lambda,\lambda'\in\RR_{>0}.$
Then,
\begin{enumerate}
\item If $\fa\subseteq \fb,$ then $\tau_A(\fa^\lambda)\subseteq \tau_A(\fb^{\lambda})$.
\item If $\lambda<\lambda',$ then $\tau_A(\fa^{\lambda'})\subseteq \tau_A(\fa^{\lambda})$.
\item There exists $\epsilon>0,$ such that $\tau_A(\fa^{\lambda})= \tau_A(\fa^{\lambda'})$,
if $\lambda'\in [\lambda,\lambda+\epsilon)$.
\end{enumerate}
\end{proposition}

In this way, to every ideal $\fa \subseteq A$ is associated a family of test ideals $\tau_A(\fa^{\lambda})$
parameterized by real numbers $\lambda \in \mathbb{R}_{>0}$. Indeed, they form a nested chain of ideals.
The real numbers where the test ideals change are called {\it $F$-jumping numbers}. To be precise:

\begin{definition}
Let $A$ be an $F$-finite regular ring and let $\fa\subseteq A$ be an ideal. A real number $\lambda$ is an \textit{$F$-jumping number}
of $\fa$ if
$$
\tau_A(\fa^{\lambda})\neq \tau_A(\fa^{\lambda-\epsilon})
$$
for every $\epsilon>0.$
\end{definition}

\section{The classical theory for regular algebras in characteristic zero}\label{sec:def-first-prop}

\subsection{Definition of the Bernstein-Sato polynomial of an hypersurface}\label{subsec:GeneralDef}

One basic reason that the ring of differential operators is useful is that we can use its action on the original ring to ``undo'' multiplication on $A$: we can bring nonunits in $A$ to units by applying a differential operator. The Bernstein-Sato functional equation yields a strengthened version of this principle. Before we state the general definition, we consider what is perhaps the most basic example.

\begin{example}
 Consider the variable $x\in \KK[x]$. Differentiation by $x$ not only sends $x$ to $1$, but, moreover, decreases powers of $x$:
\begin{equation}\label{Eq B-S OneVar} 
 \partial_x x^{s+1} = (s+1) x^s \qquad \text{for all} \ s\in \NN.
\end{equation}
	In this equation, we were able to use one fixed differential operator to turn any power of $x$ into a constant times the next smaller power. Moreover, the constant we obtain is a linear function of the exponent $s$.
\end{example}

The functional equation arises as a way to  obtain a version for  Equation \ref{Eq B-S OneVar} for any element in a $\KK$-algebra.

\begin{definition}
Let $\KK$ a field of characteristic zero and $A$ be a regular  $\KK$-algebra.
	A \emph{Bernstein-Sato functional equation} for an element $f$ in $A$ is an equation of the form
	\[ \delta(s) f^{s+1} = b(s) f^s \qquad \text{for all \ } s\in \NN,\]
	where $\delta(s)\in D_{A|\KK}[s]$ is a polynomial differential operator, and $b(s)\in \KK[s]$ is a polynomial. We say that such a functional equation is nonzero if $b(s)$ is nonzero; this implies that $\delta(s)$ is nonzero as well. We may say that $(\delta(s),b(s))$ as above \emph{determine a functional equation for $f$}.
\end{definition}

\begin{theorem}\label{exists-BS}
	Any nonzero element $f\in A$ satisfies a nonzero Bernstein-Sato functional equation. That is, there exist $\delta(s)\in D_{A|\KK}[s]$ and $b(s)\in \KK[s]\smallsetminus\{0\}$ such that
		\[ \delta(s) f^{s+1} = b(s) f^s \qquad \text{for all \ } s\in \NN.\]
\end{theorem}

We pause to make an observation. Fix $f\in A$, and suppose that $(\delta_1(s),b_1(s))$ and $(\delta_2(s),b_2(s))$ determine two Bernstein-Sato functional equations for $f$:
	\[ \delta_i(s) f^{s+1} = b_i(s) f^s \qquad \text{for all \ } s\in \NN \ \text{for} \ i=1,2.\]
Let $c(s)\in \KK[s]$ be a polynomial. Then
\[  (c(s) \delta_1(s) + \delta_2(s)) f^{s+1} = (c(s) b_1(s) + b_2(s)) f^s  \qquad \text{for all \ } s\in \NN.\]
It follows that, for $f\in A$,
\[ \{ b(s) \in \KK[s] \ | \ \exists \delta(s)\in D_{A|\KK}[s] \ \text{ such that} \ \delta(s) f^{s+1} = b(s) f^{s} \ \text{for all} \ s\in \NN \} \]
is an ideal of $\KK[s]$. By Theorem~\ref{exists-BS}, this ideal is nonzero.

\begin{definition}\label{DefBSPoly}
	The \emph{Bernstein-Sato polynomial} of $f\in A$ is the minimal monic generator of the ideal 
	\[ \{ b(s) \in \KK[s] \ | \ \exists \,\delta(s)\in D_{A|\KK}[s] \ \text{such that} \ \delta(s) f^{s+1} = b(s) f^{s} \ \text{for all} \ s\in \NN \} \subset \KK[s].\]
	This polynomial is denoted $b_f(s)$.
\end{definition}

The polynomial described in Definition~\ref{DefBSPoly}  was originaly introduced  in independent constructions by Bernstein \cite{Ber71, Ber72} to establish meromorphic extensions of distributions, and by Sato \cite{MR595585,MR1086566} as the $b$-function in the theory of prehomogeneous vector spaces.

\subsection{The $D$-modules $D_{A|\KK}[s] \boldsymbol{f^s}$ and $A_f[s] \boldsymbol{f^s}$}\label{Subsec:modules}

For the proof of Theorem~\ref{exists-BS} and for many applications, it is preferable to consider the Bernstein-Sato functional equation as a single equality in a $D_{A|\KK}[s]$-module where $f^s$ is replaced by a formal power ``$\boldsymbol{f^s}$.'' We are interested in two such modules that are closely related:
\[D_{A|\KK}[s] \boldsymbol{f^s} \subseteq A_f[s] \boldsymbol{f^s}.\]
 We give a couple different constructions of each. For much more on these modules, we refer the interested reader to Walther's survey \cite{WSurvey}.

\subsubsection{Direct construction of $A_f[s] \boldsymbol{f^s}$}

\begin{definition}
	We define the left $D_{A_f|\KK}[s]$-module $A_f[s] \boldsymbol{f^s}$ as follows:
	\begin{itemize}
	\item As an $A_f[s]$-module, $A_f[s] \boldsymbol{f^s}$ is a free cyclic module with generator $\boldsymbol{f^s}$.
	\item Each partial derivative $\partial_i$ acts by the rule
	\[ \partial_i (a(s) \boldsymbol{f^s}) = \left(\partial_i(a(s)) +  \frac{s a(s) \partial_i(f)}{f}\right) \boldsymbol{f^s} \]
	for $a(s)\in A_f[s]$.
	\end{itemize}
\end{definition} 

We often consider this as a module over the subring $D_{A|\KK}[s]\subseteq D_{A_f|\KK}[s]$ by restriction of scalars.
To justify that this gives a well-defined $D_{A_f|\KK}[s]$-module structure, one checks that $\partial_i (x_i a(s) \boldsymbol{f^s}) = x_i\partial_i (a(s) \boldsymbol{f^s}) + a(s) \boldsymbol{f^s} $.

From the definition, we see that this module is compatible with specialization $s \mapsto n\in \ZZ$. Namely, for all $n\in \ZZ$, define the specialization maps 
\[\theta_n:A_f[s] \boldsymbol{f^s} \to A_f \quad  \text{by} \quad \theta_n(a(s) \boldsymbol{f^s}) = a(n) f^n\]
and
\[ \pi_n: D_{A_f|\KK}[s] \to D_{A_f|\KK} \quad  \text{by} \quad \pi_n(\delta(s)) = \delta(n).\]
We then have $\pi_n(\delta(s)) \cdot \theta_n(a(s) \boldsymbol{f^s}) = \theta_n(\delta(s) \cdot a(s) \boldsymbol{f^s})$. This simply follows from the fact that the formula for $\partial_i (a(s) \boldsymbol{f^s})$ in the definition agrees with the power rule for derivations when $s$ is replaced by an integer $n$ and $\boldsymbol{f^s}$ is replaced by $f^n$.

\subsubsection{Local cohomology construction of $A_f[s] \boldsymbol{f^s}$}\label{subsub:localcohomconstruction} It is also advantageous to consider $A_f[s] \boldsymbol{f^s}$ as a submodule of a local cohomology module.

Consider the local cohomology module $H^1_{(f-t)}(A_f[t])$, where $t$ is an indeterminate over $A$. As an $A_f$-module, this is free with basis
\begin{equation}\label{eq-basis-H1Aft} \left\{ \left[\frac{1}{f-t}\right], \left[\frac{1}{(f-t)^2}\right] , \left[\frac{1}{(f-t)^3}\right], \ldots \right\}:\end{equation}
indeed, these are linearly independent over $A_f$, and we can rewrite any element
\[ \left[ \frac{p(t)}{(f-t)^m} \right]\in H^1_{(f-t)}(A_f[t]), \;  \text{with}\ p(t)\in A_f[t] \]
in this form by writing $t=f-(f-t)$, expanding, and collecting powers of $f-t$.
By Remark~\ref{rem:loc}, $H^1_{(f-t)}(A_f[t])$ is naturally a $D_{A_f[t]|\KK}$-module.

Consider the subring $D_{A_f|\KK}[-\partial_t t] \subseteq D_{A_f[t]|\KK}$. We note that $-\partial_t t$ commutes with every element of $D_{A_f|\KK}$ and that $-\partial_t t$ does not satisfy any nontrivial algebraic relation over $D_{A_f|\KK}$, so $D_{A_f|\KK}[-\partial_t t]\cong D_{A_f|\KK}[s]$ for an indeterminate $s$. We consider $H^1_{(f-t)}(A_f[t])$ as a $D_{A_f|\KK}[s]$-module via this isomorphism. Namely,
\[ (\delta_m s^m + \cdots + \delta_0) \cdot \left[ \frac{a}{(f-t)^n}\right] = (\delta_m( -\partial_t t)^m + \cdots + \delta_0 ) \cdot \left[ \frac{a}{(f-t)^n}\right], \]
where the action on the right is the natural action on the localization.

\begin{lemma}\label{lem-inj-alpha}
		The elements 
	\[\left\{ (-\partial_t t)^n \cdot \left[ \frac{1}{f-t}\right] \ \big| \ n\in \NN\right\} \]
	are $A_f$-linearly independent in $H^1_{(f-t)}(A[t]) \subseteq H^1_{(f-t)}(A_f[t])$.
	\end{lemma}
	
\begin{proof}	We show by induction on $n$ that the coefficient of $(-\partial_t t)^n \cdot \left[ \frac{1}{f-t}\right]$ corresponding to the element $\left[ \frac{1}{(f-t)^{n+1}}\right]$ in the $A_f$-basis \eqref{eq-basis-H1Aft} is nonzero. This is trivial if  $n=0$, and the inductive step follows from the formula
	\[ -\partial_t t \ \cdot \ \left[ \frac{a}{(f-t)^n} \right] = \left[ \frac{(n-1) a}{(f-t)^n}\right] + \left[\frac{-n f a}{(f-t)^{n+1}}\right]. \qedhere\]
\end{proof}

\begin{proposition}\label{prop-compare-module-LC}
The map
\[ \alpha: A_f[s] \boldsymbol{f^s} \to H^1_{(f-t)}(A_f[t]) \quad \text{given by} \quad \alpha(a(s)\boldsymbol{f^s}) = a(-\partial_t t) \cdot \left[ \frac{1}{f-t}\right] \]
is an injective homomorphism of $D_{A_f|\KK}[s]$-modules.
\end{proposition}

\begin{proof}
Injectivity of $\alpha$ follows from Lemma~\ref{lem-inj-alpha}. We just need to check that this map is linear with respect to the action of $D_{A_f|\KK}[s]$. We have that  $\alpha$ is $A_f[s]$-linear; we just need to check that $\alpha$ commutes with the derivatives $\partial_i$. We compute that 
\[ \alpha ( \partial_i \boldsymbol{f^s}) = \alpha\left(\frac{s \partial_i(f)}{f} \boldsymbol{f^s}\right) = -\partial_t t \frac{\partial_i(f)}{f} \left[ \frac{1}{f-t} \right] = - {\partial_i(f)} \partial_t\left[ \frac{1}{f-t} \right] = \partial_i \left[ \frac{1}{f-t} \right],\]
where in the penultimate equality we used that 
\[t \left[ \frac{1}{f-t} \right] = (f-(f-t)) \left[ \frac{1}{f-t} \right]=f\left[ \frac{1}{f-t} \right].\qedhere\]
\end{proof}

We note that $\alpha$ is not surjective in general.

As $A_f[s] \boldsymbol{f^s}$ is generated by $\boldsymbol{f^s}$ as a $D_{A_f|\KK}[s]$-module, Proposition~\ref{prop-compare-module-LC} yields the following result.

\begin{proposition}\label{prop:iso-LC-construction} The $D_{A_f|\KK}[s]$-module
	$A_f[s] \boldsymbol{f^s}$ is isomorphic to the submodule
	$D_{A_f|\KK}[s] \cdot \left[ \frac{1}{f-t}\right] \subseteq H^1_{(f-t)}(A_f[t])$, where $s$ acts on the latter by $-\partial_t t$.
\end{proposition}

\subsubsection{Constructions of the module $D_{A|\KK}[s] \boldsymbol{f^s}$}

We now give three constructions of the submodule $D_{A|\KK}[s] \boldsymbol{f^s}$ of the module $A_f[s] \boldsymbol{f^s}$. The first is exactly as suggested by the notation.

\begin{definition}
We define $D_{A|\KK}[s] \boldsymbol{f^s}$ as the $D_{A|\KK}[s]$-submodule of $A_f[s] \boldsymbol{f^s}$ generated by the element $\boldsymbol{f^s}$.
\end{definition}

\begin{proposition}\label{prop:isochar0} There is an isomorphism
\[D_{A|\KK}[s] \boldsymbol{f^s} \cong \frac{D_{A|\KK}[s]}  {\{ \delta(s) \in D_{A|\KK}[s] \ | \ \delta(n) f^n = 0 \ \text{for all} \ n\in \NN\}}.\]
\end{proposition}
\begin{proof}
	We just need to show that the annihilator of $\boldsymbol{f^s}$ in $A_f[s] \boldsymbol{f^s}$ is \[\{ \delta(s) \in D_{A|\KK}[s] \ | \ \delta(n) f^n = 0 \ \text{for all} \ n\in \NN\}.\]
	We can write $\delta(s) \boldsymbol{f^s}$ as $p(s) \boldsymbol{f^s}$ for some $p(s)\in A_f[s]$. Observe that 
	\begin{align*}
 p(s) \boldsymbol{f^s} = 0 &\Leftrightarrow p(s)=0 \\
 &\Leftrightarrow p(n)=0 \ \text{for all} \ n\in \NN\\
 &\Leftrightarrow p(n) f^n =0 \ \text{for all} \ n\in \NN\\
 &\Leftrightarrow \theta_n(p(s) \boldsymbol{f^s})=0 \ \text{for all} \ n\in \NN .
	\end{align*}
	Then, $\delta(s) \boldsymbol{f^s} = 0$ if and only if $0=\theta_n(\delta(s) \boldsymbol{f^s})= \delta(n) f^n$ for all $n\in \NN$, as required.
\end{proof}

Note that this is using characteristic zero in a crucial way: we need that a polynomial that has infinitely many zeroes (or that is identically zero on $\NN$) is the zero polynomial.

\begin{remark}\label{rem:specialization} An argument analogous to the above shows that, for $\delta(s)\in D_{A|\KK}[s]$, the following are equivalent:
	\begin{enumerate}
		\item $\delta(s) \boldsymbol{f^s} =0$ in $A_f[s]\boldsymbol{f^s}$;
		\item $\delta(n) f^n =0$ in $A$ for all $n\in \NN$;
		\item $\delta(n) f^n =0$ in $A_f$ for all $n\in \ZZ$;
		\item $\delta(n) f^n =0$ in $A_f$ for infinitely many $n\in \ZZ$.
	\end{enumerate}
Likewise, by shifting the evaluations, ones sees this is equivalent to:
\begin{enumerate}
	\setcounter{enumi}{4}
	\item $\delta(s+t) f^t\boldsymbol{f^s} =0$ in $A_f[s]\boldsymbol{f^s}$.
	\end{enumerate}

\end{remark}

Finally, we observe that $D_{A|\KK}[s] \boldsymbol{f^s}$ can be constructed via local cohomology as in Subsubsection~\ref{subsub:localcohomconstruction}. By restricting the isomorphism of Proposition~\ref{prop:iso-LC-construction}, we obtain the following result.

\begin{proposition} The $D_{A|\KK}[s]$-module
	$D_{A|\KK}[s] \boldsymbol{f^s}$ is isomorphic to the submodule $D_{A|\KK}[s] \cdot \left[ \frac{1}{f-t}\right] \subseteq H^1_{(f-t)}(A[t])$, where $s$ acts on the latter by $-\partial_t t$.
	\end{proposition}

\begin{proposition}\label{prop:othercharsBS}
	The following are equal:
	\begin{enumerate}
		\item The Bernstein-Sato polynomial of $f$;
		\item The minimal polynomial of the action of $s$ on $\displaystyle \frac{D_{A|\KK}[s]\boldsymbol{f^s}}{D_{A|\KK}[s] f \boldsymbol{f^s}}$;
		\item The minimal polynomial of the action of $-\partial_t t$ on $[\frac{1}{f-t}]$ in $$\displaystyle \frac{D_{A|\KK}[-\partial_t t] \cdot [\frac{1}{f-t}]}{D_{A|\KK}[-\partial_t t] \cdot f [\frac{1}{f-t}]};$$
		\item The monic element of smallest degree in $\KK[s] \cap (\mathrm{Ann}_{D[s]}(\boldsymbol{f^s}) + D_{A|\KK}[s] f)$.
	\end{enumerate}
\end{proposition}
\begin{proof}
	The equality between the first two follows from the definition. The equality between the second and the third follows from the previous proposition. For the equality between the second and the fourth, we observe that
	\[ \frac{D_{A|\KK}[s]\boldsymbol{f^s}}{D_{A|\KK}[s] f \boldsymbol{f^s}} \cong \mathrm{coker}\left( \frac{D_{A|\KK}[s]}{\mathrm{Ann}_{D[s]}(\boldsymbol{f^s})} \xrightarrow{\cdot f} \frac{D_{A|\KK}[s]}{\mathrm{Ann}_{D[s]}(\boldsymbol{f^s})} \right) \cong \frac{D_{A|\KK}[s]}{\mathrm{Ann}_{D[s]}(\boldsymbol{f^s}) + D_{A|\KK}[s] f}. \]
\end{proof}

\begin{remark}\label{rem:D-mod-alpha}
	For any rational number $\alpha$, we can consider the $D_{R|\KK}$-modules $D_{R|\KK} f^\alpha$ and $A_f f^\alpha$ by specializing $s\mapsto \alpha$ in the $D_{R|\KK}[s]$-modules $D_{R|\KK}[s]\boldsymbol{f^s}$ and $A_f[s] \boldsymbol{f^s}$. These modules are important in $D$-module theory, but we do not focus on them in depth here.
\end{remark}

We end this subsection with equivalent characterizations on $A_f[s] \boldsymbol{f^s} \otimes_{\KK[s]} \KK(s)$ for $f$  to have a nonzero functional equation. This lemma plays a role in 
Corollary \ref{CorExistenceBSPoly} and 
Theorem \ref{ThmExistenceHomological}.

\begin{lemma}[{\cite[Proposition 2.18]{Vfilt}}] \label{PropEquivExistence}
   Fix an element $f\in A$.
   Then, the following are equivalent\textup:
	\begin{enumerate}
		\item There exists a Bernstein--Sato polynomial for $f$;
				\item $A_f[s] \boldsymbol{f^s} \otimes_{\KK[s]} \KK(s)$ is generated by $\fs$ as a $D_{A(s)|\KK(s)}$-module;
		\item $A_f[s] \boldsymbol{f^s} \otimes_{\KK[s]} \KK(s)$ is a finitely-generated $D_{A(s)|\KK(s)}$-module.
	\end{enumerate}
\end{lemma}
\begin{proof}
We first show that (i) implies  (ii).
For every $m\in\ZZ$, we have an isomorphism of $D_{A(s)|\KK(s)}$-modules 
$$\psi_m:  A_f \boldsymbol{f^s} \otimes_{\KK[s]} \KK(s)    \to A_f \boldsymbol{f^s} \otimes_{\KK[s]} \KK(s)$$ defined by 
$$
 \frac{r(s) h}{f^\alpha } \boldsymbol{f^s} \mapsto  
 \frac{r(s-m) h}{f^{\alpha+m}} \boldsymbol{f^s}.
$$
Applying these isomoprhism to the functional equation, we obtain that $\frac{1}{f^m}  \boldsymbol{f^s} \in D_{A(s)|\KK(s)}  \boldsymbol{f^s}  $.

Since  (ii) implies  (iii) follows from definition, we focus in proving that   (iii) implies (i).
First we note that (iii) implies that $\frac{1}{f^m} \boldsymbol{f^s}. $ generates $A_f \boldsymbol{f^s} \otimes_{\KK[s]} \KK(s)$  for some $m\in\NN$. Then,  $\frac{1}{f^{m+1}} \boldsymbol{f^s}\in D_{A(s)|\KK(s)} \frac{1}{f^m} \boldsymbol{f^s}. $
Then, there exists $\delta(s)\in D_{A(s)|\KK(s)}$ such that 
$$\delta(s)  \frac{1}{f^m} \boldsymbol{f^s} =\frac{1}{f^{m+1}} \boldsymbol{f^s} .$$
After clearing denominators and shifting by ${-m-1}$, we obtain a functional equation.
\end{proof}

\subsection{Existence of Bernstein-Sato polynomials for polynomial rings via filtrations}\label{subsec:ExistencePoly}

In this subsection  $A=\KK[x_1,\ldots, x_d]$ is a polynomial ring over a field, $\KK$, of characteristic zero. 
This was proved in this case by Bernstein \cite{Ber71, Ber72}. 
We show the existence of the Bernstein-Sato polynomial using the strategy of Coutinho's book \cite{Coutinho}.

We define the \textit{Bernstein filtration }of $A$, $\cB^\bullet_{A|\KK}$ as 
$$
\cB^i_{A|\KK}=\bigoplus_{a_1+\cdots+a_d + b_1 + \cdots+b_d \leq i} \KK \cdot x_1^{a_1} \cdots x_d^{a_d} \partial_1^{b_1} \cdots \partial_d^{b_d}.
$$
We note that 
\begin{enumerate}
\item $\dim_\KK  \cB^i_{A|\KK}=\binom{n+i}{i}<\infty$,
\item $D_{A|\KK}=\bigcup_{i\in\NN} \cB^i_{A|\KK}$,
\item $ \cB^i_{A|\KK} \cB^j_{A|\KK} =\cB^{i+j}_{A|\KK}$, and
\item $[ \cB^i_{A|\KK}, \cB^j_{A|\KK}]\subseteq  \cB^{i+j-2}_{A|\KK}$.
\end{enumerate}
We observe that the associated graded ring of the filtration, $\gr(\cB^\bullet_{A|\KK},D_{A|\KK})$, is isomorphic to $\KK[x_1,\ldots, x_d,y_1,\ldots,y_d]$.

Given a left , $D_{A|\KK}$-module, $M$, we say that a filtration $\Gamma^\bullet$ of $\KK$-vector spaces is $\cB^\bullet_{A|\KK}$-compatible if
\begin{enumerate}
\item $\dim_\KK  \Gamma^i<\infty$,
\item $M=\bigcup_{i\in\NN} \Gamma^i$, and
\item $ \cB^i_{A|\KK} \Gamma^j\subseteq \Gamma^{i+j}$.
\end{enumerate}
In this manuscript,  by a $D_{A|\KK}$-module, unless specified, we mean a left  $D_{A|\KK}$-module.

We observe that $\gr(\Gamma^\bullet,M)$ is a graded $\gr(\cB^\bullet_{A|\KK},D_{A|\KK})$-module. 
Moreover, $M$ is finitely generated as a $D_{A|\KK}$-module if and only if there exists a filtration $\Gamma^\bullet$ such that 
$\gr(\Gamma^\bullet,M)$ is finitely generated as a  $\gr(\cB^\bullet_{A|\KK},D_{A|\KK})$-module. 
In this case, we say that $\Gamma$ is a \emph{good filtration for} $M$.

\begin{proposition}\label{PropAnnGooFil}
Let $M$ be a finitely generated  $D_{A|\KK}$-module. Let $G$ 
 denote the associated graded ring with respect to the Bernstein filtration.
Let $\Gamma_1^\bullet$ and $\Gamma_2^\bullet$ be two good filtrations for $M$.
Then,
$$
\sqrt{\Ann_{G} \gr(\Gamma_1^\bullet,M) }
=\sqrt{\Ann_{G} \gr(\Gamma_2^\bullet,M)}.
$$
\end{proposition}

Thanks to the previous result we are able to define the \emph{dimension} of 
 a finitely generated  $D_{A|\KK}$-module  $M$  as
 $$
 \dim_{D} (M)=  \dim_{G}  \left(   \frac{G}{  \Ann_{G} \gr(\Gamma^\bullet,M) }\right).
 $$

\begin{theorem}[Bernstein's  Inequality]
Let $M$ be a finitely generated  $D_{A|\KK}$-module.
Then,
$$
d\leq  \dim_{D} (M)\leq 2d.
$$
\end{theorem}

\begin{definition}
We say that a finitely generated  $D_{A|\KK}$-module, $M$, is \emph{holonomic} if either
$\dim_{D} (M)=d$ or $M=0$.
\end{definition}

\begin{theorem}\label{ThmHolFinLen}
Every holonomic 
$D_{A|\KK}$-module has finite length as $D_{A|\KK}$-module.
\end{theorem}
\begin{proof}
Let $M_0\subsetneqq M_1\subsetneqq \cdots \subsetneqq M_t\subseteqq M$ be a proper chain of $D_{A|\KK}$-submodules.
Let $\Gamma^\bullet$ be a good filtration. We note that  $	\Gamma^i_j=\Gamma^i\cap M_j$ is a good filtration on $M_j$. In addition,  $\overline{\Gamma}^i_j=\phi_j(\Gamma^i_j)$, where $\pi:M_j\to M_j/M_{j-1}$ is the quotient map,  is a good filtration for $M_j/M_{j-1}$.
We have the following identity of Hilbert-Samuel multiplicities of graded $\gr(\cB^\bullet_{A|\KK},D_{A|\KK})$-modules:
$$
\e( \gr(\Gamma^\bullet,M) )
=\sum^t_{j=1}
\e(\gr(\overline{\Gamma}^\bullet_j,M_j/M_{j-1})).
$$
Since the multiplicities are positive integers, we have that $t\leq \e( \gr(\Gamma^\bullet,M) )$, and so, the length of $M$ as a 
$D_{R|\KK}$-module is at most $\e( \gr(\Gamma^\bullet,M) )$.
\end{proof}

\begin{theorem}\label{ThmAfsHol}
Given any nonzero  polynomial $f\in A$,
$A_f[s] \boldsymbol{f^s}\otimes_{\KK[s]} \KK(s) $ is a holonomic $D_{A(s)|\KK(s)}$-module.
\end{theorem}
\begin{proof}
Let $t=\deg(f)$.
We set a filtration 
$$
\Gamma_i=\frac{1}{f^i}\{g\in A(s) \; | \;  \deg(g)\leq (t+1)i  \}  \boldsymbol{f^s}.
$$
We note that $\Gamma_i$ is a good filtration such that the associated graded of $A_f[s] \boldsymbol{f^s}\otimes_{\KK[s]} \KK(s) $ has dimension $d$. 
\end{proof}

\begin{corollary}[{\cite{Ber72}}]\label{CorExistenceBSPoly}
Given any nonzero  polynomial $f\in A$, the Bernstein-Sato polynomial of $f$ exists.
\end{corollary}
\begin{proof}
This follows from  Proposition \ref{PropEquivExistence} and Theorems \ref{ThmHolFinLen} \& \ref{ThmAfsHol}.
\end{proof}

\subsection{Existence of Bernstein-Sato polynomials for differentiably admissible algebras via homological methods}\label{SubSecDifAd}

In this subsection we prove the existence of Bernstein-Sato polynomials of differentiably admissible $\KK$-algebras (see Subsection~\ref{SubSecPreDiffAdm}). 
We assume that $\KK$ is a field of characteristic zero.

\begin{definition} 
Let  $A$  be  a differentiably admissible $\KK$-algebra.
Let $M\neq 0$ be a finitely generated  $D_{A|\KK}$-module. We define
$$
\Grade_{D_{A|\KK}} (M)
=\hbox{inf}\{ j \ | \ \Ext^j_{D_{A|\KK}} (M,D_{A|\KK})\neq 0\}.
$$ 
We note that 
$\Grade_{D_{A|\KK}} (M)\leq \gd (D_{R|\KK})=d.$
\end{definition}

\begin{remark}
Given a finitely generated  $D_{A|\KK}$-module, we can define the
filtrations compatible with the order filtration $D^\bullet_{A|\KK}$, good filtrations, and dimension as in Subsection~\ref{subsec:ExistencePoly}.
\end{remark}

\begin{proposition}[{\cite[Ch 2., Theorem 7.1]{Bjork79}}]\label{ThmHomologicalBernsteinIneq}
Let  $A$  be a  differentiably admissible $\KK$-algebra.
Let $M\neq 0$ be a finitely generated  $D_{A|\KK}$-module. 
Then,
$$\dim_{D_{A|\KK}}(M)+\Grade_{D} (M)=2d.$$ In particular, 
$$
\dim_{D_{A|\KK}}(M)\geq d.
$$ 
\end{proposition}

We stress that the conclusion of the previous result is satisfied for rings of differentiable type \cite{MNM,NBDT}.

\begin{definition}\label{DefBernsteinClass}
Let  $A$  be  a  differentiably admissible $\KK$-algebra.
Let $M$ be a finitely generated  left (right) $D_{A|\KK}$-module. 
We say that $M$ is in the  \emph{left (right)  Bernstein class} if either $M=0$ or
if $\dim_{D}(M)=d$.
\end{definition}

Let $M$ be a finitely generated $D_{A|\KK}$-module.
If $M$ is in the Bernstein class of $D_{A|\KK}$, then
$\Ext^i_{D_{A|\KK}}(M,D_{A|\KK})\neq 0$ if and only if $i= d$ \cite{Bjork79}.
Then, the functor 
that sends
$M$ to $\Ext^d_{D_{A|\KK}}(M,D_{A|\KK})$
is an exact contravariant functor that interchanges the left Bernstein
class and the right Bernstein class.
 Furthermore, $M\cong \Ext^d_{D_{A|\KK}}( \Ext^d_{D_{A|\KK}}(M,D_{A|\KK}),D_{A|\KK})$ for modules in the Bernstein class.
Since $D_{R|\KK}$ is left and right Noetherian, the modules in the Bernstein class 
are both Noetherian and Artinian.  We conclude that the modules in the Bernstein class have finite length as $D_{A|\KK}$-modules
\cite[Proposition 1.2.5]{MNM})

This class of Bernstein modules is an analogue of the class of holonomic modules.
In particular, it 
is closed under submodules, quotients,  extensions, and localizations \cite[Proposition 1.2.7]{MNM}).

\begin{theorem}\label{ThmExistenceHomological}
Let  $A$  be  a differentiably admissible $\KK$-algebra of dimension $d$.
Given any nonzero  element $f\in A$,  the Bernstein-Sato polynomial of $f$ exists.
\end{theorem}
\begin{proof}[Sketch of proof]
In this sketch we follow the ideas of Mebkhout and Narv\'{a}ez-Macarro \cite{MNM} (see also \cite{NBDT}). In particular, we refer the interested reader to their work on the base change $\KK$ to $\KK(s)$ regarding differentiably admissible algebras  \cite[Section 2]{MNM}.
Let $A(s)=A\otimes_\KK \KK(s)$.
We observe that  $A(s)$ is not always a differentiably admissible $\KK(s)$-algebra. Specifically, the residue fields of $A(s)$  might not be always algebraic.
However, $D_{A(s)|\KK(s)}$ satisfies the conclusions of Theorem \ref{ThmDifType}. In particular, the conclusions of Theorem~\ref{ThmHomologicalBernsteinIneq} hold, and its Bernstein class is well defined. 
We have that the dimension and global dimension of $D_{A(s)|\KK(s)}$ and $D_{A|\KK}$
are equal.
One can show that $A_f[s] \boldsymbol{f^s}\otimes_{\KK[s]} \KK(s) $   has a $D_{A(s)|\KK(s)}$-submodule $N$ is in the Bernstein class of  $D_{A(s)|\KK(s)}$ such that $N_f=A_f[s] \boldsymbol{f^s}\otimes_{\KK[s]} \KK(s) $
\cite[Proposition 1.2.7 and Proof of Theorem 3.1.1]{MNM}.
Then,  there exists $\ell\in\NN$ such that $f^\ell \boldsymbol{f^s}  \in N$. Since $N$ has finite length as  $D_{A(s)|\KK(s)}$-module
 the chain
 $$
 D_{A(s)|\KK(s)} f^\ell \boldsymbol{f^s} \supseteq  D_{A(s)|\KK(s)} f^{\ell+1} \boldsymbol{f^s} \supseteq  D_{A(s)|\KK(s)} f^{\ell+2} \boldsymbol{f^s} \supseteq \ldots
 $$
 stabilizes.
 Then, there exists  $m\in\NN$ and a differential operator $\delta(s)\in  D_{A(s)|\KK(s)}$ such that
 $$
 \delta(s) f^{\ell+m+1} \boldsymbol{f^s} = f^{\ell+m} \boldsymbol{f^s} .
 $$
 After clearing denominators and a shifting, there exists $\widetilde{\delta}(s)\in  D_{A|\KK}[s]$ such that
 $$
\widetilde{ \delta}(s) f \boldsymbol{f^s} =  \boldsymbol{f^s} .
 $$
\end{proof}

\subsection{First properties of the Bernstein-Sato polynomial} \label{First properties}

A first observation about the Bernstein-Sato polynomial is that $s+1$ is always a factor.

\begin{lemma} \label{s+1}
	For $f\in A$, we have $(s+1) \,| \,b_f(s)$ if and only if $f$ is not a unit.
\end{lemma}
\begin{proof}
	If $f$ is a unit, then we can take $f^{-1} f^{s+1} = 1 f^s$ as a functional equation, so $b(s)=1$ is the Bernstein-Sato polynomial of $f$.  
	
	For the converse, by definition, we have $\delta(s) f \boldsymbol{f^s} = b_f(s) \boldsymbol{f^s}$ in $A_f[s] \boldsymbol{f^s}$. By Remark~\ref{rem:specialization}, $\delta(n) f^{n+1} = b_f(n) f^n$ in $A_f$ for all $n\in \ZZ$. In particular, for $n=-1$, we get $\delta(-1) 1 = b_f(-1) f^{-1}$. As $\delta(-1)\in D_{A|\KK}$, we have $\delta(-1) 1\in A$. Thus, $b_f(-1)=0$, so $s+1$ divides $b_f(s)$.
	\end{proof}
	
Quite nicely, the factor $(s+1)$ characterizes the regularity of $f$.

\begin{proposition}[\cite{BM_Rabida}]
For $f\in A$,  we have $A/fA$ is smooth if and only if $b_f(s)=s+1$.
\end{proposition}

\begin{definition}
	The \emph{reduced Bernstein-Sato polynomial} of a nonunit $f\in A$ is \[\tilde{b}_f(s) =b_f(s)/(s+1).\]
\end{definition}

The analogue of Proposition~\ref{prop:othercharsBS} for the reduced Bernstein-Sato polynomial is as follows.

\begin{proposition}\label{prop:othercharsrelativeBS}
	The following are equal:
	\begin{enumerate}
		\item $\tilde{b}_f(s)$,
		\item The minimal polynomial of the action of $s$ on $\displaystyle (s+1)\frac{D_{A|\KK}[s]\boldsymbol{f^s}}{D_{A|\KK}[s] f \boldsymbol{f^s}}$,
		\item The monic element of smallest degree in $$\KK[s] \cap (\mathrm{Ann}_{D[s]}\big(\boldsymbol{f^s}) + D_{A|\KK}[s] ( f, \partial_1(f), \dots , \partial_n(f) \big)).$$
	\end{enumerate}
\end{proposition}

\begin{proof}
Once again, the first two are equivalent by definition. 

Given a functional equation $\delta(s) f \boldsymbol{f^s} = (s+1) \tilde{b}(s) \boldsymbol{f^s}$, we have that $\delta(-1)\in D_{A|\KK}$ with $\delta(-1) \cdot 1=0$. We can write $\delta(s)=(s+1) \delta'(s) + \delta(-1)$ for some $\delta'(s)\in D_{A|\KK}[s]$, so $\delta(s)=(s+1)\delta'(s) + \sum_{i=1}^d \delta_i \partial_i$ for some $\delta_i\in D_{A|\KK}$. Then, using that $\partial_i (f \boldsymbol{f^s}) = (s+1) \partial_i(f) \boldsymbol{f^s}$, we have
\[(s+1) \tilde{b}(s) \boldsymbol{f^s} =(s+1)  \delta'(s)f \boldsymbol{f^s} + \sum_{i=1}^d \delta_i \partial_i f \boldsymbol{f^s} = (s+1) (\delta'(s) f +\sum_{i=1}^d \delta_i \partial_i(f) )\boldsymbol{f^s}. \]
Thus, such a functional equation implies that $\tilde{b}(s) \boldsymbol{f^s}\in D_{A|\KK}[s] (f, \partial_1(f),\dots,\partial_d(f))$. Conversely, if $\tilde{b}(s) \boldsymbol{f^s}\in D_{A|\KK}[s] (f, \partial_1(f),\dots,\partial_d(f))$, again using that $\partial_i (f \boldsymbol{f^s}) = (s+1) \partial_i(f) \boldsymbol{f^s}$, we can write $(s+1)\tilde{b}(s) \boldsymbol{f^s} \in D_{A|\KK}[s] f \boldsymbol{f^s}$. This implies the equivalence of the first and the last characterizations.
\end{proof}

We may also be interested in the characteristic polynomial of the action of $s$. Traditionally, with the convention of a sign change,  the roots of the characteristic polynomial are known as the $b$-exponents of $f$.

\begin{definition}
The $b$-exponents of $f\in A$ are the roots of the characteristic polynomial of the action of $-s$ on $\displaystyle (s+1)\frac{D_{A|\KK}[s]\boldsymbol{f^s}}{D_{A|\KK}[s] f \boldsymbol{f^s}}$.
\end{definition}

So far we have considered Bernstein-Sato polynomials over different regular rings $A$ but, a priori, it is not clear how they are related.
Our next goal is to address this issue.
We start considering $A=\KK[x_1,\dots,x_d]$, a polynomial ring over a field $\KK$ of characteristic zero and denote by $b_f^{\KK[x]}(s)$ the Bernstein-Sato polynomial of $f\in A$. Given any maximal ideal $\mathfrak{m}\subseteq A$ we also consider the Bernstein-Sato polynomial over the localization $A_\mathfrak{m}$ that we denote $b_f^{\KK[x]_\mathfrak{m}}(s)$.

\begin{proposition} \label{local_global_BS}
We have:
\[ b_f^{\KK[x]}(s) = {\rm lcm} \{b_f^{\KK[x]_\mathfrak{m}}(s) \; | \; \mathfrak{m} \subseteq A  \hskip 2mm {\rm maximal} \hskip 2mm {\rm ideal} \}. \]
\end{proposition}

\begin{proof}
Let $b(s) \in \KK[s]$ be a polynomial. 
 The module $\displaystyle b(s)\frac{D_{A|\KK}[s]\boldsymbol{f^s}}{D_{A|\KK}[s] f \boldsymbol{f^s}}$ vanishes if and only if it vanishes locally. The localization at a maximal ideal $\mathfrak{m}\subseteq A$ is 
$$b(s) \frac{D_{A_\mathfrak{m}|\KK}[s]\boldsymbol{f^s}}{D_{A_\mathfrak{m}|\KK}[s] f \boldsymbol{f^s}}$$ and the result follows. 
\end{proof}

For a polynomial $f\in A$ we may also consider the Bernstein-Sato polynomial $b_f^{\KK\llbracket x\rrbracket}(s)$ in the formal power series ring $\KK\llbracket x_1,\dots,x_d\rrbracket$.

\begin{proposition} \label{local_C}
Let $\mathfrak{m}= (x_1,\dots, x_d) \subseteq A$ be the homogeneous maximal ideal. We have:
\[ b_f^{\KK[x]_\mathfrak{m}}(s)  =   b_f^{\KK\llbracket x\rrbracket}(s).\]
\end{proposition}

\begin{proof}
$B=\KK\llbracket x_1,\dots,x_d\rrbracket$ is faithfully flat over $ A_\mathfrak{m}=\KK [ x_1,\dots,x_d ]_{\mathfrak{m}}$. Since 
$$ B \otimes_{A_\mathfrak{m}} b(s) \frac{D_{A_\mathfrak{m}|\KK}[s]\boldsymbol{f^s}}{D_{A_\mathfrak{m}|\KK}[s] f \boldsymbol{f^s}} = b(s) \frac{D_{B|\KK}[s]\boldsymbol{f^s}}{D_{B|\KK}[s] f \boldsymbol{f^s}}$$ the result follows.
\end{proof}

When $\KK=\CC$ we may also consider the ring $\CC\{x_1-p_1,\dots,x_d-p_d\}$ of convergent power series   in a neighborhood of a point $p =(p_1,\dots, p_d)\in\CC^d$.

\begin{corollary}
We have 
\begin{enumerate}
\item   $b_f^{\CC[x]}(s) = {\rm lcm} \{b_f^{\CC\{ x-p\}}(s) \; | \; p \in \CC^d \}$.
\item   $b_f^{\CC\{ x-p\}}(s) = b_f^{\CC\llbracket x-p\rrbracket}(s)$.
\end{enumerate}

\end{corollary}

\begin{proof}
Working over $\CC$ we have that all the maximal ideals correspond to points so $(i)$ follows from \Cref{local_global_BS}.
For part $(ii)$  we use the same faithful flatness trick we used in \Cref{local_C} for  $\CC\{ x_1-p_1,\dots,x_d-p_d\}$. 
\end{proof}

Let  $f\in \KK[x_1,\dots,x_d]$ be a polynomial and  $\mathbb{L}$ a field containing $\KK$. Let $b_f^{\KK[x]}(s)$ and $b_f^{\mathbb{L}[x]}(s)$ be the Bernstein-Sato polynomial of $f$ in $\KK[x_1,\dots,x_d]$ and $\mathbb{L}[x_1,\dots,x_d]$. respectively

\begin{proposition}
We have $b_f^{\KK[x]}(s) = b_f^{\mathbb{L}[x]}(s)$.
\end{proposition}

\begin{proof}
Notice that $b_f^{\mathbb{L}[x]}(s) \hskip 2mm | \hskip 2mm b_f^{\KK[x]}(s)$ so we have to prove the other divisibility condition.
Let $\{e_i\}_{i\in I}$ be a basis of $\mathbb{L}$  as a $\KK$-vector space. We have 
\[ \frac{D_{A|\mathbb{L}}[s]\boldsymbol{f^s}}{D_{A|\mathbb{L}}[s] f \boldsymbol{f^s}} = \mathbb{L} \otimes_{\KK}  \frac{D_{A|\KK}[s]\boldsymbol{f^s}}{D_{A|\KK}[s] f \boldsymbol{f^s}} =  \bigoplus_{i\in I} \left( \frac{D_{A|\mathbb{K}}[s]\boldsymbol{f^s}}{D_{A|\mathbb{K}}[s] f \boldsymbol{f^s}} \right) e_i.\]
Let $b(s) \in \mathbb{L}[s]$ be such that $b(s) \frac{D_{A|\mathbb{L}}[s]\boldsymbol{f^s}}{D_{A|\mathbb{L}}[s] f \boldsymbol{f^s}}=0$. 
Then $b(s)=\sum b_i(s)$ with only finitely many nonzero $b_i(s)\in \KK[s]$ such that $b_i(s) \frac{D_{A|\mathbb{K}}[s]\boldsymbol{f^s}}{D_{A|\mathbb{K}}[s] f \boldsymbol{f^s}}=0$. Since $b_f^{\KK[x]}(s) \hskip 2mm | \hskip 2mm b_i(s)$ for all $i$ it follows that $b_f^{\KK[x]}(s) \hskip 2mm | \hskip 2mm b_f^{\mathbb{L}[x]}(s)$.
\end{proof}

\begin{remark}
Let $f\in \KK[x_1,\dots,x_d]$ be a polynomial with an isolated singularity at the origin, where $\KK$ is a subfield of $\mathbb{C}$. Then we have 
$b_f^{\KK[x]}(s)= b_f^{\KK\llbracket x \rrbracket}(s) = b_f^{ \CC\{x\}}(s)$.
\end{remark}

Combining all the results above with the following fundamental result of Kashiwara \cite{KashiwaraRationality},  we conclude that the Bernstein-Sato polynomial of $f\in \KK[x_1,\dots , x_d]$ is a polynomial $b_f(s) \in \QQ[s]$.

\begin{theorem}[{\cite{KashiwaraRationality, BernsteinRationalMalgrange}}] \label{rationality_BS}
The Bernstein-Sato polynomial of an element $f \in \CC \{ x_1,\dots,x_d \}$, or $f\in \KK[x_1,\dots,x_d]$ for $\KK\subseteq \CC$, factors completely over $\QQ$, and all of its roots are negative rational numbers.
\end{theorem}

In \Cref{Complex_zeta} we will provide  a refinement of this result given by Lichtin \cite{Lic89}.

\section{Some families of examples}\label{SecExamples}

Computing explicit examples of Bernstein-Sato polynomials is a very challenging task. 
There are general algorithms based on the theory of Gr\"obner bases over rings of differential operators but they have a very high complexity so only few examples can be effectively computed \cite{Oaku97,Oaku1997, LevandovskyyMartinMorales2012}. In this section we review some of the scarce examples that we may find in the literature. 
The first systematic method of producing examples can be found in the work of Yano \cite{Yano1978} where he considered, among others, the case of  isolated quasi-homogeneous singularities (see also \cite{BGM_pre}). The case of isolated semi-quasi-homogeneous singularities was studied later on by Saito \cite{saito89} and Brian\c{c}on, Granger, Maisonobe, and Miniconi \cite{BrianconAlgorithm}. 

\vskip 2mm
  
A case that has been extensively studied is that of plane curves, see \cite{Yano1982, Kato1981, Kato1982, Cassou1986, Cassou1987, Cassou1988, HerSta, Briancon2007}.  In particular, a conjecture of Yano regarding the $b$-exponents of a generic irreducible plane curve among those in the same equisingularity class has been recently proved by Blanco \cite{Blanco_yano} (see also \cite{Cassou1988, ACNLM17, Bla19}).  Finally we want to mention 
that the case of hyperplane arrangements has been studied by Walther \cite{WaltherBS} and Saito \cite{Saito_BS_arrangements}.

\vskip 2mm

We start with some known examples where a Bernstein-Sato functional equation  $\delta(s) f^{s+1} = b(s) f^s$ can be given by hand:

\begin{itemize}
\item[i)] Let \(f=x_1^2 + \cdots + x_n^2\) be a sum of squares. Then  $$\frac{1}{4} (\partial_1^2 + \cdots + \partial_n^2) f^{s+1} = (s+1)(s+\frac{n}{2}) f^s.$$

\item[ii)] Let \(f= \det (x_{ij})\) be the determinant of an $n \times n$ generic matrix and set $\partial_{ij}:=\frac{d}{dx_{ij}}$. The classic Cayley identity states  $$\det (\partial_{ij})  f^{s+1} = (s+1)(s+2) \cdots (s+n) f^s.$$
There are similar identities for determinants of symmetric and antisymmetric matrices   \cite{CSS_Cayley}.

\vskip 2mm

\item[iii)]   Let $f=x_1^{\alpha_1} \cdots x_n^{\alpha_n}$ be a monomial. Then
$$\frac{1}{\alpha_1^{\alpha_1} \cdots \alpha_n^{\alpha_n} } (\partial_1^{\alpha_1}  \cdots  \partial_n^{\alpha_n}) f^{s+1} = \prod_{i=1}^{n} \prod_{k=1}^{\alpha_i} (s+\frac{k}{\alpha_i}) f^s.$$

\end{itemize}

\vskip 2mm 

We warn the reader that it requires some extra work to prove that the above polynomials are minimal so they are indeed  Bernstein-Sato polynomials of the corresponding $f$.

\vskip 2mm

Let $A=\CC\{x_1,\dots , x_d\}$ and assume that $f$ has an isolated singularity at the origin. In this case, Yano \cite{Yano1978}  uses the fact that the support of the holonomic $D_{A|\CC}$-module $\widetilde{\mathcal{M}}:=(s+1) \frac{D_{A|\CC}[s] \boldsymbol{f^s}}{ D_{A|\CC}[s] f\boldsymbol{f^s}}$   is the maximal ideal and thus it is isomorphic to a number of copies of  $  D_{A|\CC}/D_{A|\CC} \langle x_1, \dots , x_d\rangle\cong H_{\mathfrak{m}}^d(A)$. Dualizing this module we get the module of differential $d$-forms  $\Omega^d= D_{A|\CC}/ \langle \partial_1, \dots , \partial_d\rangle D_{A|\CC}$.

\begin{proposition} [{\cite[Theorem~3.3]{Yano1978}}]\label{isolated_singularities_Yano}
The reduced Bernstein-Sato polynomial $\tilde{b}_f(s)$ of an isolated singularity $f$ is the minimal polynomial of the action of $s$ on either
${\rm Hom}_{D_{A|\CC}}(\widetilde{\mathcal{M}}, H_{\mathfrak{m}}^d(A))$ or $\Omega ^n \otimes_{D_{A|\CC}} \widetilde{\mathcal{M}}$.
\end{proposition}

Then, Yano's method boils down to the following steps: 

\begin{itemize}
\item[(i)] Compute a free resolution of $\widetilde{\mathcal{M}}$ as a $D_{A|\CC}$-module
\[ 0 \leftarrow \widetilde{\mathcal{M}} \leftarrow (D_{A|\CC})^{\beta_0} \leftarrow (D_{A|\CC})^{\beta_1} \leftarrow \cdots \]

\item[(ii)] Apply the functor ${\rm Hom}_{D_{A|\CC}}(-, H_{\mathfrak{m}}^d(A))$
\[ 0 \rightarrow {\rm Hom}_{D_{A|\CC}}(\widetilde{\mathcal{M}}, H_{\mathfrak{m}}^d(A)) \rightarrow (H_{\mathfrak{m}}^d(A))^{\beta_0} \rightarrow (H_{\mathfrak{m}}^d(A))^{\beta_1} \rightarrow \cdots \]

\item[(iii)] Compute the matrix representation of the action of $s$ and its minimal polynomial.
\end{itemize}

\vskip 2mm

Yano could effectively work out some cases depending on the following invariant of the singularity:
\begin{eqnarray*}
L(f):=\min\{L\mid \delta(s)=s^L+\delta_{1}s^{L-1}+\cdots +\delta_L\in {\rm
  Ann}_{D[s]}(\boldsymbol{f^s})\mbox{ , }{\rm ord}(\delta_i)\leq i\}.
\end{eqnarray*} 
The existence of such a differential operator is given by Kashiwara \cite[Theorem 6.3]{KashiwaraRationality}. More precisely, he could describe step (1) in 
the cases $L(f)=1,2$, and $3$ where the case $L(f)=1$ is equivalent to having a quasi-homogeneous singularity.

\vskip 2mm

\subsection{Quasi-homogeneous singularities} Let $f=\sum_{\alpha} a_\alpha x_1^{\alpha_1} \cdots x_n^{\alpha_d} \in A$ be a quasi-homogeneous isolated singularity of degree $N$ with respect to a weight vector $w:=(w_1,\dots, w_d) \in \mathbb{Q}_{> 0}^d$. We have $\chi(f) = N f$  where 
\[\chi= \sum_{i=1}^d w_i x_i \partial_i \] is the Euler operator and  $ \chi - Ns \in {\rm Ann}_{D[s]}(\boldsymbol{f^s})$. Set $f_i' = \partial_i(f)$ for $i=1,\dots , d$. Yano's method is as follows:
\begin{enumerate}
\item We have a free resolution
\(  \xymatrix{ 0   &\widetilde{\mathcal{M}} \ar[l] & D_{A|\CC}   \ar[l] && (D_{A|\CC})^{n} \ar[ll]_{(f_1', \dots, f_d')}  &  0  \ar[l]} \).

\item We obtain a presentation 
\( {\rm Hom}_{D_{A|\CC}}(\widetilde{\mathcal{M}}, H_{\mathfrak{m}}^d(A)) =\{  v\in  H_{\mathfrak{m}}^d(A) \, | \, f_i' v = 0 \hskip 2mm \forall i  \} \).

\item The action of $s$ on $v \in {\rm Hom}_{D_{A|\CC}}(\widetilde{\mathcal{M}}, H_{\mathfrak{m}}^d(A))$ is the same as the action of $\frac{1}{N} \chi$. Notice that applying $\chi$ to a cohomology class $\big[\frac{1}{x_1^{\alpha_1} \cdots x_d^{\alpha_d} }\big]$ is nothing but multiplying by the  weight of this class.
\end{enumerate}

\begin{example} \label{Kato1}
Consider the quasi-homogeneous polynomial $f=x^7+y^5 \in \mathbb{C}\{x,y\}$ of degree $N=35$ with respect to the weight $w=(5,7)$. 
A basis of the vector space \[\{  v\in  H_{\mathfrak{m}}^2(A) \, | \, x^6 v = 0, y^4v=0  \}\] is given by the classes $\big[\frac{1}{x^iy^j}\big]$
with $1\leq i \leq 6$ and $1\leq j \leq 4$. The action of $\frac{1}{35} \chi= \frac{1}{35} ( 5x\partial_x + 7 y \partial_y)$ on these elements yields
{\small \[ \frac{1}{35} \chi \left(\frac{1}{xy}\right)= -\frac{12}{35}  \left(\frac{1}{xy}\right),  \frac{1}{35} \chi \left(\frac{1}{x^2y}\right)= -\frac{17}{35}  \left(\frac{1}{x^2y}\right),\dots, \frac{1}{35} \chi \left(\frac{1}{x^6y^4}\right)= -\frac{58}{35}  \left(\frac{1}{x^6y^4}\right).\]}
The matrix representation of the action of $s=\frac{1}{35} \chi $ has a diagonal form with distinct eigenvalues and thus the characteristic and the minimal polynomials coincide. The negatives of the roots of the reduced Bernstein-Sato polynomial $\tilde{b}_f(s)$, or equivalently, the roots of $\tilde{b}_f(-s)$ are 
{\Small \[\bigg\{ \frac{12}{35}, \frac{17}{35}, \frac{19}{35}, \frac{22}{35}, \frac{24}{35}, \frac{26}{35}, \frac{27}{35}, \frac{29}{35},  
\frac{31}{35}, \frac{32}{35}, \frac{33}{35}, \frac{34}{35} ,  \frac{36}{35}, \frac{37}{35}, \frac{38}{35}, \frac{39}{35}, \frac{41}{35}, \frac{43}{35}, \frac{44}{35}, \frac{46}{35},  
{\frac{48}{35}}, {\frac{51}{35}}, {\frac{53}{35}}, {\frac{58}{35}}        \bigg\}. \]}
\end{example}

\vskip 2mm 

\begin{remark} \label{x5y5}
In general, the diagonal form of the matrix representation of the action of $s$ has repeated eigenvalues so the minimal polynomial only counts them once. Take for example the quasi-homogeneous polynomial $f=x^5+y^5 \in \mathbb{C}[x,y]$ of degree $N=5$ with respect to the weight $w=(1,1)$.  The roots of $\tilde{b}_f(-s)$ are 
\[\bigg\{ \frac{2}{5}, \frac{3}{5}, \frac{4}{5}, 1, \frac{6}{5}, \frac{7}{5}, \frac{8}{5} \bigg\}. \]
\end{remark}

\begin{theorem}[{\cite{Yano1978, BGM_pre}}]
 Let $f \in A$ be a quasi-homogeneous isolated singularity of degree $N$ with respect to a weight vector $w:=(w_1,\dots, w_d) \in \mathbb{Q}_{> 0}^d$. Then, the Bernstein-Sato polynomial of $f$ is 
 \[ b_f(s)= (s+1) \prod_{\ell \in W} \left( s+ \frac{\ell}{N}\right)\]
 where $W$ is the set of weights, without repetition, of the cohomology classes in $\{  v\in  H_{\mathfrak{m}}^d(A) \, | \, f_i' v = 0 \hskip 2mm \forall i  \}$.
\end{theorem}

Recall from  \Cref{isolated_singularities_Yano} that the reduced Bernstein-Sato polynomial $\tilde{b}_f(s)$ of an isolated singularity $f$ is the minimal polynomial of the action of $s$ on  $\Omega^d \otimes_{D_{A|\CC}} \widetilde{\mathcal{M}}$. In the quasi-homogeneous case we have 
\[\Omega^d \otimes_{D_{A|\CC}} \widetilde{\mathcal{M}} \cong  A/ (f_1', \dots, f_d').\] 
Notice that the monomial basis of the Milnor algebra is dual, with the convenient shift, of the cohomology classes basis of $\{  v\in  H_{\mathfrak{m}}^d(A) \, | \, f_i' v = 0 \hskip 2mm \forall i  \}$. In this case, the action of $s$  is $-\frac{1}{N}(\chi + \sum_{i=1}^n w_i)$.

\vskip 2mm

\subsection{ Irreducible plane curves}  Some of the examples considered by Yano deal with the case of plane curves and his methods were used by
Kato  to compute the following  example  which is a continuation of \Cref{Kato1}.

\begin{example} [{\cite{Kato1981}}]\label{Kato2}
The roots of $\tilde{b}_f(-s)$  for $f=x^7+y^5$ are 
{\tiny \[\bigg\{ \!\underbrace{\frac{12}{35}, \frac{17}{35}, \frac{19}{35}, \frac{22}{35}, \frac{24}{35}, \frac{26}{35}, \frac{27}{35}, \frac{29}{35},  
\frac{31}{35}, \frac{32}{35}, \frac{33}{35}, \frac{34}{35} }_{\lambda },   \underbrace{\frac{36}{35}, \frac{37}{35}, \frac{38}{35}, \frac{39}{35}, \frac{41}{35}, \frac{43}{35}, \frac{44}{35}, \frac{46}{35},  
\boxed{\frac{48}{35}}, \boxed{\frac{51}{35}}, \boxed{\frac{53}{35}}, \boxed{\frac{58}{35}}}_{2- \lambda }   \!      \bigg\}. \]}
Notice that the roots are symmetric with respect to $1$ and we point out that those $\lambda <1 $ are jumping numbers of the multiplier ideals of $f$ (see \Cref{Multipliers}). Now consider a  deformation of the singularity, 
\[f_t=x^7+y^5 - t_{3,3} x^3y^3 - t_{5,2} x^5y^2 - t_{4,3} x^4y^3 - t_{5,3} x^5y^3.\] 
Then we have a stratification of the space of parameters where some of the roots  of $\tilde{b}_f(-s)$ may change. More precisely, the boxed roots may change to the same root shifted by $1$. 

\vskip 2mm

\hskip .3cm $\{ t_{3,3} = 0 , t_{5,2}=0 , t_{4,3}= 0, t_{5,3}\neq 0 \}$. The root $\frac{58}{35}$ changes to $\frac{23}{35}$.

\hskip .3cm $\{ t_{3,3} = 0 , t_{5,2}=0, t_{4,3}\neq 0 \}$. The roots $\frac{58}{35}, \frac{53}{35}$ change to $\frac{23}{35}, \frac{18}{35}$. 

\hskip .3cm $\{ t_{3,3} = 0 , t_{5,2}\neq 0, t_{4,3}= 0 \}$. The roots $\frac{58}{35}, \frac{51}{35}$ change to $\frac{23}{35}, \frac{16}{35}$. 

\hskip .3cm $\{ t_{3,3} = 0 , t_{5,2}t_{4,3}\neq 0 \}$. The roots $\frac{58}{35}, \frac{53}{35}, \frac{51}{35}$ change to $\frac{23}{35}, \frac{18}{35}, \frac{16}{35}$. 

\hskip .3cm $\{ t_{5,2} \neq 0 , 6t_{5,2} +175t_{3,3}^4=0 \}$. The roots $\frac{58}{35}, \frac{53}{35}, \frac{48}{35}$ change to $\frac{23}{35}, \frac{18}{35}, \frac{13}{35}$.

\hskip .3cm $\{ t_{5,2} \neq 0 , 6t_{5,2} +175t_{3,3}^4\neq 0 \}$. The roots $\frac{58}{35}, \frac{53}{35}, \frac{51}{35}, \frac{48}{35}$ change to $\frac{23}{35}, \frac{18}{35}, \frac{16}{35}, \frac{13}{35}$. 

\vskip 2mm

\noindent In this last stratum we have a Zariski open set where the roots are
{\tiny \[\bigg\{ \frac{12}{35}, \boxed{\frac{13}{35}}, \boxed{\frac{16}{35}}  \frac{17}{35}, \boxed{\frac{18}{35}}, \frac{19}{35}, \frac{22}{35}, \boxed{\frac{23}{35}},\frac{24}{35}, \frac{26}{35}, \frac{27}{35}, \frac{29}{35},  
\frac{31}{35}, \frac{32}{35}, \frac{33}{35}, \frac{34}{35} ,   \frac{36}{35}, \frac{37}{35}, \frac{38}{35}, \frac{39}{35}, \frac{41}{35}, \frac{43}{35}, \frac{44}{35}, \frac{46}{35}   \bigg\}, \]}
\noindent and thus they are in the interval $[ {\rm lct}(f), {\rm lct}(f) + 1 )$. We say that these are the generic roots of the Bernstein-Sato polynomial of $f_t$.
\end{example}

An interesting issue in this example is that, even though they have different Bernstein-Sato polynomials,  all the fibres of the deformation $f_t$ have the same Milnor number so they belong to the same equisingularity class. Roughly speaking, all the fibres have the same log-resolution meaning that they have the same combinatorial information, which can be encoded in weighted graphs such as  the Enriques diagram\index{Enriques diagram} \cite[\S IV.I]{EC15} \cite[\S 3.9]{Cas00}, the dual graph\index{dual graph} \cite[\S 4.4]{Cas00} \cite[\S 3.6]{Wall04} or the Eisenbud-Neumann\index{Eisenbud-Neumann diagram} diagrams \cite{EN85}.

\vskip 2mm

From now on let \( f \in \mathbb{C}\{x,y\} \) be a defining equation of the germ of an irreducible plane curve. 
A complete set of numerical invariants for the equisingularity class of $f$ is given by the \emph{characteristic exponents}
\((n, \beta_1, \dots, \beta_g) \) where $n \in \mathbb{Z}_{>0}$ is the multiplicity at the origin of $f$ and 
the integers \( n < \beta_1 < \cdots < \beta_g \) can be obtained from the Puiseux parameterization of $f$.
To describe the equisingularity class of $f$ we may also consider its \emph{semigroup} \( \Gamma:= \langle \overline{\beta}_{0}, \overline{\beta}_{1}, \dots, \overline{\beta}_{g} \rangle \) that comes from the valuation of $\mathbb{C}\{x,y\}/\langle f\rangle$ given by the Puisseux parametrization of $f$.

\vskip 2mm

A quasihomogeneous plane curve $f=x^a+y^b$ with $a<b$ and gcd$(a,b)=1$ is irreducible with semigroup \( \Gamma= \langle a,b \rangle \).
Adding higher order terms $x^iy^j$  with $bi + aj > ab$ does not change the equisingularity class but we do not need all the higher order terms. Indeed, every irreducible curve with semigroup \( \Gamma= \langle a,b \rangle \) is analytically isomorphic to one of the fibers of the miniversal deformation  $$f=x^a+y^b - \sum t_{i,j} \hskip 1mm x^i y^j,$$ where  the sum is
taken over the monomials $x^i y^j$ such that $0\leq i \leq a-2$,
$0\leq j \leq b-2$ and $bi + aj > ab$. This is the setup considered in \Cref{Kato2}.

\vskip 2mm

Cassou-Nogu\`es  \cite{Cassou1987} described the stratification by the Bernstein-Sato polynomial of any irreducible plane curve with a single characteristic exponent using analytic continuation of the complex zeta function. 

\vskip 2mm

To construct a miniversal deformation of an irreducible plane curve with $g$ characteristic exponents is much more complicated and one has to use, following Teissier \cite{Zariski_Appendix},  the monomial curve $C^\Gamma$  associated to the semigroup \( \Gamma= \langle \overline{\beta}_{0}, \overline{\beta}_{1}, \dots, \overline{\beta}_{g} \rangle \) by the parametrization \( u_i = t^{\overline{\beta}_{i}} , \, i=1,\dots, g\).  Teissier  proved the existence of a miniversal semigroup constant deformation of this monomial curve. It turns out that every irreducible plane curve with semigroup $\Gamma$ is analytically isomorphic to one of the fibres of the miniversal deformation of $C^\Gamma$. To give explicit equations in 
\( \mathbb{C}\{x,y\} \) is more complicated and we refer to the work of Blanco \cite{Bla19} for more details. For the convenience of the reader we  illustrate an example with two characteristic exponents.

\begin{example}
 The semigroup  of an irreducible plane curve 
\( f= (x^a+y^b)^c + x^iy^j\) with $bi + a j =d$ is \( \Gamma= \langle ac,bc,d \rangle \). All the fibres of the deformation 
\[ f_t= \left(x^a+y^b+ \sum_{bk+a\ell > ab} t_{k,\ell} x^ky^\ell\right)^c + x^iy^j +\sum_{bck+ac\ell +dr> cd} t_{k,\ell} x^ky^\ell\big(x^a + y^b\big)^r\]
have the same semigroup.
\end{example}

The ultimate goal would be to find a stratification by the Bernstein-Sato polynomial of all the irreducible plane curves with a fixed semigroup but this turns out to be a wild problem. However, one may ask about the roots of the Bernstein-Sato polynomial of a generic fibre of a deformation of an irreducible plane curve with a given semigroup. That is, to find the roots in a Zariski open set in the space of parameters of the deformation that we call the generic roots of the Bernstein-Sato polynomial.

\vskip 2mm

Amazingly, Yano \cite{Yano1982} conjectured a formula for the generic $b$-exponents (instead of the generic roots) of any irreducible plane curve. 
These generic $b$-exponents can be described in terms of the semigroup $\Gamma$ but we use a simple interpretation in terms of the numerical data of a log-resolution of $f$. 
Let  $\pi: X' \rightarrow \CC^n$ be a {log-resolution} of an  irreducible plane curve with $g$ characteristic exponents. Let $F_\pi$ be the total transform divisor and $K_\pi$ the relative canonical divisor. In this case we have $g$ distinguished exceptional divisors, the so-called \emph{rupture divisors} that  intersect three or more divisors in the support of $F_\pi$. For simplicity we denote them by $E_1,\dots, E_g$ with the corresponding values $N_i$ and $k_i$ in $F_\pi$ and $K_\pi$ respectively.

\begin{conjecture}[{\cite{Yano1982}}] \label{thm-yano}\index{Yano's conjecture}
Let \( f \in \mathbb{C}\{x,y\} \) be a defining equation of the germ of an irreducible plane curve with semigroup \( \Gamma = \langle \overline{\beta}_0, \overline{\beta}_1, \dots, \overline{\beta}_g \rangle \). Then, for generic curves in some \( \Gamma \)-constant deformation of \( f \), the \( b \)-exponents are
\[
{\displaystyle \bigcup_{i = 1}^g} \bigg\{ \lambda_{i, \ell} = \frac{k_i + 1 +\ell}{N_i } \ \bigg|\ 0 \leq \ell < N_i,\  \overline{\beta}_i \lambda_{i, \ell} \not\in \mathbb{Z} , e_{i-1} \lambda_{i, \ell} \not\in \mathbb{Z} \bigg\}
\] where $e_{i-1}=\gcd (\overline{\beta}_0, \overline{\beta}_1, \dots, \overline{\beta}_{i-1} )$.
\end{conjecture}

If we consider the irreducible plane curve studied by Kato in \Cref{Kato2} we see that Yano's conjecture holds true.

\begin{example}
The Yano set associated to the semigroup \( \Gamma = \langle 5,7\rangle \) is 
\[
 \bigg\{ \lambda_{1, \ell} = \frac{12 +\ell}{35 } \ \bigg|\ 0 \leq \ell < 35,\  7 \lambda_{1, \ell} \not\in \mathbb{Z} , 5\lambda_{1, \ell} \not\in \mathbb{Z} \bigg\}
\] 
which gives the generic $b$-exponents given in \Cref{Kato2}.
\end{example}

From the stratification given by Cassou-Nogu\`es \cite{Cassou1987} one gets that Yano's conjecture is true for  irreducible plane curves with a single characteristic exponent (see \cite{Cassou1988}).
Almost thirty years later, Artal-Bartolo, Cassou-Nogu\`es, Luengo, and Melle-Hern\'andez \cite{ACNLM17} proved Yano's conjecture for irreducible plane curves with two characteristic exponents with the extra assumption that the eigenvalues of the monodromy are different. Under the same extra condition, Blanco \cite{Bla19} gave a proof for any number of characteristic exponents. Both papers use the analytic continuation of the complex zeta function.
The extra condition on the eigenvalues of the monodromy being different ensures that the characteristic and the minimal polynomial of the action of $s$ on $(s+1) \frac{D_{A|\CC}[s] \boldsymbol{f^s}}{ D_{A|\CC}[s] f\boldsymbol{f^s}}$ are the same. 

\vskip 2mm

The shortcomings of the analytic continuation techniques, which deal with the Bernstein-Sato polynomial instead of the $b$-exponents, can be seen in examples such as the following.

\begin{example}
The Yano sets associated to the semigroup \( \Gamma = \langle 10,15,36\rangle \) are
\[
 \bigg\{ \lambda_{1, \ell} = \frac{5 +\ell}{30 } \ \bigg|\ 0 \leq \ell < 30,\  15 \lambda_{1, \ell} \not\in \mathbb{Z} , 10\lambda_{1, \ell} \not\in \mathbb{Z} \bigg\},
\] 
and
\[
 \bigg\{ \lambda_{2, \ell} = \frac{31 +\ell}{180 } \ \bigg|\ 0 \leq \ell < 180,\  36 \lambda_{2, \ell} \not\in \mathbb{Z} , 5\lambda_{2, \ell} \not\in \mathbb{Z} \bigg\}.
\] 
We have that $\frac{11}{30}, \frac{17}{30}, \frac{23}{30}, \frac{29}{30}$ appear in both sets. Therefore they appear with multiplicity $2$ as $b$-exponents but only once as roots of the Bernstein-Sato polynomial.
\end{example}

Blanco \cite{Blanco_yano} has recently proved Yano's conjecture in its generality. His work uses periods of integrals along vanishing cycles on the Milnor fiber as considered by Malgrange \cite{Mal74b, Mal74} and Varchenko \cite{Var80, Var81}.  In particular he extends vastly the results of Lichtin \cite{Lic89}  and Loeser \cite{Loeser88} on the expansions of these periods of integrals.

\vskip 2mm

\subsection{ Hyperplane arrangements} Let $f \in \CC[x_1,\dots ,  x_d]$ be a reduced polynomial defining an arrangement of hyperplanes so $f=f_1 \cdots f_\ell$ decomposes as a product of polynomials $f_i$ of degree one. The Bernstein-Sato polynomial of $f$ has been studied by Walther \cite{WaltherBS} under the assumptions that the arrangement is:
\begin{itemize}
\item \emph{Central:} $f$ is homogeneous so all the hyperplanes contain the origin.
\item \emph{Generic:} The intersection of any $d$ hyperplanes  is the origin.
\end{itemize}

The main result of Walther, with the assistance of Saito \cite{Saito_BS_arrangements} to compute the multiplicity of $-1$ as a root, is the following.

\begin{theorem} [{\cite{WaltherBS, Saito_BS_arrangements}}]
The Bernstein-Sato polynomial of a generic central hyperplane arrangement $f \in \CC[x_1,\dots ,  x_d]$ of degree $\ell \geq d$ is 
$$b_f(s)= (s+1)^{d-1} \prod_{j=0}^{2\ell-d-2} \left(s + \frac{j+d}{\ell}\right).$$
\end{theorem}

\begin{example}
The homogeneous polynomial $f=x^5+y^5 \in \CC[x,y]$ considered in \Cref{x5y5} defines an arrangement of five lines through the origin. Walther's formula gives
$$b_f(s)= (s+1)^2\left(s+\frac{2}{5}\right)\left(s+\frac{3}{5}\right)\left(s+\frac{4}{5}\right)\left(s+\frac{6}{5}\right)\left(s+\frac{7}{5}\right)\left(s+\frac{8}{5}\right).$$
\end{example}

It is an open question to determine the roots of the Bernstein-Sato polynomial of a nongeneric arrangement. In this general setting, Leykin  \cite{WaltherBS} noticed that $-1$ is the only integer root of  $b_f(s)$.

A natural question that arise when dealing with invariants of hyperplane arrangements is whether these invariants are \emph{combinatorial}, meaning that they only depend on the lattice of intersection of the hyperplanes together with the codimensions of these intersections, and it does not depend on the position of the hyperplanes.  Unfortunately this is not the case. Walther \cite{Walther_inv} provides examples of combinatorially equivalent arrangements with different Bernstein-Sato polynomial.

\begin{example} [{\cite{Walther_inv, Saito_BS_arrangements}}]
The following nongeneric arrangements have the same intersection lattice
\begin{itemize}
\item[]  \hskip -1cm $f = xyz(x+3z)(x+y+z)(x+2y+3z)(2x+y+z)(2x+3y+z)(2x+3y+4z),$
\item[]  \hskip -1cm $g = xyz(x + 5z)(x + y + z)(x + 3y + 5z)(2x + y + z)(2x + 3y + z)(2x + 3y + 4z).$
\end{itemize}
However the Bernstein-Sato polynomials differ by the root $-\frac{16}{9}$:
\begin{align*} b_f(s) &= (s+1) \prod_{j=2}^4 \left( s+ \frac{j}{3} \right) \prod_{j=3}^{16} \left( s+ \frac{j}{9} \right)\\
b_g(s) &= (s+1) \prod_{j=2}^4 \left( s+ \frac{j}{3} \right) \prod_{j=3}^{15} \left( s+ \frac{j}{9} \right).
\end{align*}
\end{example}

\section{The case of nonprincipal ideals and relative versions}\label{SecNonPrincipalRelative}

In this section we study different extensions  of  Bernstein-Sato polynomials for ideals that are not necessarily principal. Sabbah \cite{Sabbah_ideal} introduced the notion of Bernstein-Sato ideal $B_F \subseteq \KK[s_1,\dots, s_\ell]$ associated to a tuple of elements $F=f_1,\dots, f_\ell $. More recently, Budur, Musta\c{t}\u{a}, and Saito \cite{BMS2006a} defined a Bernstein-Sato polynomial $b_{\fa}(s) \in \KK[s]$ associated to an ideal $\fa \subseteq A$ which is independent of the set of generators.  The approach to Bernstein-Sato polynomials of nonprincipal ideals has been simplified by Musta\c{t}\u{a} \cite{Mustata2019}.

In order to provide a description of the $V$-filtration of a holonomic $D$-module, Sabbah introduced a relative version of Bernstein-Sato polynomials that is also considered in  the version for nonprincial ideals \cite{BMS2006a}. This relative version is also important to describe multiplier ideals (see \Cref{Multipliers}).

\subsection{Bernstein-Sato polynomial for general ideals in differentiably admissible algebras}\label{SubNonPrincipalDiffAd}

We  start studying the Bernstein-Sato polynomial for general ideals using the  recent approach given by  Musta\c{t}\u{a} \cite{Mustata2019}.  In this section we show its existence for  general ideals in differentiably admissible algebras in Theorem \ref{ThmExistinceNonPrincipal}.

\begin{definition}
Let $\KK$ a field of characteristic zero, $A$ be a regular  $\KK$-algebra, and $\fa\subseteq A$ be a nonzero ideal.
Let $F=f_1,\ldots,f_\ell$ be a set of generators for $\fa$, and $g=f_1y_1+\cdots+f_\ell y_\ell\in A[y_1,\ldots,y_\ell]$.
We denote by  $b_{F}(s)$ the monic polynomial in $\KK[s]$ of least degree among those polynomials $b(s)\in \KK[s]$  such that
	\[ \delta(s) g^{s+1} = b(s) g^s \qquad \text{for all \ } s\in \NN,\]
	where $\delta(s)\in D_{ A[y_1,\ldots,y_\ell]|\KK}[s]$ is a polynomial differential operator. That is, $b_{F}(s)$ is the Bernstein-Sato polynomial of $g$.
\end{definition}

Before we discuss properties of this notion of the Bernstein-Sato polynomial, we show that the definition of $b_{F}(s)$ does not depend on the choice of generators for $\fa$. 

\begin{proposition}[{\cite[Remark 2.1]{Mustata2019}}]\label{PropWellDefMM}
Let $\KK$ a field of characteristic zero, $A$ be a regular  $\KK$-algebra, and $\fa\subseteq A$ be a nonzero ideal.
Let $F=f_1,\ldots,f_\ell$ and $G={g}_1,\ldots,{g}_{m} $ be two sets of generators for $\fa$. Then $b_{F}(s)=b_{G}(s)$.
\end{proposition}
\begin{proof}
It suffices to show that
$b_{F}(s)=b_{G}(s)=b_{H}(s),$ where $H=F \cup G$.
This follows from showing that $b_{F}(s)=b_{G}(s)$ when $G=F\cup g$
for $g\in \fa$.
Let $r_1,\ldots, r_\ell$ such that $g=r_1 f_1+\cdots+r_\ell f_\ell$.
We have that 
 \begin{align*}
f_1y_1+\cdots+f_\ell y_\ell+g y_{\ell+1}&=
f_1y_1+\cdots+f_\ell y_\ell+(r_1 f_1+\cdots+r_\ell f_\ell) y_{\ell+1}\\
& f_1(y_1+r_1 y_{\ell+1})+\cdots+f_\ell (y_\ell+r_\ell y_{\ell+1}).
\end{align*}
After a change of variables $y_i\mapsto y_i+r_i y_{\ell+1}$, this polynomial becomes $f$.
Since the Bernstein-Sato polynomial does not change by change of variables, we conclude that  $b_{F}(s)=b_{G}(s)$.
\end{proof}

Given the previous result, we can define the Bernstein-Sato polynomial of a nonprincipal ideal. Notice that  $f_1y_1+\cdots+f_\ell y_\ell$ is not a unit in $A[y_1,\ldots,y_\ell]$ so we may consider its  reduced Bernstein-Sato polynomial $\tilde{b}_{F}(s) = \frac{b_{F}(s)}{s+1}$.

\begin{definition}
Let $\KK$ a field of characteristic zero, $A$ be a regular  $\KK$-algebra, and $\fa\subseteq A$ be a nonzero ideal.
Let $F=f_1.\ldots,f_\ell$ be a set of generators for $\fa$.
We define the Bernstein-Sato polynomial of $\fa$ as  the reduced Bernstein-Sato polynomial of $f_1y_1+\cdots+f_\ell y_\ell$. That is $$b_{\fa}(s) :=\tilde{b}_{F}(s).$$
\end{definition}

We point out that the previous definition is not the original given by  Budur, Musta\c{t}\u{a}, and Saito \cite{BMS2006a}, which we discuss in the next subsection. 
This approach  given by Musta\c{t}\u{a} \cite{Mustata2019} has a couple of differences. First, the existence of Bernstein-Sato polynomials for nonprincipal ideals would follow from the existence of certain Bernstein-Sato polynomials for a single element. This way in particular gives the existence of Bernstein-Sato polynomials for nonprincipal ideals in any differentiably admissible 
algebras (see Subsection~\ref{SubSecDifAd}) such as power series rings over a field of characteristic zero.
Second, the treatment given by Musta\c{t}\u{a} \cite{Mustata2019} can be done without using $V$-filtrations.

We now focus on showing the existence of Bernstein-Sato polynomial for nonprincipal ideals in differentiably admissible algebras.
We start recalling a theorem from Matsumura's book \cite{Matsumura}.

\begin{theorem}[{\cite[Theorem 99]{Matsumura}}]\label{99}
Let $(A,\m,\KK)$ be a regular local commutative Notherian ring with unity of dimension $d$ containing a field $\KK_0$. Suppose that $\KK$ is
an algebraic separable extension of $\KK_0$. Let $\hat{A}$ denote the completion of $A$ with respect to~$\m$.
Let $x_1,\ldots,x_d$ be a regular system of parameters of $A$. Then, $\widehat{A}=\KK\llbracket x_1,\dots,x_d\rrbracket$ is the power 
series ring with coefficients in $\KK$, and $\Der_{\hat{A}|\KK}$ is a free $\widehat{A}$-module with
basis $\partial_1, \ldots,\partial_d$. Moreover, the following conditions are equivalent:
\begin{enumerate}
\item $\partial_i$ ($i=1,\ldots , d$) maps $A$ into $A$, equivalently, $\partial_i\in Der_{A|\KK_0}$;
\item there exist derivations $\delta_1,\ldots, \delta_d \in \Der_{A|\KK_0} $ and elements $f_1,\ldots, f_d\in A$ such that $\delta_i f_j=1$ if $i=j$ and 
$0$ otherwise;
\item there exist derivations $\delta_1,\ldots, \delta_d \in \Der_{A|\KK_0}$ and elements $f_1\ldots, f_d\in R$ such that $\det(\delta_i f_j) \not\in \m$;
\item $\Der_{A|\KK_0}$ is a free module of rank $d$ (with basis $\delta_1,\ldots, \delta_d$);
\item $\Rank(\Der_{A|\KK_0})=d$.
\end{enumerate}
\end{theorem}

We now show that a power series ring over a differentiably admissible $\KK$-algebra is also a differentiably admissible $\KK$-algebra.
We point out that this fact does not hold for polynomial rings, as the residue field can be a transcendental extension of $R$. A example of this is $A=\KK\llbracket x\rrbracket$, where $\n=(xy-1)\subseteq A[y]$ is a maximal ideal with residue field $\Frac(A)$.

\begin{proposition}\label{PropPowerSeriesDiffAd}
Let $A$ be a differentiably admissible $\KK$-algebra of dimension $d$.
Then, the power series ring $A\llbracket y\rrbracket$ is also a differentiably admissible $\KK$-algebra of dimension $d+1$.
\end{proposition}
\begin{proof}
Since every regular Noetherian ring is product of regular domains, we assume without loss of generality that $A$ is a domain. 
Let $\n$ be a maximal ideal in $A\llbracket y\rrbracket$. Then, there exists a maximal ideal $\m\subseteq A$ such that $\n=\m A\llbracket y\rrbracket+(y)$. It follows that
$\n$ is generated by a regular sequence of $d+1$ elements. We conclude that $(A\llbracket y\rrbracket)_\n$ is a regular ring of dimension $d+1$.
We also have that $A\llbracket y\rrbracket/\n \cong A/\m$ is an algebraic extension of $\KK$.
 
 It remains to show that $\Der_{A\llbracket y\rrbracket|K}$ is a projective module of rank $d+1$ and it behaves well with localization.
 We note that every derivation $\delta$ in $A$ can be extended to a derivation  $A\llbracket y\rrbracket$ by $\delta(\sum^{\infty}_{n=0} f_n y^n)=\sum^{\infty}_{n=0} \delta(f_n) y^n$.  
 Let $M=A\llbracket y\rrbracket\otimes_A  \Der_{A|K} \oplus A\llbracket y\rrbracket \partial_y\subseteq \Der_{A\llbracket y\rrbracket|K}$.
We note that the natural maps
$$
M_\n\to A\llbracket y\rrbracket_\n \otimes_A \Der_{A\llbracket y\rrbracket | \KK} \to \Der_{A\llbracket y\rrbracket_\n|\KK}
$$ 
 are  injective.
 We fix $\n\subseteq A \llbracket y\rrbracket$  a maximal ideal and a maximal ideal $\m\subseteq R$ such that $\n=\m A\llbracket y\rrbracket+(y)$.
We fix $\delta_1,\ldots,\delta_d\in \Der_{A_\m |\KK}$ and elements $f_1,\ldots, f_n\in \m A_\m$ such that $\delta_i f_j=1$ if $i=j$ and 
$0$ otherwise. We can do this by Theorem \ref{99}.
Then, $\delta_1,\ldots,\delta_d, {\partial_y}$ satisfy Theorem \ref{99}(3). We conclude that 
$\delta_1,\ldots,\delta_d, {\partial_y}$ generate $\Der_{A\llbracket y\rrbracket_\n|\KK}$. Then, the composition of the maps
$$M_\n\to A\llbracket y\rrbracket_\n \otimes_A \Der_{A\llbracket y\rrbracket | \KK} \to \Der_{A\llbracket y\rrbracket_\n|\KK}$$
 is surjective.
 We conclude that they are isomorphic.
Since 
$$
M_\m=\big(A\llbracket y\rrbracket_\n\otimes_{A_\m}  (\Der_{A|\KK})_\m\big) \oplus  A\llbracket y\rrbracket_\n {\partial_y}
$$ 
 is free of rank $d+1$, we have that
 $$
 (M_\m)_\n=M_\n\cong \Der_{A\llbracket y\rrbracket_\n|\KK}
 $$
 is free of rank $d+1$.
\end{proof}

\begin{theorem}\label{ThmExistinceNonPrincipal}
Let $A$ be differentiably admissible, and $\fa\subseteq A$.
Then, the Bernstein-Sato polynomial of $\fa$ exists.
\end{theorem}
\begin{proof}
Let $f_1,\ldots,f_\ell$ be a set of generators for $\fa$.
Let $f=f_1 y_1+\cdots+f_\ell y_\ell\in A \llbracket y_1,\ldots, y_\ell\rrbracket $.
There exists $b(s)\in\KK[s]\smallsetminus\{0\}$ and $\delta(s)\in A\llbracket y_1,\ldots, y_\ell\rrbracket [s]$ such that
$$
\delta(s) f  \boldsymbol{f^s}=b(s) \boldsymbol{f^s}
$$
in $A_f[s] \boldsymbol{f^s}$ by Proposition \ref{PropPowerSeriesDiffAd} and Theorem \ref{ThmExistenceHomological}.
There exist finitely many $\beta\in\NN^\ell$, $j\in \NN$, $\delta_{\beta,j}[s]\in D_{A | \KK }[s]$, and $g_{\beta, j}\in  A\llbracket y_1,\ldots, y_\ell\rrbracket$ such that
$$
\delta(s) =\sum_{\beta,j} g_{\beta,j} \delta_{\beta,j}(s) \frac{\partial^\beta}{\partial y^\beta}
$$
because $D_{A\llbracket y_1,\ldots, y_\ell\rrbracket | \KK }$ is generated by derivations by Remark~\ref{RmkDiffAdGenDer}, and by the description of 
$\Der_{A\llbracket y_1,\ldots, y_\ell\rrbracket|\KK}$ in the proof of Proposition \ref{PropPowerSeriesDiffAd}.
Then, there exists $h_{\alpha,\beta,j}\in A$ such that $ g_{\beta,j}=\sum_{\alpha\in\NN^\ell} h_{\alpha,\beta,j}y^\alpha$.
Then,
$$
\delta(s) =\sum_{\beta,j}  \sum_{\alpha\in\NN^\ell} h_{\alpha,\beta,j}  \delta_{\beta,j}(s)  y^\alpha    \frac{\partial^\beta}{\partial y^\beta}.
$$
We have that
\begin{align*}
b(s) \boldsymbol{f^s}&=\delta(s) f  \boldsymbol{f^s}\\
&=    \sum_{\beta,j}  \sum_{\alpha\in\NN^\ell} h_{\alpha,\beta,j}y^\alpha \delta_{\beta,j}(s)   \frac{\partial^\beta}{\partial y^\beta}     f  \boldsymbol{f^s}\\
&=  \sum_{\beta,j}  \sum_{\alpha\in\NN^\ell} h_{\alpha,\beta,j}\delta_{\beta,j}(s)  y^\alpha \frac{\partial^\beta}{\partial y^\beta}     f  \boldsymbol{f^s}.
\end{align*}
After specializing for $t\in \NN$,
we have that
$$
b(t) f^t=  \sum_{\beta,j}  \sum_{\alpha\in\NN^\ell} h_{\alpha,\beta,j}\delta_{\beta,j}(t)  y^\alpha \frac{\partial^\beta}{\partial y^\beta}     f^{t+1}.
$$
Then,
$$
 \sum_{\beta,j}  \sum_{|\alpha|\neq |\beta|-1} h_{\alpha,\beta,j}\delta_{\beta,j}(t)  y^\alpha \frac{\partial^\beta}{\partial y^\beta}     f^{t+1}=0.
 $$
 by comparing the degree in $y_1,\ldots,y_\ell$.
 Then,
 $$
  \sum_{\beta,j}  \sum_{|\alpha|\neq |\beta|-1} h_{\alpha,\beta,j}\delta_{\beta,j}(s)  y^\alpha \frac{\partial^\beta}{\partial y^\beta}     f  \boldsymbol{f^s}=0.
 $$
 We have that 
$$
 \tilde{\delta}(s)  =   \sum_{\beta,j}  \sum_{|\alpha|= |\beta|-1} h_{\alpha,\beta,j}\delta_{\beta,j}(s)  y^\alpha \frac{\partial^\beta}{\partial y^\beta}.
 $$
 satisfies the functional equation and belongs to $D_{A[y_1,\ldots,y_\ell ] | \KK}[s]$.
 Then, the Bernstein-Sato polynomial of $\fa$ exists.
 \end{proof}
 
 \subsection{Bernstein-Sato polynomial of general ideals revisited}\label{SubGeneralIdeals2}
 In this subsection we  review the original definition of Bernstein-Sato polynomial of an ideal given by Budur, Musta\c{t}\u{a}, and Saito \cite{BMS2006a}. Indeed they provide two equivalent approaches depending on the ring of differential operators we are working with.

Let $\KK$ a field of characteristic zero, $A$ be a regular  $\KK$-algebra, and  let $F=f_1,\dots, f_\ell$ be a set of generators of an ideal $\fa\subseteq A$. Let 
$S=\{s_{ij}\}_{1\leq i,j\leq \ell}$ be a new set of variables satisfying the following relations:
\begin{itemize}
\item[(i)] $s_{ii}=s_i$ for $i=1,\dots, \ell$.
\item[(ii)] $[s_{ij}, s_{k\ell}] = \delta_{jk}s_{i\ell} - \delta_{i\ell}s_{kj}$,
\end{itemize}
where $\delta_{ij}$ is the Kronecker's delta function. Then we consider the ring $\KK\langle S \rangle$ generated by $S$ and $D_{A|\KK}\langle S \rangle:= D_{A|\KK} \otimes_{\KK} \KK \langle S \rangle$.

In this setting we have the following Bernstein-Sato type functional equation.


\begin{definition}
Let $\KK$ be a field of characteristic zero and $A$ a regular $\KK$-algebra.
	A \emph{Bernstein-Sato functional equation} in $D_{A|\KK}\langle S \rangle$ for  $F= f_1,\dots, f_\ell$ is an equation of the form
	\[ \sum_{i=1}^\ell \delta_i(S) f_i  f_1^{s_1} \cdots f_\ell^{s_\ell}= b(s_1+ \cdots + s_\ell) f_1^{s_1} \cdots f_\ell^{s_\ell} \]
	where $\delta_i(S)\in D_{A|\KK}\langle S \rangle$  and $b(s)\in \KK[s]$. 
\end{definition}

\begin{definition}
Let $\KK$ be a field of characteristic zero and $A$ a regular $\KK$-algebra.  Let $F=f_1,\dots, f_\ell$ be a set of generators of an ideal $\fa\subseteq A$. The Bernstein-Sato polynomial $b_{\mathfrak{a}} (s)$ of $\mathfrak{a}$ is the monic polynomial of smallest degree satisfying a Bernstein-Sato functional equation in $D_{A|\KK}\langle S \rangle$.
\end{definition}

 Budur, Musta\c{t}\u{a}, and Saito  proved the existence of such Bernstein-Sato polynomial. 
 Moreover, they also proved that it does not depend on the set of generators of the ideal so it 
is well-defined (see \cite[Theorem 2.5]{BMS2006a}).

After a convenient shifting we can define the Bernstein-Sato polynomial of an algebraic variety.

\begin{theorem}[{\cite{BMS2006a}}]
 Let $Z(\mathfrak{a}) \subseteq  \CC^d$ be the closed variety defined by an ideal
$\mathfrak{a}\subseteq A$ and $c$ be the codimension of $Z(\mathfrak{a})$ in $\CC^d$. Then 
$$b_{Z(\mathfrak{a})} (s) := b_{\mathfrak{a}} (s - c)$$ depends only on the affine scheme $Z(\mathfrak{a})$ and not on $\mathfrak{a}$.
\end{theorem}

In this setting we also have that the Bernstein-Sato functional equation in $D_{A|\KK}\langle S \rangle$ is  an equality in $A_f[s_1,\dots, s_p] \boldsymbol{f^s}$.  The $D_{A|\KK}\langle S \rangle$-module structure on this module is given by
$$ s_{ij} \cdot a(s_1,\dots , s_p) \boldsymbol{f^s} := s_i a(s_1,\dots, s_i-1, \dots, s_j +1, \dots , s_p) \frac{f_j}{f_i}\boldsymbol{f^s}$$
where $ a(s_1,\dots , s_p) \in A_f[s_1,\dots, s_p]$. The $D_{A|\KK}\langle S \rangle$-submodule generated by  $\boldsymbol{f^s}$ has a presentation
\[ {D_{A|\KK}\langle S \rangle\boldsymbol{f^s}} \cong  \frac{D_{A|\KK}\langle S \rangle}{\mathrm{Ann}_{D\langle S \rangle}(\boldsymbol{f^s}) }, \] 
and thus
\[ \frac{D_{A|\KK}\langle S \rangle\boldsymbol{f^s}}{D_{A|\KK}\langle S \rangle(f_1,\dots, f_p) \boldsymbol{f^s}} \cong  \frac{D_{A|\KK}\langle S \rangle}{\mathrm{Ann}_{D\langle S \rangle}(\boldsymbol{f^s}) + D_{A|\KK}\langle S \rangle (f_1,\dots, f_p) }. \]
We have an analogue of \Cref{prop:othercharsBS} that is used in order to provide algorithms for the computations of these Bernstein-Sato polynomials \cite{AndresLevMM}.

\begin{proposition}\label{prop:othercharsBSI}
The Bernstein-Sato polynomial of an ideal  $\mathfrak{a}\subseteq A$ generated by $F=f_1,\dots, f_\ell$  is the monic generator of the ideal $$\big( b_{\fa}(s_1 +\dots + s_p)\big)= \KK[s_1 +\dots + s_p] \cap \big(\mathrm{Ann}_{D\langle S \rangle}(\boldsymbol{f^s}) + D_{A|\KK}\langle S \rangle (f_1,\dots, f_p)\big).$$
\end{proposition}

Budur, Musta\c{t}\u{a}, and Saito \cite[Section 2.10]{BMS2006a} gave an equivalent definition of Bernstein-Sato polynomial of $\fa$ using a functional equation in $D_{A|\KK}[s_1,\dots, s_\ell]$  instead of $D_{A|\KK}\langle S \rangle$.

\begin{theorem}[\cite{BMS2006a}]
Let $\KK$ a field of characteristic zero, $A$ be a regular  $\KK$-algebra, and $\fa\subseteq A$ be a nonzero ideal.
Let $F=f_1,\ldots,f_\ell$ be a set of generators for $\fa$.
Then, $b_{\fa}(s)\in\KK[s]$ is the monic polynomial of least degree, $b(s)$ such that 
$$
b(s_1+\cdots+s_\ell) f^{s_1}_1\cdots f^{s_\ell}_\ell\in \sum_{|\alpha|=1} D_{R|\KK}[s_1,\ldots,s_\ell] \cdot \prod_{\alpha_i} \binom{s_i}{-\alpha_i} 
f^{s_1+\alpha_1}_1\cdots f^{s_\ell+\alpha_\ell}_\ell,
$$
where $\alpha=(\alpha_1,\ldots,\alpha_\ell)\in\ZZ^\ell$, $|\alpha|=\alpha_1+\cdots+\alpha_\ell,$
$\binom{s_i}{m}= \frac{1}{m!} \prod^{m-1}_{j=0}(s_i-j).$
\end{theorem}

Musta\c{t}\u{a} \cite[Theorem 1.1]{Mustata2019} uses this characterization to show that $b_{\fa}(s)$ coincides with the reduced Bernstein-Sato polynomial of $f_1y_1+\cdots+f_\ell y_\ell \in A[y_1,\dots , y_\ell]$.

One may be tempted to consider a general element $\lambda_1f_1+\cdots+ \lambda_\ell f_\ell \in \fa$ whose log-resolution has the same numerical data as the log-resolution of the ideal $\fa$. 

\begin{example}
Let $\fa=(x^4,xy^2,y^3) \subseteq \CC[x,y]$ be a monomial ideal and consider a general element of the ideal $g= x^4+xy^2+y^3$. 
The roots of the Bernstein-Sato polynomial  $b_\fa(s)$ are
$$\left\{ -\frac{5}{8}, -\frac{2}{3}, -\frac{3}{4}, -\frac{7}{8}, -1, -\frac{9}{8}, -\frac{5}{4}, -\frac{4}{3}, -\frac{11}{8}, -\frac{3}{2}\right\}, $$
with $-1$ being a root with multiplicity $2$.
Meanwhile, the roots of the reduced Bernstein-Sato polynomial $\tilde{b}_g(s)$  are 
$$\left\{ -\frac{5}{8},  -\frac{7}{8}, -1, -\frac{9}{8},  -\frac{11}{8}\right\} $$
The exceptional part of the log-resolution divisor $F_\pi$ in both cases is of the form $3 E_1 + 4 E_2 + 8 E_3$. The roots of $\tilde{b}_g(s)$ are only contributed by the rupture divisor $E_3$ but this is not the case for $b_\fa(s)$.
\end{example}

\subsubsection{Monomial ideals}  Let $\mathfrak{a} \subseteq \CC[x_1,\dots, x_d]$ be a monomial ideal. Let  $P_\mathfrak{a} \subseteq \RR^d_{\geq 0}$ be the Newton polyhedron associated to  $\mathfrak{a}$ which is the convex hull of the semigroup
$$\Gamma_{\mathfrak{a}} =\{ a=(a_1,\dots, a_d )\in \NN^d \hskip 2mm | \hskip 2mm x_1^{a_1} \cdots x_d^{a_d} \in \mathfrak{a}\}.$$
For any face $Q$ of $P_\mathfrak{a}$ we define:
\begin{itemize}
\item[(i)] $M_Q$ the subsemigroup of $\ZZ^d$ generated  by $a-b$ with $a \in \Gamma_{\mathfrak{a}}$ and $b\in \Gamma_{\mathfrak{a}} \cap Q$.
\item[(ii)] $M'_Q:= c+ M_Q$ for $c\in \Gamma_{\mathfrak{a}} \cap Q$. 
\end{itemize}
 $M'_Q$ is a subset of $M_Q$ that is independent of the choice of $c$. For a face $Q$ of $P_\mathfrak{a} $ not contained in a coordinate hyperplane we consider a function $L_Q: \RR^d \rightarrow \RR$ with rational coefficients such that $L_Q=1$ on $Q$. Set 
 $$R_Q=\{ L_Q(a) \hskip 2mm | \hskip 2mm  a \in ((1,\dots , 1) + (M_Q \smallsetminus M'_Q) ) \cap V_Q \},$$
 where $V_Q$ is the linear subspace generated by $Q$. 
 
 Budur, Musta\c{t}\u{a}, and Saito \cite{BMS2006c} gave a closed formula for the roots of the Bernstein-Sato polynomial of $\mathfrak{a}$ in terms of these sets $R_Q$. 
 
 \begin{theorem}[\cite{BMS2006c}]
  Let $\mathfrak{a} \subseteq \CC[x_1,\dots, x_d]$ be a monomial ideal. Let $\rho_{\mathfrak{a}}$ be the set of roots of $b_{\mathfrak{a}}(-s)$. Then
$$\rho_{\mathfrak{a}} = \bigcup_Q R_Q$$  where the union is over the faces $Q$ of $P_\mathfrak{a}$ not contained in coordinate hyperplanes.
 \end{theorem}

\subsubsection{Determinantal varieties} 
The theory of equivariant $D$-modules has been successfully used in recent years to study local cohomology modules of determinantal varieties.
These techniques have also been used by L\H{o}rincz, Raicu, Walther, and Weyman \cite{LRWW} to determine the Bernstein-Sato polynomial of the ideal of maximal minors of a generic matrix.

\begin{theorem}[\cite{LRWW}]
Let $X=(x_{ij})$ be a generic $m\times n$ matrix with $m \geq  n$.  Let $\mathfrak{a}_n \subseteq A = \CC[x_{ij}]$ be the ideal generated by the $n\times n$ minors of $X$. The Bernstein-Sato polynomials of the ideal $\mathfrak{a}_n$ and the corresponding variety are
$$b_{\mathfrak{a}_n}(s) =\prod_{\ell=m-n+1}^m\left( s+\ell \right).$$
$$b_{Z(\mathfrak{a}_n)}(s) =\prod_{\ell=0}^{n-1}\left( s+\ell \right).$$
\end{theorem}

They also provided a formula for sub-maximal Pfaffians.

\begin{theorem}[\cite{LRWW}]
Let $X=(x_{ij})$ be a generic $(2n+1)\times(2n+1)$ skew-symmetric matrix, i.e $x_{ii} = 0, x_{ij} = -x_{ji}$.  Let $\mathfrak{b}_{2n} \subseteq A = \CC[x_{ij}]$ be the ideal generated by the $2n \times 2n$ Pfaffians of $X$. The Bernstein-Sato polynomials of the ideal $\mathfrak{b}_{2n} $ and the corresponding variety are
$$b_{\mathfrak{b}_{2n} }(s) =\prod_{\ell=0}^{n-1}\left( s+ 2\ell +3 \right).$$
$$b_{Z(\mathfrak{b}_{2n} )}(s) =\prod_{\ell=0}^{n-1}\left( s+2\ell \right).$$
\end{theorem}

\subsection{Bernstein-Sato ideals}

In this subsection we consider the theory of Bernstein-Sato ideals associated to a tuple of elements $F=f_1,\dots, f_\ell$ developed by Sabbah \cite{Sabbah_ideal}.

\begin{definition}
Let $\KK$ be a field of characteristic zero and $A$  a  regular $\KK$-algebra.
	A \emph{Bernstein-Sato functional equation} for a tuple $F=f_1,\dots, f_\ell$ of elements of $A$ is an equation of the form
	\[ \delta(s_1,\dots, s_\ell) f_1^{s_1+1} \cdots f_\ell^{s_\ell+1}= b(s_1,\dots, s_\ell) f_1^{s_1} \cdots f_\ell^{s_\ell} \]
	where $\delta(s_1,\dots, s_\ell)\in D_{A|\KK}[s_1,\dots, s_\ell]$  and $b(s_1,\dots, s_\ell)\in \KK[s_1,\dots, s_\ell]$. 
\end{definition}

All the polynomials $b(s_1,\dots, s_\ell)$ satisfying a Bernstein-Sato functional equation form an ideal $B_F \subseteq \KK[s_1,\dots, s_\ell]$ that we refer to as the \emph{Bernstein-Sato ideal}. 

\begin{remark}
More generally, given $a=(a_1,\dots, a_\ell)\in \ZZ^\ell_{\geq 0}$,  we may also consider the functional equations
\[ \delta(s_1,\dots, s_\ell) f_1^{s_1+a_1} \cdots f_\ell^{s_\ell+a_\ell}= b(s_1,\dots, s_\ell) f_1^{s_1} \cdots f_\ell^{s_\ell} \hskip 2mm \text{for all \ } s_i\in \NN, \]
leading to other Bernstein-Sato ideals $B^{a}_F \subseteq \KK[s_1,\dots, s_\ell]$.
\end{remark}

 As in the case $\ell=1$ we first wonder about the existence of such functional equations.

\begin{theorem}[\cite{Sabbah_ideal}]\label{exists-BSI}
Let $\KK$ be a field of characteristic zero, and let $A$ be either 
 $\KK[x_1,\dots,x_d]$ or $\CC\{x_1,\dots,x_d\}$.
Any nonzero tuple $F=f_1,\dots, f_\ell$ of elements of $A$ satisfies a nonzero Bernstein-Sato functional equation and thus $B_F \neq 0$.
\end{theorem}

Sabbah \cite{Sabbah_ideal} proved this result in the local analytic case $A=\CC\{x_1,\dots, x_d\}$. 
The proof in  the polynomial ring  case $A=\KK[x_1,\dots, x_d]$ is completely analogous to the one given in \Cref{subsec:ExistencePoly} for the case $\ell=1$. 


The Bernstein-Sato functional equation is an equality in 
$A_f[s_1,\dots, s_\ell] \boldsymbol{f^s}$ where  $f=f_1\cdots f_\ell$ and $\boldsymbol{f^s}:= \boldsymbol{f_1^{s_1}} \cdots  \boldsymbol{f_\ell^{s_\ell}}$.  We also have that the $D_{A|\KK}[s_1,\dots, s_\ell]$-submodule generated by  $\boldsymbol{f^s}$ has a presentation
\[ {D_{A|\KK}[s_1,\dots, s_\ell]\boldsymbol{f^s}} \cong  \frac{D_{A|\KK}[s_1,\dots, s_\ell]}{\mathrm{Ann}_{D[s_1,\dots, s_\ell]}(\boldsymbol{f^s}) }, \] 
%
%
and, given the fact that
\[ \frac{D_{A|\KK}[s_1,\dots, s_\ell]\boldsymbol{f^s}}{D_{A|\KK}[s_1,\dots, s_\ell] f \boldsymbol{f^s}} \cong  \frac{D_{A|\KK}[s_1,\dots, s_\ell]}{\mathrm{Ann}_{D[s_1,\dots, s_\ell]}(\boldsymbol{f^s}) + D_{A|\KK}[s_1,\dots, s_\ell] f}. \]
we get an analogue of \Cref{prop:othercharsBS} that reads as

\begin{proposition}\label{prop:othercharsBSI}
The Bernstein-Sato ideal of $F=f_1,\dots, f_\ell$ is $$B_F= \KK[s_1,\dots, s_\ell] \cap (\mathrm{Ann}_{D_{A|\KK}[s_1,\dots, s_\ell]}(\boldsymbol{f^s}) + D_{A|\KK}[s_1,\dots, s_\ell] f).$$
\end{proposition}

Some properties of Bernstein-Sato ideals are the natural extension of those satisfied by Bernstein-Sato polynomials. We start with the ones considered  in \Cref{First properties}. The analogue of  \Cref{s+1} is the following result.

\begin{lemma} [\cite{May, BriMay}]
Let  $F=f_1,\dots, f_\ell$ be a tuple where the $f_i$ are pairwise without common factors. Then 
$$B_F \subseteq \Big( (s_1+1)\cdots (s_\ell+1)\Big).$$  Equality is achieved if and only if $A/(f_1,\dots, f_\ell)$ is smooth.
\end{lemma}

 We summarize the relations between the Bernstein-Sato ideals when we change the ring $A$ in the following lemma. For the convenience of the reader we use temporally the same notation as in \Cref{First properties}.

\begin{lemma} [\cite{BriMai02}] We have:
\begin{enumerate}
\item $B_F^{\KK[x]} = \bigcap_{\mathfrak{m} \hskip 1mm {\rm max} \hskip 1mm {\rm ideal}}  B_F^{\KK[x]_{\mathfrak{m}}}$.
\item $B_F^{\KK[x]_\mathfrak{m}} = B_F^{\KK[[x]]}$, where $\mathfrak{m}$ is the homogeneous maximal ideal.
\item $B_F^{\CC\{x-p\}} = B_F^{\CC[[x-p]]}$,  where $p\in \CC^d$.
\item $ B_F^{\mathbb{L}[x]} = \mathbb{L} \otimes_{\KK} {B_F^{\KK[x]}} $ where $\mathbb{L}$ is a field containing $\KK$.
\end{enumerate}
\end{lemma}

The first rationality result for Bernstein-Sato ideals is given by Gyoja \cite{Gyo} and Sabbah \cite{Sabbah_ideal} where they proved the existence of  an element of $B_F$ which is a product of polynomials of degree one of the form $a_1 s_1+ \cdots + a_\ell s_\ell + a , $ with $a_i \in \QQ_{\geq 0}$ and $a \in \QQ_{> 0}$. This fact prompted Budur \cite{Bud15} to make the following:

\begin{conjecture}  \label{conj_Budur}
The Bernstein-Sato ideal of  a tuple $F=f_1,\dots, f_\ell $ of elements in $\CC\{x_1,\dots, x_d\}$ is generated by products of polynomials of degree one 
$$a_1 s_1+ \cdots + a_\ell s_\ell + a , $$ with $a_i \in \QQ_{\geq 0}$ and $a \in \QQ_{> 0}$
\end{conjecture}

Notice that this would imply that the irreducible components of the zero locus $Z(B_F)$ are linear. 
The best result so far towards this conjecture is the following. 

\begin{theorem}[{\cite{Mai16a}}]
Every irreducible component of $Z(B_F)$ of codimension $1$ is a hyperplane of type $a_1 s_1+ \cdots + a_\ell s_\ell + a , $ with $a_i \in \QQ_{\geq 0}$ and $a \in \QQ_{> 0}$. Every irreducible component of $Z(B_F)$ of codimension $> 1$ can be translated by an element of $\ZZ^\ell$ inside a component of codimension $1$.
\end{theorem}

Recall that the work of Kashiwara and Malgrange relates the roots of the Bernstein-Sato polynomials to the eigenvalues of the monodromy
and these eigenvalues are roots of unity by the monodromy theorem.  An extension  to the case of Bernstein-Sato ideals  of Kashiwara and Malgrange result has been given recently by Budur \cite{Bud15}  and Budur, van der Veer, Wu, and Zhou \cite{BvVWZ}.
There is also an extension of the Monodromy theorem in this setting given by Budur and Wang \cite{BudWang}  and  Budur, Liu, Saumell, and Wang \cite{BLSW}.  Unfortunately these results are not enough to settle \Cref{conj_Budur}.

The main difference with the classical case is that Bernstein-Sato ideals are not necessarily principally generated. 
Brian\c{c}on and Maynadier \cite{BriMay} gave a theoretical proof of this fact for the following example.  The  explicit computation was given by Balhoul and Oaku \cite{BahloulOaku}.

\begin{example} [\cite{BriMay, BahloulOaku}]
Let $F=z,x^4+y^4+zx^2y^2$ be a pair of elements in $\CC\{x,y,z\}$.
The local Bernstein-Sato ideal is nonprincipal
\begin{align*}
B_{F}^{\CC\{x\}}= \Big( & (s_1 + 1)(s_2 + 1)^2(2s_2 + 1)(4s_2 + 3)(4s_2 + 5)(s_1 + 2), \\
& (s_1 + 1)(s_2 + 1)^2(2s_2 + 1)(4s_2 + 3)(4s_2 + 5)(2s_2 + 3) \Big). 
\end{align*}
However, when we consider $F$ in $\CC[x,y,z]$ the global Bernstein-Sato ideal is 
$$B_F^{\CC[x]} = \Big( (s_1 + 1)(s_2 + 1)^2(2s_2 + 1)(2s_2 + 3)(4s_2 + 3)(4s_2 + 5)\Big).$$
\end{example}

The following example is also given by Balhoul and Oaku.

\begin{example} [\cite{ BahloulOaku}]
Let $F=z,x^5+y^5+zx^2y^3 $ be a pair of elements in $\CC[x,y,z]$. Then the local and the global Bernstein-Sato ideals coincide and are nonprincipal. Specifically, $B_F$ is generated by $(s_1 + 1)(s_2 + 1)^2(5s_ 2 + 2)(5s_2 + 3)(5s_2 + 4)(5s_2 + 6)(s_1 + 2)(s_1 + 3)(s_1 + 4)(s_1 + 5)$,  $(s_1 + 1)(s_2 + 1)^2(5s_2 + 2)(5s_2 + 3)(5s_2 + 4)(5s_2 + 6)(5s_2 + 7)(s_1 + 2)$, and 
 $(s_1 + 1)(s_2 + 1)^2(5s_2 + 2)(5s_2 + 3)(5s_2 + 4)(5s_2 + 6)(5s_2 + 7)(5s_2 + 8)$.
\end{example}

There are interesting examples worked out in several computational articles by  Balhoul \cite{Bahloul}, Balhoul and Oaku  \cite{BahloulOaku},  Castro-Jim\'enez and Ucha-Enr\'iquez  \cite{UchaCastro}, Andres, Levandovskyy, and Mart\'in-Morales \cite{AndresLevMM}. However, we cannot find many closed formulas for families of examples.  Maynadier \cite{May} studied the case of quasi-homogeneous isolated complete intersection singularities and we highlight the case of hyperplane arrangements.

\subsubsection {Hyperplane arrangements:} Let $f \in \CC[x_1,\dots ,  x_d]$ be a reduced polynomial defining an arrangement of hyperplanes. 
The most natural tuple $F=f_1, \cdots, f_\ell$ associated to $f$ is the one given by its  degree one components. The following result is an extension of Walther's work to this setting.  It  was first obtained by Maisonobe \cite{Mai16b} for the case $\ell=d+1$ and further extended by Bath \cite{Bath} for $\ell\geq d+1$. We point out that Bath also provides a formula for other tuples  associated to different decompositions of the arrangement $f$. 

\begin{theorem}[\cite{Mai16b, Bath}]
Let $f=f_1\cdots f_\ell \in \CC[x_1,\dots ,  x_d]$, with $\ell \geq d +1$, be the decomposition of  a generic central hyperplane arrangement as a product of linear forms.
The Bernstein-Sato ideal of  the tuple $F=f_1,\dots , f_\ell$ is
$${B}_F= \left( \prod_{i=1}^\ell (s_i+1) \prod_{j=0}^{2\ell-d-2} \left(  s_1+\cdots +s_\ell + j + d \right) \right).$$
\end{theorem}

\subsection{Relative versions}

In this section we discuss a more general version of the Bernstein-Sato polynomials in which the functional equation includes an element of a $D$-module $M$ \cite{Sabbah87, Mebkhout_book}. As in the classical case, we consider this functional equation as an equality in a given module that we define next.

%

\begin{definition}
Let $A$ be a differentiably admissible $\KK$-algebra, and $M$ a left $D_{A|\KK}$-module. For $f\in A\smallsetminus\{0\}$, we define 
the left $D_{A_f|\KK}[s]$-module $M_f[s] \boldsymbol{f^s}$ as follows:
	\begin{enumerate}
	\item As an $A_f[s]$-module, $M_f[s] \boldsymbol{f^s}$ is isomorphic to  $M_f[s]$.
	\item Each partial derivative $\partial\in \Der_{A|\KK}$ acts by the rule
	\[ \partial (a(s)v \boldsymbol{f^s}) = \left(a(s)\partial (v) +  \frac{s a(s) \partial(f)}{f}\right) \boldsymbol{f^s} \]
	for $a(s)\in A_f[s]$.
	\end{enumerate}
\end{definition} 

Alternative descriptions can be given analogously to Subsection~\ref{Subsec:modules}, but we do not need them here.

\begin{theorem}[{\cite[Theorem 3.1.1]{MNM}, \cite{Sabbah87}}]
Let $A$ be a differentiably admissible $\KK$-algebra, $M$ a left $D_{A|\KK}$-module in the Bernstein class, and $f\in A\smallsetminus\{0\}$.
For any element $v\in M$ there exists $\delta(s)\in D_{A|\KK}[s]$ and $b(s)\in\KK[s]\smallsetminus\{0\}$ such that
$$
\delta(s) v f  \boldsymbol{f^s}=b(s)v \boldsymbol{f^s}.
$$
\end{theorem}

There are not many explicit examples of Bernstein-Sato polynomials in this generality that we may find in the literature. Torrelli \cite{Torrelli02, Torrelli03} has some results in the case that $M$ is the local cohomology module of a complete intersection or a hypersurface with isolated singularities. Reichelt, Sevenheck, and Walther \cite{RSW18} studied the case of hypergeometric systems.

In the case of $M$ being the ring itself, we find the Bernstein-Sato polynomial of $f$ relative to an element $h\in A$. Of course, when $h=1$ we recover the classical version.

\begin{corollary}
Let $A$ be a differentiably admissible $\KK$-algebra and $f\in A\smallsetminus\{0\}$.
For any element $h\in A$ there exists $\delta(s)\in D_{A|\KK}[s]$ and $b(s)\in\KK[s]\smallsetminus\{0\}$ such that
$$
\delta(s) h f  \boldsymbol{f^s}=b(s)h \boldsymbol{f^s}.
$$
\end{corollary}

\begin{definition}
Let $A$ be a differentiably admissible $\KK$-algebra, $M$ a left $D_{A|\KK}$-module in the Bernstein class, $f\in A\smallsetminus\{0\}$, and $v\in M$. We define the relative Bernstein-Sato polynomial $b_{f,v}(s)$ to be the monic polynomial of minimal degree for which there is a nonzero functional equation
\[\delta(s) v f  \boldsymbol{f^s}=b_{f,v}(s)v \boldsymbol{f^s}. \]
\end{definition}

A basic example shows that $s=-1$ need not always be a root of the relative Bernstein-Sato polynomial $b_{f,g}(s)$.

\begin{example}
	Let $A=\CC[x]$, and take $f=g=x$. We have a functional equation
	\[ \partial_x x^{s+1} x = (s+2) x^s x \ \text{for all} \ s,\]
	so $s=-1$ is not a root of $b_{x,x}(s)$. It follows from the next proposition that $b_{x,x}(s)=s+2$.
\end{example}

We record a basic property of relative Bernstein-Sato polynomials that may be considered as an analogue to Lemma~\ref{s+1}.

\begin{lemma}\label{s+n}
	Let $A$ be a differentially admissible $\KK$-algebra, and $f,g\in A\smallsetminus \{0\}$. If $g\in (f^{n-1})\smallsetminus(f^{n})$, then $s=-n$ is a root of $b_{f,g}(s)$.
\end{lemma}
\begin{proof} Evaluating the functional equation at $s=-n$, we have
	\[ \delta(-n) f f^{-n} g = b(-n) f^{-n} g.\]
	Since $g/f^{n-1}\in R$, and $g/f^n\notin R$, we must have $b(-n)=0$.
	\end{proof}

We make another related observation.

\begin{lemma}\label{shift}
	Let $A$ be a differentially admissible $\KK$-algebra, and $f,g\in A\smallsetminus \{0\}$. Then $b_{f,f^ng}(s) = b_{f,g}(s+n)$ for all $n$.
\end{lemma}
\begin{proof}
	Given a functional equation
	\[ \delta(s)   g f\boldsymbol{f^s}= b_{f,g}(s) g \boldsymbol{f^s}, \]
	shifting by $n$ yields
		\[ \delta(s+n)   g f^n f\boldsymbol{f^s}= b_{f,g}(s+n) g f^n \boldsymbol{f^s}, \]
		so $b_{f,g}(s+n) \ | \ b_{f,f^ng}(s)$. Similarly,
		given a functional equation
		\[ \delta'(s)   g f^n f\boldsymbol{f^s}= b_{f,f^ng}(s) g f^n \boldsymbol{f^s}, \]
		we also have
			\[ \delta'(s-n)   g f\boldsymbol{f^s}= b_{f,f^ng}(s-n) g \boldsymbol{f^s}, \]
			from which the equality follows.
	\end{proof}

This notion of relative Bernstein-Sato polynomials has been extended to the case of nonprincipal ideals by Budur, Musta\c{t}\v{a} and Saito \cite{BMS2006a} following the approach given in Subsection~\ref{SubGeneralIdeals2}. 

\begin{theorem}[\cite{BMS2006a}]
Let $\KK$ a field of characteristic zero, $A$ be a regular  finitely generated $\KK$-algebra, and $\fa\subseteq A$ be a nonzero ideal.
Let $F=f_1,\ldots,f_\ell$ be a set of generators for $\fa$ and consider an element $h\in A$.
Then, $b_{\fa,h}(s)\in\KK[s]$ is the monic polynomial of least degree, $b(s)$ such that 
$$
b(s_1+\cdots+s_\ell) h f^{s_1}_1\cdots f^{s_\ell}_\ell\in \sum_{|\alpha|=1} D_{R|\KK}[s_1,\ldots,s_\ell] \cdot \prod_{\alpha_i} \binom{s_i}{-\alpha_i} 
h f^{s_1+\alpha_1}_1\cdots f^{s_\ell+\alpha_\ell}_\ell,
$$
where $\alpha=(\alpha_1,\ldots,\alpha_\ell)\in\ZZ^\ell$, $|\alpha|=\alpha_1+\cdots+\alpha_\ell,$
$\binom{s_i}{m}= \frac{1}{m!} \prod^{m-1}_{j=0}(s_i-j).$
\end{theorem}

\subsection{$V$-filtrations}

In this subsection, we give a quick overview of the $V$-filtration and its relationship with the relative versions of Bernstein-Sato polynomials. For further details regarding $V$-filtrations we refer to Budur's survey on this subject \cite{SurveyBudur}.

\begin{definition}\label{DefVfil}
Suppose that $\KK$ has  characteristic zero. Let $A$ be a regular Noetherian    $\KK$-algebra. Let $T=t_1,\ldots,t_\ell$ be a sequence of variables, and let $A[t_1,\ldots,t_\ell]$ be a polynomial ring over $A$.
The $V$-filtration along the ideal $(T)$ on the ring of differential operators  $D_{A[T]|\KK}$ is the filtration indexed by integers $i\in \mathbb{Z}$ defined by
   \begin{equation*}
   V^{i}_{(T)} {D_{A[T]|\KK}} = \{ \delta \in D_{A[T]|\KK} :  \delta \act (T)^j \subseteq  (T)^{j+i} \ \text{for all} \ j \in \ZZ \},
\end{equation*}
   where $(T)^j = A[T]$ for  $j \leq 0$.
\end{definition}

\begin{remark}
We consider  $D_{A[T]|\KK}$ as a graded ring where $\deg(t_i)=1$ and $\deg(\partial_{t_i})=-1$.
Then,
   \begin{equation*}
      V^{i}_{(T)} {D_{A[T]|\KK}}
      = \bigoplus_{\substack{\pt{a},\pt{b} \in \NN^\ell \\ |\pt{a}|-|\pt{b}|\,\geq \, i}} D_{A|\KK} \cdot  t_1^{a_1}\cdots t_\ell^{a_\ell} \pd{t_1}^{b_1} \cdots \pd{t_\ell}^{b_\ell}.
   \end{equation*}
\end{remark}

The $V$-filtration along the ideal $(T)$ on a $D_{A[T]|\KK}$-module $M$ is defined as follows.

\begin{definition} \label{VfilM}
Suppose that $\KK$ has  characteristic zero. Let $A$ be a regular Noetherian    $\KK$-algebra. Let $T=t_1,\ldots,t_\ell$ be a sequence of variables, and let $A[t_1,\ldots,t_\ell]$ be a polynomial ring over $A$.
   Let $M$ be a $D_{A[T]|\KK}$-module.
   A \emph{$V$-filtration on $M$ along the ideal $(T)=(t_1,\ldots,t_\ell)$} is a decreasing filtration $\{V^{\alpha}_{(T)} M\}_\alpha$ on $M$, indexed by $\alpha\in \QQ$, satisfying the following conditions.
   \begin{enumerate}
      \item For all $\alpha\in \QQ$, $V^{\alpha}_{(T)} M$ is a Noetherian  $V^{0}_{(T)} {D_{A[T]|\KK}}$-submodule of $M$.
      \item The union of the $V^{\alpha}_{(T)} M$, over all $\alpha \in \QQ$, is $M$.
      \item  $V^{\alpha}_{(T)} M = \bigcap_{\gamma < \alpha} V^{\gamma}_{(T)} M$ for all $\alpha$, and the set $J$  consisting of all $\alpha\in \QQ$ for which $V^{\alpha}_{(T)} M \neq \bigcup_{\gamma > \alpha} V^{\gamma}_{(T)} M$ is discrete.
      \item For all $\alpha \in \QQ$ and all $1 \leq i \leq \ell$,
      \[ t_i \act V^\alpha_{(T)} M \subseteq V^{\alpha+1}_{(T)} M \, \text{ and } \,\pd{t_i} \act V^\alpha_{(T)} M \subseteq V^{\alpha-1}_{(T)} M, \] i.e., the filtration is compatible with the $V$-filtration on $D_{A[T]|\KK}$.
      \item For all $\alpha\gg 0$,  $\sum^\ell_{i=1}\left( t_i \act V^\alpha_{(T)} M\right) = V^{\alpha+1}_{(T)} M$.
      \item For all $\alpha \in \QQ$,
      \[\sum_{i=1}^\ell \pd{t_i} t_i - \alpha\]
      acts nilpotently on $V^\alpha_{(T)} M / (\bigcup_{\gamma > \alpha} V^{\gamma}_{(T)} M)$.
   \end{enumerate}
\end{definition}

\begin{proposition}[{\cite{SurveyBudur}}]
Suppose that $\KK$ has  characteristic zero. Let $A$ be a regular  Noetherian  $\KK$-algebra. Let $T=t_1,\ldots,t_\ell$ be a sequence of variables, and let $A[t_1,\ldots,t_\ell]$ be a polynomial ring over $A$.
Let $M$ be a finitely generated $D_{A[T]|\KK}$-module. If a $V$-filtration on $M$ along $(T)$ exists, then it is unique.
\end{proposition}
		
We now define the $V$-filtration on a $D_{A |\KK}$-module $M$ along $F= f_1,\ldots, f_\ell \in A$, where $M$ is a  $D_{R|\KK}$-module.
For this, we need the  direct image of $M$ under the graph embedding $i_{\seq{f}}$. We recall that this is the local cohomology module $H^\ell_{(T-F)} (M[T])$, where $(T-F)= (t_1-f_1,\ldots,t_\ell-f_\ell)$.

\begin{definition}\label{DefVfilIdeal}
Suppose that $\KK$ has  characteristic zero. 
Let $A$ be a regular Noetherian $\KK$-algebra. 
Given indeterminates $T = t_1,\ldots,t_\ell$, and $F=f_1,\ldots,f_\ell\in A$, consider the ideal 
 $(T-F)$ of the polynomial ring $A[T]$ generated by $t_1-f_1,\ldots, t_\ell-f_\ell$. 
For a $D_{A|\KK}$-module  $M$, let $M'$ denote the $D_{A[T]|\KK}$-module $H^{\ell}_{(T-F)}(M[T])$, and 
 identify $M$ with the isomorphic module $0 :_{M'} (T-F) \subseteq M'$. Suppose that $M'$ admits a $V$-filtration along $(T)$ over $A[T]$.
Then the \emph{$V$-filtration on $M$ along $(T-F)$} is defined, for $\alpha \in \QQ$, as
 \[V^{\alpha}_{(F)} M\coloneqq V^{\alpha}_{(T)} M' \cap M = ( 0 :_{V^{\alpha}_{(T)} M'} (T-F)).\]
\end{definition}

We point out that  $V$-filtration over $A$ along $F$ only depends on the ideal $\fa=(F)$ and not on the generators chosen.

We now give a result that guarantees the existence of $V$-filtrations. We point out that we have not defined regular or quasi-unipotent $D_{A|\KK}$-modules. 
We omit these definitions, but we mention that all principal localizations $A_f$ and all local cohomology modules $H^i_{\fa}(A)$ of the ring $A$ satisfy these properties.

\begin{theorem}[{\cite{KashiwaraVfil,MalgrangeVfil}}]
Suppose that $\KK$ has  characteristic zero. 
Let $A=\KK[x_1,\ldots, x_d]$ be a polynomial ring and $M$ be a quasi-unipotent regular holonomic left $D_{A|\KK}$-module.
Then, $M$ has a $V$-filtration  along $F= f_1,\ldots, f_\ell \in A$.
\end{theorem}

Once we ensure the existence of $V$-filtrations we have the following characterization in terms of relative Bernstein-Sato polynomials.

\begin{theorem}[{\cite{Sabbah87,BMS2006a}}]
Suppose that $\KK$ has  characteristic zero. 
Let $A=\KK[x_1,\ldots, x_d]$ be a polynomial ring and $M$ be  a quasi-unipotent regular holonomic left $D_{A|\KK}$-module.
Then,
$$
V^\alpha_{(F)} M=\{v\in M \; | \: \alpha\leq c \hbox{ if }b_{(F),v}(-c)=0\}.
$$
\end{theorem}

\section{Bernstein-Sato theory in prime characteristic}\label{sec:positivechar}

We now discuss Bernstein-Sato theory in positive characteristic. Throughout this section, $\KK$ is a perfect field of characteristic $p>0$, and $A=\KK[x_1,\dots,x_d]$ is a polynomial ring. The main purpose of this section is to discuss the theory developed by Musta\c{t}\u{a} \cite{MustataBSprime}, Bitoun \cite{BitounBSpos}, and Quinlan-Gallego \cite{EamonBSpos}.

Before we do so, as motivation, we briefly discuss the notion of the Bernstein-Sato functional equation in positive characteristic.
Note that for $b(s)\in \KK[s]$, we have $b(s) f^s= c(s) f^s$ for all $s\in \NN$ if and only if $b$ and $c$ determine the same function from $\FF_p$ to $\KK$. This gives a recipe for many unenlightening functional equations: we can take $b(s)$ to be a function identically zero on $\FF_p$, e.g., $s^p-s$, and $\delta(s)$ to be some operator that annihilates every power of $f$, e.g., the zero operator. For this reason, the notion of Bernstein-Sato polynomial in characteristic zero is not as well-suited for consideration in positive characteristic.

Instead, we return to an alternative characterization of the Bernstein-Sato polynomial discussed in Subsection~\ref{Subsec:modules}. As a consequence of Proposition~\ref{prop:othercharsBS}, for polynomial rings in characteristic zero, we can characterize the roots of the Bernstein-Sato polynomial of $f$ as the eigenvalues of the action of $-\partial_t t$ on $[\frac{1}{f-t}]$ in \[\frac{D_{A|\KK}[-\partial_t t] \cdot [\frac{1}{f-t}]}{D_{A|\KK}[-\partial_t t] f \cdot  [\frac{1}{f-t}]}.\]
In characteristic $p>0$, we consider the eigenvalues of a sequence of operators that are closely related to $-\partial_t t$.

\begin{definition} Consider $D_{A[t]|\KK}$ as a graded ring, with grading induced by giving each $x_i$ degree zero, and $t$ degree $1$. We set $[D_{A[t]|\KK}]_0$ to be the subring of homogeneous elements of degree zero, and $[D_{A[t]|\KK}]_{\geq 0}$ to be the subring spanned by elements of nonnegative degree.
\end{definition}

We note that $[D_{A[t]|\KK}]_{\geq 0}$ is also characterized by the $V$-filtration as $V^0_{(t)} D_{A[t]|\KK}$.

\begin{lemma}
	$[D_{A[t]|\KK}]_0=D_{A|\KK} [s_0, s_1,\dots]$, where $s_e=-\frac{\partial_t^{p^e}}{p^e !} t^{p^e}$. In this ring, the operators $s_i$ commute with one another and elements of $D_{A|\KK}$, and $s_i^{p}=s_i$ for each~$i$.
\end{lemma}
\begin{proof}
	We omit the proof that these elements generate. It is clear that each $s_i$ commutes with elements of $D_{A|\KK}$. For an element $f(t)=\sum_j a_j t^j \in A[t]$, with $a_j\in A$, using Lucas' Lemma, we compute
	\[ s_i f(t) = \sum_j -\binom{j + p^i}{p^i} a_j t^j = \sum_j -([j]_{i}+1) a_j t^j , \]
	where $[j]_i$ is the $i$th digit in the base $p$ expansion of $j$; our convention that the unit digit is the $0$th digit. The other claims follow from this computation.
\end{proof}

We can interpret the computation in the previous lemma as saying that the $\alpha_i$-eigenspace of $s_i$ on $A[t]$ is spanned by the homogeneous elements such that the $i$th base $p$ digit of the degree is $\alpha_i-1$. By way of terminology, we say that the $(\alpha_0,\alpha_1,\alpha_2,\dots)$-multieigenspace of $(s_0,s_1,s_2,\dots)$ is the intersection of the $\alpha_i$-eigenspace of $s_i$ for all $i$. Then, the $(\alpha_0,\alpha_1,\alpha_2,\dots)$-multieigenspace of $(s_0,s_1,s_2,\dots)$ on $A[t]$ is the collection of homogeneous elements of degree $\sum_i (\alpha_i-1) p^i$ for a tuple with $\alpha_i=0$ for $i\gg 0$. This motivates the idea that a ``Bernstein-Sato root'' in positive characteristic should be determined by a multieigenvalue of the action of $(s_0,s_1,s_2,\dots)$
 on $[\frac{1}{f-t}]$ in \[\frac{[D_{A[t]|\KK}]_{\geq 0} \cdot [\frac{1}{f-t}]}{[D_{A[t]|\KK}]_{\geq 0} f \cdot  [\frac{1}{f-t}]}.\]
 
Based on this motivation, we give two closely related notions of Bernstein-Sato roots appearing in the literature.

\subsection{Bernstein-Sato roots: $p$-adic version}\label{subsec:p-adic}

The first definition of Bernstein-Sato roots that we present follows the treatment of Bitoun \cite{BitounBSpos}. To each element $\mathbf{\alpha}=(\alpha_0,\alpha_1,\alpha_2,\dots)\in \FF_p^{\NN}$ we associate the $p$-adic integer $I(\alpha)=\alpha_0 + p\alpha_1 + p^2 \alpha_2 + \cdots$.

\begin{theorem}[\cite{BitounBSpos}]\label{Bitountheorem} For any $f\in A$, the module
	\[\frac{[D_{A[t]|\KK}]_{\geq 0} \cdot [\frac{1}{f-t}]}{[D_{A[t]|\KK}]_{\geq 0} f \cdot  [\frac{1}{f-t}]}\]
	decomposes as a finite direct sum of multieigenspaces of $(s_0,s_1,s_2,\dots)$. The image of each multieigenvalue under $I$ is negative, rational, and at least negative one. Moreover, the map $I$ induces a bijection between multieigenvalues and the set of negatives of the $F$-jumping numbers in the interval $(0,1]$ with denominator not divisible by $p$.
	\end{theorem}

In this context, we consider the image of the multieigenvaues under the map $I$ as the set of \textit{Bernstein-Sato roots} of $f$.
Moreover, Bitoun constructs a notion of a Bernstein-Sato polynomial as an ideal in a certain ring; however, this yields equivalent information to the set of Bernstein-Sato roots just defined.

\begin{example}[{\cite{BitounBSpos}}]
	\begin{enumerate}
		\item Let $f=x_1^2 + \cdots + x_n^2$, with $n\geq 2$, and $p>2$. Then the set of Bernstein-Sato roots of $f$ is $\{-1\}$. Contrast this with the situation in characteristic zero, where $-n/2$ is also a root.
		\item Let $f=x^2+y^3$, and $p>3$. If $p\equiv 1 \ \mathrm{mod} \, 3$, then the set of Bernstein-Sato roots is $\{ -1,-5/6\}$, and if $p\equiv 2 \ \mathrm{mod} \, 3$, then the set of Bernstein-Sato roots is $\{ -1\}$.
	\end{enumerate}
\end{example}

\subsection{Bernstein-Sato roots: base $p$ expansion version}

The second definition of Bernstein-Sato roots that we present is historically the first, following the treatment of Musta\c{t}\u{a}. To each element $\mathbf{\alpha}=(\alpha_0,\alpha_1,\alpha_2,\dots,\alpha_e)\in \FF_p^{e+1}$ we associate the real number $E(\mathbf{\alpha})=\frac{1}{p^{e+1}} \alpha_0 + \frac{1}{p^{e}} \alpha_1 + \cdots + \frac{1}{p} \alpha_e$.

\begin{theorem}[\cite{MustataBSprime}]\label{Mirceatheorem} For $\mathbf{\alpha}\in \FF_p^{e+1}$, we have that $\mathbf{\alpha}$ is a multieigenvalue of 
	\[\frac{[D^{(e)}_{A[t]|\KK}]_{\geq 0} \cdot [\frac{1}{f-t}]}{[D^{(e)}_{A[t]|\KK}]_{\geq 0} f \cdot  [\frac{1}{f-t}]}\]
if and only if there is an $F$-jumping number of $f$ contained in the interval $(E(\mathbf{\alpha}),E(\mathbf{\alpha})+1/p^{e+1} ]$.
\end{theorem}

For each level $e$, one then obtains a set of \textit{Bernstein-Sato roots}, given as the image of the multieigenvalues under the map $E$.

Relative versions of the above result, for an element in a unit $F$-module, were considered by Stadnik \cite{Stadnik} and Blickle and St\"abler \cite{BliSta}.

\subsection{Nonprincipal case}

Both of the approaches above were extended to the nonprincipal case by Quinlan-Gallego \cite{EamonBSpos}. To state these generalizations, for an $n$-generated ideal $\fa=(f_1,\dots,f_n)$, we consider the following.

\begin{definition}
	Consider $D_{A[t_1,\dots,t_n]|\KK}$ as a graded ring, with grading induced by giving each $x_i$ degree zero, and each $t_i$ degree one. We set $[D_{A[t_1,\dots,t_n]|\KK}]_{\geq 0}$ to be the subring spanned by homogeneous elements of nonnegative degree. We also set
	\[ s_e = -\sum_{a_1+\cdots+a_n=p^e} \frac{\partial_1^{a_1}}{a_1!} \cdots \frac{\partial_n^{a_n}}{a_n!} t_1^{a_1} \cdots t_n^{a_n}. \]
\end{definition}
	
	 Theorems~\ref{Bitountheorem} and ~\ref{Mirceatheorem} have analogues in this setting; we state the former here and refer the reader to \cite{EamonBSpos} for the latter.
	
	\begin{theorem}
		Let $\fa=(f_1,\dots,f_n)$, and let 
		\[\eta= \left[\frac{1}{(f_1-t_1)\cdots(f_n-t_n)}\right] \in H^n_{(f_1-t_1,\dots,f_n-t_n)}(A[t_1,\dots,t_n]).\]
		Then, the module 	\[\frac{[D_{A[t_1,\dots,t_n]|\KK}]_{\geq 0} \cdot \eta}{[D_{A[t_1,\dots,t_n]|\KK}]_{\geq 0} \fa \cdot  \eta}\]
		decomposes as a finite direct sum of multieigenspaces of $(s_0,s_1,s_2,\dots)$. The image of each multieigenvalue under the map $I$ from Subsection~\ref{subsec:p-adic} is rational and negative. Moreover, there is an equality of cosets in $\QQ/\ZZ$:
		\begin{align*} &\{ I(\mathbf{\alpha}) \ | \ \alpha \text{ is a multieigenvalue of }(s_0,s_1,s_2,\dots)\} + \ZZ =\\  \{ &\text{negatives of $F$-jumping numbers of } \fa \text{ with denominator not a multiple of }p\} + \ZZ. \end{align*}
	\end{theorem}

In this setting, we consider the image of the set of multieigenvalues under the map $I$ as the set of \textit{Bernstein-Sato roots} of $\fa$.

\begin{example}[{\cite{EamonBSmon}}] Let $\fa = (x^2, y^3)$. Then, for $p=2$, the set of Bernstein-Sato roots is $\{-4/3 ,-5/3, -2\}$. For $p=3$, the set of roots is $\{-3/2,-2\}$. For $p\gg0$, by \cite[Theorem~3.1]{EamonBSmon}, the set of roots is $\{-5/6,-7/6,-4/3,-3/2,-5/3,-2\}$.
\end{example} 

The connection between Bernstein-Sato roots and $F$-jumping numbers largely stems from the following proposition, and the fact that $\cC^e_A \fa=\cC^e_A\fb$ if and only $D^{(e)}_{A} \fa=D^{(e)}_{A} \fb$.

\begin{proposition}[{\cite[Section~6]{MustataBSprime},\cite[Theorem~3.11]{EamonBSpos}}] The multieigenspace  cooresponding to $(\alpha_0,\alpha_1,\alpha_2,\dots,\alpha_{e-1})$ of $(s_0,s_1,s_2,\dots,s_{e-1})$ acting on 
	\[\frac{[D^{(e)}_{A[t_1,\dots,t_n]|\KK}]_{\geq 0} \cdot \eta}{[D^{(e)}_{A[t_1,\dots,t_n]|\KK}]_{\geq 0} \fa \cdot  \eta}\] decomposes as the direct sum of the modules
	\[ \frac{ D^{(e)}_A \cdot \fa^{I(\alpha) + s p^e}}{ D^{(e)}_A \cdot \fa^{I(\alpha) + s p^e+1}} \qquad s=0,1,\dots,n-1. \]
\end{proposition}

\section{An extension to singular rings} \label{sec:Singular}

We now consider the notion of Bernstein-Sato polynomial in rings of characteristic zero that may be singular. Throughout this section, $\KK$ is a field of characteristic zero, and $R$ is a $\KK$-algebra.

 As in Section~\ref{sec:def-first-prop}, the definition is as follows:

\begin{definition}
	A \emph{Bernstein-Sato functional equation} for an element $f$ in $R$ is an equation of the form
	\[ \delta(s) f^{s+1} = b(s) f^s \qquad \text{for all \ } s\in \NN,\]
	where $\delta(s)\in D_{R|\KK}[s]$ is a polynomial differential operator, and $b(s)\in \KK[s]$ is a polynomial. We say that such a functional equation is nonzero if $b(s)$ is nonzero; this implies that $\delta(s)$ is nonzero as well.
	
	If there exists a nonzero functional equation for $f$, we say that $f$ admits a Bernstein-Sato polynomial, and the Bernstein-Sato polynomial of $f$ is the minimal monic generator of the ideal
		\[ \{ b(s) \in \KK[s] \ | \ \exists \delta(s)\in D_{R|\KK}[s] \ \text{such that} \ \delta(s) f^{s+1} = b(s) f^{s} \ \text{for all} \ s\in \NN \} \subseteq \KK[s].\]
		We denote this as $b_f(s)$, or as $b_f^R(s)$ if we need to keep track of the ring in which we are considering $f$ as an element.
		
		If every element of $R$ admits a Bernstein-Sato polynomial, we say that $R$ has Bernstein-Sato polynomials. 
\end{definition} 

The set specified above is an ideal of $\KK[s]$ for the same reason as in Section~\ref{sec:def-first-prop}.

The proof of existence of Bernstein-Sato polynomials uses the hypothesis that $R$ is regular crucially in multiple steps; thus, a priori Bernstein-Sato polynomials may or may not exist in singular rings. Before we consider examples, we want to consider the functional equation as a formal equality in a $D$-module.

\begin{theorem}[{\cite{Vfilt}}] There exists a unique (up to isomorphism) $D_{R_f|\KK}[s]$-module, $R_f[s]\boldsymbol{f^s}$, that is a free as an $R_f[s]$-module, and that is equipped with maps\\ ${\theta_n:R_f[s]\boldsymbol{f^s} \to R_f}$, such that $\pi_n(\delta(s))\cdot \theta_n(a(s)\boldsymbol{f^s})=\theta_n(\delta(s) \cdot a(s) \boldsymbol{f^s})$ for all $n\in \NN$. An element $a(s)\boldsymbol{f^s}$ is zero in $R_f[s]\boldsymbol{f^s}$ if and only if $\theta_n(a(s)\boldsymbol{f^s})=0$ for infinitely many (if and only if all) $n\in \NN$.
\end{theorem}

\begin{remark}\label{rem:equivs-gen}
	From this theorem, we see that the following are equivalent, as in the regular case:
	\begin{enumerate}
		\item $\delta(s) f\boldsymbol{f^s} = b(s) \boldsymbol{f^s}$ in $R_f[s]\boldsymbol{f^s}$;
		\item $\delta(s) f^{s+1} = b(s) f^s$ for all $s\in \NN$;
		\item $\delta(s+t) f^{t+1}\boldsymbol{f^s} = b(s) f^t \boldsymbol{f^s}$ in $R_f[s]\boldsymbol{f^s}$ for some/all $t\in \ZZ$.
	\end{enumerate}
We note also that Proposition~\ref{prop:othercharsBS} holds in this setting, by the same argument.
\end{remark}

\subsection{Nonexistence of Bernstein-Sato polynomials}

In this subsection, we give some examples of rings with elements that do not admit Bernstein-Sato polynomials. This is based on a necessary condition on the roots that utilizes the following definition.

\begin{definition}\label{def:D-ideal}
	A $D$-ideal of $R$ is an ideal $\fa\subseteq R$ such that $D_{R|\KK}(\fa) =\fa$.
\end{definition}

As $R\subseteq D_{R|\KK}$, we always have $\fa\subseteq D_{R|\KK}(\fa)$, so the nontrivial condition in the definition above is $D_{R|\KK}(I) \subseteq I$. We always have that $0$ and $R$ are $D$-ideals. Sums, intersections, and minimal primary components of $D$-ideals (when $R$ is Noetherian) are also $D$-ideals  \cite[Proposition~4.1]{Travesmono}. When $R$ is a polynomial ring, the only $D$-ideals are  $0$ and $R$; in other rings, there may be more. We make a simple observation.

\begin{lemma}\label{lem:rootDideal}
	Let $f\in R$, and let $\fa\subseteq R$ be a $D$-ideal. Let $\delta(s) f^{s+1} = b(s) f^s$ be a functional equation for $f$.
	If $f^{n+1}\in \fa$ and $f^n\notin \fa$, then $b(n)=0$. In particular, if $f$ admits a Bernstein-Sato polynomial $b_f(s)$, then $b_f(n)=0$.
\end{lemma}
\begin{proof}
	After specializing the functional equation, we have $\delta(n) f^{n+1} = b_f(n) f^n$. Since $\delta(n) f^{n+1}\in \fa$, we must have $b_f(n) f^n \in \fa$, which implies $b_f(n)=0$.
\end{proof}

From the previous lemma, we obtain the following result.

\begin{proposition}
	Let $R$ be a reduced $\NN$-graded $\KK$-algebra. If $D_{R|\KK}$ lives in nonnegative degrees, then no element $f\in [R]_{>0}$ admits a Bernstein-Sato polynomial.
\end{proposition}
\begin{proof}
	Let $\delta(s) f^{s+1} = b(s) f^s$ be a functional equation for $f$.
	Suppose $f\in [R]_w \smallsetminus [R]_{w-1}$. Since $D_{R|\KK}$ has no elements of negative degree, $[R]_{\geq w(n+1)}$ is a $D$-ideal for each $n\in \NN$, and $f^{n+1} \in [R]_{\geq w(n+1)}$, while $f^n\notin [R]_{\geq w(n+1)}$. Thus, $b(n)=0$ for all $n$, so $b(s)\equiv 0$. Thus, $f$ does not admit a Bernstein-Sato polynomial.
	\end{proof}

Large classes of rings with no differential operators of negative degree are known. In particular, we have the following.

\begin{theorem}[{\cite[Corollary~4.49]{BJNB},\cite{HsiaoD},\cite{Mallory}}]\label{thm:log-terminal-negdeg}
	Let $\KK$ be an algebraically closed field of characteristic zero and let
	$R$ be a standard-graded normal $\KK$-domain with an isolated singularity and that is a
	Gorenstein ring. If $R$ has differential operators of negative degree, then $R$ has log-terminal and rational singularities.
	
	 In particular, if $R$ is a hypersurface, and $R$ has differential operators of negative degree, then the degree of $R$ is less than the dimension of $R$.
	\end{theorem}

Mallory recently showed that the hypothesis of log-terminal singularities is not sufficient.

\begin{theorem}[\cite{Mallory}]\label{thm:Mallory}
	Let $\KK$ be an algebraically closed field of characteristic zero.  There are no differential operators of negative degree on the log-terminal hypersurface $R=\KK[x_1,x_2,x_3,x_4]/(x_1^3+x_2^3+x_3^3+x_4^3)$.
\end{theorem}

\begin{corollary}
	For $R$ as in Theorems~\ref{thm:log-terminal-negdeg} and~\ref{thm:Mallory}, no element of $[R]_{\geq 1}$ admits a Bernstein-Sato polynomial.
\end{corollary}

\subsection{Existence of Bernstein-Sato polynomials}

While some rings do not admit Bernstein-Sato polynomials, large classes of singular rings do. 

\begin{definition}
	Let $R,S$ be two rings. We say that $R$ is a direct summand of $S$ if $R\subseteq S$, and there is an $R$-module homomorphism $\beta:S\to R$ such that $\beta|_R$ is the identity on $R$. 
\end{definition}

A major source of direct summands comes from invariant theory: if $G$ is a linearly reductive group acting on a polynomial ring $B$, then $R=B^G$ is a direct summand of $B$. In particular, direct summands of polynomial rings include:

\begin{enumerate}
	\item invariants of finite groups (including the simple singularities $A_n$, $D_n$, $E_n$),
	\item normal toric rings,
	\item determinantal rings, and
	\item coordinate rings of Grassmannians.
\end{enumerate}

We note that a ring $R$ may be a direct summand of a polynomial ring in different ways; i.e., as different subrings of polynomial rings. For example, the $A_1$ singularity $R=\CC[a,b,c]/(c^2-ab)$ embeds as a direct summand of $B=\CC[x,y]$ by the maps
\begin{align*} &\phi_1:R\to B \quad &\phi_1(a)=x^2, \phi_1(b)=y^2, \phi_1(c)=xy&, \ \text{and}\\
 &\phi_2:R\to B \quad &\phi_2(a)=x^4, \phi_2(b)=y^4, \phi_2(c)=x^2y^2;& \ \text{likewise}\\
 &\phi_3:R\to B[z] \quad &\phi_3(a)=x^2, \phi_3(b)=y^2, \phi_3(c)=xy \ \text{splits}.&
 \end{align*}
 We note also that if $R$ is a direct summand of a polynomial ring, there may be other embeddings of $R$ into a polynomial ring that are not split. E.g., for $R$ and $B$ as above,
 \[\phi_4: R\to B \qquad\qquad\qquad \phi_4(a)=x, \phi_4(b)=xy^2, \phi_4(c)=xy\]
 is injective, but no splitting map $\beta|_R$ exists.
 
 \begin{definition}[\cite{BJNB,Vfilt}]\label{deds}
 		Let $R,S$ be two rings. We say that $R$ is a differentially extensible direct summand of $S$ if $R$ is a direct summand of $S$, and for every differential operator $\delta\in D_{R|\KK}$, there is some $\tilde{\delta}\in D_{S|\KK}$ such that $\tilde{\delta}|_R=\delta$. 
 		\end{definition}
 	
 This notion is implicit in a number of papers on differential operators, e.g., \cite{Kantor,LS,Musson,Schwarz}. Differentially extensible direct summands of polynomial rings include
 
 \begin{enumerate}
 	\item invariants of finite groups (including the simple singularities $A_n$, $D_n$, $E_n$),
 	\item normal toric rings,
 	\item determinantal rings, and
 	\item coordinate rings of Grassmannians of lines $\mathrm{Gr}(2,n)$.
 \end{enumerate}

As with the direct summand property, a ring may be a differentially extensible direct summand of a polynomial ring by some embedding, but fail this property for another embedding into a polynomial ring. For the example considered above, $R$ is a differentially extensible direct summand of $B$ via $\phi_1$ and $\phi_3$, but not $\phi_2$ or $\phi_4$.

\begin{theorem}[\cite{AMHNB,BJNB}]\label{existence-DS}
	Let $R$ be a direct summand of a differentiably admissible algebra $B$ over a field $\KK$ of characteristic zero. Then every element $f\in R$ admits a Bernstein-Sato polynomial $b^R_f(s)$, and $b^R_f(s) \, | \, b^B_f(s)$.
	
	If, in addition, $R$ is a differentially extensible direct summand of $B$, then $b^R_f(s) = b^B_f(s)$ for all $f\in R$.
\end{theorem}
\begin{proof}
	Let $\beta:B\to R$ be the splitting map. The key point is that for $\delta\in D_{B|\KK}$, the map $\beta \circ \delta|_R$ is a differential operator on $R$; this is left as an exercise using the inductive definition, or see \cite{DModFSplit}. Thus, given a functional equation $\forall s\in \NN, \delta(s) f^{s+1} = b(s) f^s$ for $f$ in $B$, we have $\forall s\in \NN, \beta \circ \delta(s)|_R \,f^{s+1} = \beta (b(s) f^s)= b(s) f^s$ in $R$. This implies that $f$ admits a Bernstein-Sato polynomial in $R$, and that $b_f^R(s) \, | \, b_f^B(s)$.
	
	If $R$ is a differentially extensible direct summand of $B$, then for any functional equation $\forall s\in \NN, \delta(s) f^{s+1} = b(s) f^s$ for $f$ in $R$, we can take an extension $\tilde{\delta}(s)$ by extending each $s^i$-coefficient, and we then have $\forall s\in \NN, \tilde{\delta}(s) f^{s+1} = b(s) f^s$ in $B$. Thus, $b_f^B(s) \, | \, b_f^R(s)$, so equality holds.
\end{proof}

Note that for direct summands of polynomial rings, all roots of the Bernstein-Sato polynomial are negative and rational, as in the regular case.

We end this section with two examples of Bernstein-Sato polynomials in rings that are not direct summands of polynomial rings.

\begin{example}[\cite{Vfilt}]
	Let $R=\CC[x,y]/(xy)$, and $f=x$. The operator $x \partial^2_x$ is a differential operator on $R$ \cite{Tripp}, and it yields a functional equation
	\[ x \partial^2_x \, x^{s+1} = s(s+1) x^s.\]
	Thus, $b_f^R(s)$ exists, and divides $s(s+1)$. In fact, we have $b_f^R(s)=s(s+1)$. The ideal $(x)$ is a minimal primary component of $(0)$, hence a $D$-ideal. By Lemma~\ref{lem:rootDideal}, $s=0$ is a root; $s=-1$ is also a root since $x$ is not a unit.
	\end{example}

\begin{example}[\cite{Vfilt}]
   Let $R=\CC[t^2,t^3]\cong \displaystyle \frac{\CC[x,y]}{(x^3-y^2)}$ and $f=t^2$.
   Consider the differential operator of order two 
   $$
    \delta = (t \partial_t -1) \circ \partial_t^2 \circ (t \partial_t -1)^{-1},
    $$
   where $(t \partial_xt-1)^{-1}$ is the inverse function of $t \partial_t -1$ on $\R$.
   The equation 
   $$\delta \act t^{2(\ell+1)} = (2\ell+2)(2\ell-1) t^{2\ell}$$ holds for every $\ell\in \NN$. Then, the functional equation 
   \[\delta \act t^2 \boldsymbol{(t^2)^{s}} = (2s+2)(2s-1) \boldsymbol{(t^2)^s}\]
   holds in $R_{t^2}[s] \boldsymbol{(t^2)^{s}}$.
   Thus, $b_{t^2}^{T}(s)$ divides $(s-\frac{1}{2})(s+1)$.

We now see that the equality holds.
We already know that  $s=-1$ is a root of $b_{t^2}^{R}(s)$, because $\frac{1}{t^2}\not\in R$.
 Every differential operator of degree $-2$ on $R$ can be written as $(t \partial_t -1) \circ \partial_t^2 \circ \gamma \circ (t \partial_t -1)^{-1}$ for some $\gamma \in \CC[t \partial_t]$ \cite{PaulSmith,SmithStafford}. 
   Since $R_{t^2}[s] \boldsymbol{(t^2)^{s}}$ is a graded module we can decompose the functional equation as a sum of homogeneous pieces. 
   Using previous  description of such operators, it follows that $s=\frac{1}{2}$ must be a root of $b^{R}_{t^2}(s)$.
\end{example}

\subsection{Differentiable direct summands}

\begin{definition}[{\cite[Definition~3.2]{AMHNB}}]
   Let $R \subseteq B$ be an inclusion of $\KK$-algebras with $R$-linear splitting $\beta\colon B \to R$.
   Recall that, for $\zeta\in D^n_{B|\KK}$, the map $\beta \circ \zeta|_R \colon R \to R$ is an element of $D^n_{R|\KK}$.
   By abuse of notation, for $\delta\in D_{B|\KK}$, we write $\beta \circ \delta|_R$ for the element of $D_{R|\KK}$ obtained from $\delta$ by applying $\beta \circ - |_R$.

   We say that a $D_{R|\KK}$-module $M$ is a \emph{differential direct summand} of a $D_{B|\KK}$-module $N$ if $M\subseteq N$ and there exists an $R$-linear splitting $\Theta\colon N\to M$, called a \emph{differential splitting}, such that
   \[  \Theta(\delta \act v) = (\beta \circ \delta|_R) \act v \]
   for every $\delta\in D_{B|\KK}$ and $v\in M$,
   where the action on the left-hand side is the $D_{B|\KK}$-action, considering $v$ as an element of $N$, and the action on the right-hand side is the $D_{R|\KK}$-action.
\end{definition}

A key property for differential direct summands is that one can deduce finite length.

\begin{theorem}[{\cite[Proposition 3.4]{AMHNB}}]
   Let $R\subseteq B$ be $\KK$-algebras such that $R$ is a direct summand of $B$.
   Let $M$ be a $D_{R|\KK}$-module and $N$ be a $D_{B|\KK}$-module  such that $M$ is a differential direct summand of $N$. Then,
 $$\Len_{D_{R|\KK}}(M)\leq \Len_{D_{B|\KK}}(N).$$ 
In particular,  if  $\Len_{D_{B|\KK}}(N)$ is finite, then $\Len_{D_{R|\KK}}(M)$  is also finite.
\end{theorem}

\begin{definition}[{\cite[Definition~3.5]{AMHNB}}]
   Let $R\subseteq B$ be $\KK$-algebras such that $R$ is a direct summand of $B$.
   Fix  $D_{R|\KK}[\seq{s}]$-modules $M_1$ and $M_2$ that are differential direct summands of $D_{B|\KK}[\seq{s}]$-modules  $N_1$ and $N_2$, respectively, with differential splittings $\Theta_1\colon N_1\to M_1$ and $\Theta_2\colon N_2\to M_2$.
   We call $\phi\colon N_1\to N_2$ a \emph{morphism of differential direct summands} if $\phi\in \Hom_{D_{B|\KK}[\seq{s}]}(N_1,N_2)$, $\phi(M_1)\subseteq M_2$, $\phi|_{M_1}\in \Hom_{D_{R|\KK}[\seq{s}]}(M_1,M_2)$, and the following diagram commutes:
   \[
      \xymatrix{
         M_1 \ar[d]^{\phi|_{M_1}} \ar[r]^{\subseteq} & N_1\ar[d]^{\phi} \ar[r]^{\Theta_1} & M_1\ar[d]^{\phi|_{M_1}}\\
         M_2 \ar[r]^{\subseteq}  & N_2\ar[r]^{\Theta_2} & M_2
      }
   \]
   For simplicity of notation, we often write $\phi$ instead of $\phi|_{M_1}.$

   Further, a complex $M_\bullet$ of $D_{R|\KK}[\seq{s}]$-modules is called a \emph{differential direct summand} of a complex $N_\bullet$ of $D_{B|\KK}[\seq{s}]$-modules if each $M_i$ is a differential direct summand of $N_i$, and each differential is a morphism of differential direct summands. 
\end{definition}

\begin{remark}
 Let $R\subseteq B$ be $\KK$-algebras such that $R$ is a direct summand of $B$.
It is known that the property of being a differential direct summand is preserved under localization at elements of $R$. In addition, it is preserved under taking kernels and cokernels of morphisms of differential direct summands \cite[Proposition~3.6, Lemma~3.7]{AMHNB}.
\end{remark}

We now present several examples of differentiable direct summands built from the previous remark.

\begin{example}
   Let $R\subseteq B$ be $\KK$-algebras such that $R$ is a direct summand of $B$
   \begin{enumerate}
   \item For every $f\in R\smallsetminus\{0\}$, $R_f$ is a differentiable direct summand of $B_f$.
   \item For every ideal $\fa\subseteq R$, $H^i_{\fa}(R)$ is a differentiable direct summand of $H^i_{\fa}(B)$.
   \item For every sequence of ideals $\fa_1,\ldots,\fa_\ell\subseteq R$, $H^i_{\fa_1}\cdots H^i_{\fa_\ell}(R)$ is a differentiable direct summand of $H^i_{\fa_1}\cdots H^i_{\fa_\ell}(B)$.   
   \end{enumerate}

\end{example}

We end this subsection showing that  $R_f[s] \boldsymbol{f^s}$ is a differentiable direct summand of $R_f[s] \boldsymbol{f^s}$.
This gives a more complete approach to prove the existence of the Bernstein-Sato polynomial. 

\begin{theorem}
   Let $R\subseteq B$ be $\KK$-algebras such that $R$ is a direct summand of $B$, and $f\in R\smallsetminus\{0\}$.
   Then,  $R_f[s] \boldsymbol{f^s}$ is a differentiable direct summand of $R_f[s] \boldsymbol{f^s}$.
   In particular, if $B$ is a differentiably admissible $\KK$-algebra, then 
       $M^R[\boldsymbol{f^s}]\otimes_\KK \KK(s)$ has finite length as $D_{R(s)|\KK(s)}$-module, and so, there exists a functional equation 
       $$
       \delta(s) f \boldsymbol{f^s}=b(s) \boldsymbol{f^s},
       $$
       where  $ \delta(s) \in D_{R|\KK}$ and $b(s)\in\KK[s]\smallsetminus\{0\}.$
\end{theorem}

\section{Local cohomology}
In this section we discuss some properties of local cohomology modules for regular rings that follow from the existence of the Bernstein-Sato polynomial.

\begin{proposition}\label{PropRffingen}
	Let $\KK$ be a field of characteristic zero,  $R$ be a $\KK$-algebra, and  $f\in R$ be a nonzero element.
	If $R$ has Bernstein-Sato polynomials, then, $R_f$ is a finitely generated $D_{R|\KK}$-module. In particular, if $b_f^R(s)$ has no integral root less than or equal to $-n$, then $\displaystyle R_f=D_{R|\KK} \cdot \frac{1}{f^{n-1}}$.
\end{proposition}
\begin{proof}
	After specializing the functional equation, we have 
	\[\delta(-t) \frac{1}{f^{t-1}} =b^R_f(-t) \frac{1}{f^{t}}\]
	for all $t\geq n$, with each $b^R_f(-t)\neq 0$. We conclude that each power of $f$, and hence all of $R_f$, is in $D_{R|\KK}\cdot \frac{1}{f^{n-1}}$.
\end{proof}

In fact, a converse to this theorem is true.

\begin{proposition}[{\cite[Proposition~1.3]{WaltherBS}}]\label{uli}
	Let $\KK$ be a field of characteristic zero, $R$ be a $\KK$-algebra, and $f\in R$ have a Bernstein-Sato polynomial. If $-n$ is the smallest integral root of $b_f(s)$, then $\displaystyle \frac{1}{f^{n}} \notin D_{R|\KK} \cdot  \frac{1}{f^{n-1}} \subseteq R_f$.
\end{proposition}

We give a proof of this proposition here, since it appears in the literature only in the regular case. 

\begin{lemma}[{\cite[Proposition~6.2]{KashiwaraRationality}}]
	If $-n$ is the smallest integral root of $b_f(s)$, then 
	\[(s+n+j) D_{R|\KK}[s] \boldsymbol{f^s} \cap D_{R|\KK}[s] f^j \boldsymbol{f^s} = (s+n+j) D_{R|\KK}[s] \boldsymbol{f^s} \  \text{for all}  \ j> 0.\]
\end{lemma}
\begin{proof}
	We proceed by induction on $j$.
	
	Since $b_f(s)$ is the minimal polynomial of the action of $s$ on $\displaystyle \frac{D_{R|\KK}[s] \boldsymbol{f^s}}{D_{R|\KK}[s] f\boldsymbol{f^s}}$ and $-n-j$ is not a root of $b_f(s)$ for $j\geq 1$, the map
	\[ \frac{D_{R|\KK}[s] \boldsymbol{f^s}}{D_{R|\KK}[s] f\boldsymbol{f^s}} \xrightarrow{s+n+j} \frac{D_{R|\KK}[s] \boldsymbol{f^s}}{D_{R|\KK}[s] f\boldsymbol{f^s}} \]
	is an isomorphism. Thus, $(s+n+j)  D_{R|\KK}[s] \boldsymbol{f^s}  \cap D_{R|\KK}[s] f\boldsymbol{f^s} = (s+n+j)D_{R|\KK}[s] f\boldsymbol{f^s}$. In particular, for $j=1$, this covers the base case.
	
	Let $\Sigma: D_{R|\KK}[s] \boldsymbol{f^s} \to D_{R|\KK}[s] \boldsymbol{f^s}$ be the map given by the rule $\Sigma(\delta(s) \boldsymbol{f^s}) = \delta(s+1) f \boldsymbol{f^s}$. Using the induction hypothesis, for $j\geq 2$ we compute
	\begin{align*} (s+n+j)  D_{R|\KK}[s] \boldsymbol{f^s} \,\cap\,  &D_{R|\KK}[s] f^j \boldsymbol{f^s} \subseteq (s+n+j) D_{R|\KK}[s] f\boldsymbol{f^s}  \cap  D_{R|\KK}[s] f^j \boldsymbol{f^s} \\ &= \Sigma ((s+n+j-1) D_{R|\KK}[s] \boldsymbol{f^s} \cap D_{R|\KK}[s] f^{j-1} \boldsymbol{f^s}) \\ 
	&= \Sigma((s+n+j-1) D_{R|\KK}[s] f^{j-1} \boldsymbol{f^s} ) \\
	&= (s+n+j) D_{R|\KK}[s] f^{j} \boldsymbol{f^s}. \qedhere
	\end{align*}
\end{proof}

\begin{lemma}[{\cite[Proposition~6.2]{KashiwaraRationality}}]
	If $-n$ is the smallest integral root of $b_f(s)$, then \[\mathrm{Ann}_{D}(f^{-n}) = D_{R|\KK} \cap (\mathrm{Ann}_{D[s]}(\boldsymbol{f^s})+ D_{R|\KK}[s] (s+n)).\]
\end{lemma}
\begin{proof}
	Let $\delta\in \mathrm{Ann}_{D}(f^{-n})$. Write ${\delta \boldsymbol{f^s} = f^{-m} g(s) \boldsymbol{f^s}}$, with $g(s)\in R[s]$. In fact, we can take $m$ to be the order of $\delta$. Then $g(-n)=0$. By Remark~\ref{rem:equivs-gen}, 
	$$\delta \cdot f^m \boldsymbol{f^s} = g(s+m) \boldsymbol{f^s}.$$ 
Set $h(s)=g(s+m)$. We then have that ${h(-n-m)=g(-n)=0}$, so ${(s+n+m) | h(s)}$. Thus, ${\delta \cdot f^m \boldsymbol{f^s} \in (s+m+n) D_{R|\KK}[s] \boldsymbol{f^s}}$, and ${\delta \cdot f^m \boldsymbol{f^s} \in D_{R|\KK}[s] f^m \boldsymbol{f^s}}$ by definition. By the previous lemma, we obtain that ${\delta \cdot f^m \boldsymbol{f^s} \in (s+m+n) D_{R|\KK}[s] f^m \boldsymbol{f^s}}$. We can then write ${\delta \cdot f^m \boldsymbol{f^s} = (s+m+n) h'(s) \boldsymbol{f^s}}$ for some $h'(s)\in R[s]$. By Remark~\ref{rem:equivs-gen}, we have that ${\delta \cdot \boldsymbol{f^s} = (s+n) h'(s-m)\boldsymbol{f^s}}$. Thus, we can write $\delta$ as a sum of a multiple of $(s+n)$ and an element in the annihilator of $\boldsymbol{f^s}$.
\end{proof}

\begin{proof}[Proof of Proposition~\ref{uli}]
	Suppose that $ \displaystyle \frac{1}{f^{n}} \in D_{R|\KK} \frac{1}{f^{n-1}}$. Then we can write $D_{R|\KK} = D_{R|\KK} f + \mathrm{Ann}_{D}(\frac{1}{f^{n}})$. From the previous lemma, 
	we have that \[\mathrm{Ann}_{D}(\frac{1}{f^{n}}) = D_{R|\KK} \cap (\mathrm{Ann}_{D[s]}(\boldsymbol{f^s})+D_{R|\KK}[s](s+n)).\]
	Then, \[1\in D_{R|\KK} f + \mathrm{Ann}_{D[s]}(\boldsymbol{f^s})+D_{R|\KK}[s](s+n).\]
	Multiplying by $\frac{b_f(s)}{s+n}$, we get
	\[\frac{b_f(s)}{s+n} \in \mathrm{Ann}_{D[s]}(\boldsymbol{f^s}) + D_{R|\KK} f  +  D_{R|\KK}[s]\, b_f(s).\]
	Since $b_f(s)\in D_{R|\KK} f + \mathrm{Ann}_{D[s]}(\boldsymbol{f^s})$, using Remark~\ref{rem:equivs-gen} we have	\[\frac{b_f(s)}{s+n} D_{R|\KK}[s] \in \mathrm{Ann}_{D[s]}(\boldsymbol{f^s}) + D_{R|\KK}[s] f,\]
	which contradicts that $b_f(s)$ is the minimal polynomial in $s$ contained in 
	$$\mathrm{Ann}_{D[s]}(\boldsymbol{f^s}) + D_{R|\KK}[s] f.$$
\end{proof}

\begin{remark}
 Proposition~\ref{uli} extends to the setting of the $D_{R|\KK}$-modules $D_{R|\KK} f^{\alpha}$ for $\alpha\in \QQ$ discussed in Remark~\ref{rem:D-mod-alpha}. Namely, if $\alpha\in \QQ$ is such that $b_f(\alpha)=0$ and $b_f(\alpha-i)\neq 0$ for all integers $i>0$, then $f^{\alpha} \notin D_{R|\KK} \cdot f^{\alpha+1}$ in the $D_{R|\KK}$-module $R_f f^{\alpha}$.
 
 It is not true in general that $b_f(\alpha)=0$ implies $f^{\alpha} \notin D_{R|\KK} \cdot f^{\alpha+1}$, even in the regular case: an example is given by Saito \cite{Saitorootnojump}. However, this implication does hold when $R=A$ is a polynomial ring, and $f$ is quasihomogeneous with an isolated singularity \cite{BitSch}. We are not aware of an example where $b_f(n)=0$ and $f^{n} \in D_{R|\KK} \cdot f^{n+1}$ for an integer $n$.
\end{remark}

We also relate existence of Bernstein-Sato polynomials to finiteness properties of local cohomology.

\begin{theorem}\label{thm:finite}
Let $\KK$ be a field of characteristic zero,  $R$ be a $\KK$-algebra, and  $f\in R$ be a nonzero element.
Suppose that  $R$ has Bernstein-Sato polynomials and $D_{R|\KK}$ is a Noetherian ring.
Then, $H^i_\fa(R)$ is a finitely generated $D_{R|\KK}$-module, and $\Ass_R(H^i_\fa(R))$ is finite for every ideal $\fa\subseteq R$. 
\end{theorem}
\begin{proof}
Let $F=f_1,\ldots, f_\ell$ be a set of generators for $\fa$. 
We have that the \v{C}ech complex asociated to $F$ is a complex of finitely generated $D_{R|\KK}$-modules.
Since $D_{R|\KK}$ is Noetherian,
 the   \v{C}ech complex  is a complex of Noetherian  $D_{R|\KK}$-modules.
 Then, the cohomology of this complex is also a Noetherian $D_{R|\KK}$-module.

It suffices to show that  a Noetherian $D_{R|\KK}$-module, $N$,  has a finite set of associated primes.
 We build inductively a sequence of $D_{R|\KK}$-submodules $N_i\subseteq N$
as follows.  We set $N_0=0$. Given $N_t$, we pick a maximal element $\p_t\in \Ass_R (N/N_t)$. This is possible if and only if $\Ass_R (N/N_t)\neq \varnothing$.
We set  $\tilde{N}_{t+1}=H^0_{\p}(N/N_t)$, which is nonzero,  and $N_{t+1}$ the preimage of $\tilde{N}_{t+1}$ in $N$ under the quotient map. We have that $\Ass_R(\tilde{N}_{t+1})=\{\p\}$, and so, $\Ass_R (N_{t+1})=\{\p\}\cup \Ass_R(N_t)$.
We note that this sequence cannot be infinite, because $N$ is Noetherian. Then, the sequence stops, and there is a $k\in\NN$ such that  $N_k=N$. We conclude that $\Ass_R(N)\subseteq \{\p_1,\ldots, \p_k\}$.
\end{proof}

\section{Complex zeta functions} \label{Complex_zeta}

The foundational work of Bernstein \cite{Ber71, Ber72} where he developed the theory of $D$-modules and proved the existence of Bernstein-Sato polynomials was motivated by a question of I.~M.~Gel'fand \cite{Gelf54} at the 1954 edition of the International Congress of Mathematicians
regarding the analytic continuation of the \emph{complex zeta function.} Bernstein's work relates the poles of the complex zeta function to the roots of the Bernstein-Sato polynomials. Previously, Bernstein and S.~I.~Gel'fand \cite{BG69} and independently Atiyah \cite{Ati70}, gave a different approach to the same question using resolution of singularities. 

\vskip 2mm

Throughout this section we  consider  \(A=\mathbb{C}[x_1, \dots, x_d]\) and the corresponding ring of differential operators  \( D_{A|\CC} \). 
 Given a differential operator \( \delta(s) = \sum_{\alpha} a_\alpha(x,s) \partial^\alpha \in D_{A|\CC}[s]\), which is polynomial in $s$, we  denote the \emph{conjugate} and the \emph{adjoint} of \(\delta(s)\) as
\begin{equation*}
\bar{\delta}(s) := \sum_{\alpha} a_{\alpha}(\bar{x}, \bar{s})  \overline{\partial}^{\alpha}, \quad \delta^*(s) := \sum_{\alpha} (-1)^{|\alpha|} \partial^{\alpha} a_{\alpha}(x, s),
\end{equation*}
where we are using the multidegree notation $\partial ^\alpha:=\partial_1^{\alpha_1}\cdots \partial_d^{\alpha_d}$ and $\overline{\partial}^\alpha:=\overline{\partial}_1^{\alpha_1}\cdots \overline{\partial}_d^{\alpha_d}$ with $\overline{\partial}_i = \frac{d}{d\overline{x_i}}$.

\vskip 2mm

Let \(f(x) \in A\) be a non-constant polynomial and let \( \varphi(x) \in C^\infty_c(\mathbb{C}^d) \) be a \emph{test function}: an infinitely many times differentiable function with compact support.   We define the parametric distribution \(f^s : C^{\infty}_c(\mathbb{C}^d) \longrightarrow \mathbb{C}\) by means of the integral
\begin{equation} \label{int-eq}
\langle f^s, \varphi \rangle := \int_{\mathbb{C}^d}|f(x)|^{2s}  \varphi(x, \bar{x}) dx d\bar{x},
\end{equation}
which is well-defined analytic function for any \(s \in \mathbb{C}\) with \(\textnormal{Re}(s) > 0\). We point out that test functions  have holomorphic and  antiholomorphic part so we use the notation \(\varphi = \varphi(x, \bar{x})\). We refer to \( f^s \) or \(\langle f^s, \varphi \rangle\) as the \emph{complex zeta function} of \( f \).

\vskip 2mm

The approach given by Bernstein  in order to solve I.~M.~Gel'fand's question  uses the Bernstein-Sato polynomial and integration by parts as follows: 
\begin{align*}
 \langle f^s, \varphi \rangle & = \int_{\mathbb{C}^d} \varphi(x, \bar{x})|f(x)|^{2s} dxd\bar{x} \\ & = \frac{1}{b^2_{f}(s)} \int_{\mathbb{C}^d} \varphi(x, \bar{x}) \big[\delta(s) \cdot f^{s+1}(x)\big]\big[\bar{\delta}(s) \cdot f^{s+1}(\bar{x})\big] dx d\bar{x} \\
 & = \frac{1}{b^2_{f}(s)} \int_{\mathbb{C}^d} \bar{\delta}^*\delta^*(s)\big(\varphi(x, \bar{x})\big)|f(x)|^{2(s+1)} dx d\bar{x}\\
  & = \frac{\langle f^{s+1}, \bar{\delta}^*\delta^*(s)(\varphi) \rangle}{b^2_{f}(s)} .
\end{align*}
Thus we get an analytic function whenever \(\textnormal{Re}(s) > -1\), except for possible poles at \(b_{f}^{-1}(0)\), and it is equal to \(\langle f^s, \varphi \rangle\) in \(\textnormal{Re}(s) > 0\). Iterating the process we get
\begin{equation*} \label{anal-cont2}
\langle f^s, \varphi \rangle = \frac{\langle f^{s+\ell+1}, \bar{\delta}^*\delta^*(s+\ell)\cdots \bar{\delta}^*\delta^*(s)(\varphi) \rangle}{b^2_{f}(s) \cdots b^2_{f}(s+\ell)}, \quad {\textnormal{Re}}(s) > -\ell -1,
\end{equation*}

In particular we have the following relation between the poles of the complex zeta function and the roots of the Bernstein-Sato polynomial.

\begin{theorem}\label{poles_BS} The complex zeta function $f^s$ admits a meromorphic continuation to \(\CC\)  and the set of poles is included in 
\(
  \{ \lambda - \ell \hskip 2mm | \hskip 2mm  b_f(\lambda) = 0 \quad {\rm and} \quad \ell \in \mathbb{Z}_{\geq 0} \}.
\)
\end{theorem}

Both sets are equal for reduced plane curves and isolated quasi-homogeneous singularities by work of Loeser \cite{Loe85}.

\vskip 2mm

On the other hand, the approach given by Bernstein and S.~I.~Gel'fand, and independently Atiyah uses resolution of singularities in order to reduce the problem to the monomial case,  which was already solved  by Gel'fand and Shilov \cite{GS64}.  
Let  $\pi: X' \rightarrow \CC^n$ be a \emph{log-resolution} of $f\in A$ and  \[ F_\pi := \sum_{i=1}^r N_i E_i + \sum_{j=1}^s N'_j S_j  \hskip 2mm \text{and} \hskip 2mm K_\pi := \sum_{i=1}^r k_i E_i\] be the total transform  and the relative canonical divisors. 

\vskip 2mm

The analytic continuation problem is attacked in this case using a change of variables.
\begin{equation*} \label{int-eq_3}
\langle f^s, \varphi \rangle = \int_{\mathbb{C}^d}|f(x)|^{2s}  \varphi(x, \bar{x}) dx d\bar{x} = \int_{X'} |\pi^* f|^{2s} (\pi^* \varphi)  |d\pi|^2 
\end{equation*}
where $|d\pi|^2 =(\pi^\ast dx )(\pi^\ast d\overline{x})$ and $d\pi$ is the Jacobian determinant of  $\pi$. In order to describe the terms of the last integral we consider a finite affine open cover \( \{U_{\alpha}\}_{\alpha \in \Lambda} \) of \( E\subseteq X' \) such that \( \textnormal{Supp}(\varphi) \subseteq \pi(\cup_\alpha U_{\alpha}) \). Consider a set of local coordinates $z_1,\dots , z_d$ in a given $U_{\alpha}$. Then we have 
\[ \pi^\ast f = u_\alpha(z) z_1^{N_{1,\alpha}} \cdots z_d^{N_{d,\alpha}}, \quad |d\pi|^2 = |v_\alpha(z)|^2 |z_1|^{2k_{1,\alpha}} \cdots |z_d|^{k_{d,\alpha}} dz d\overline{z} \]
where $u_\alpha(z)$ and $v_\alpha(z)$ are units and $N_{i,\alpha}$ may denote both the multiplicities of the exceptional divisors or of the strict transform. Take \( \{\eta_\alpha\} \) a partition of unity subordinated to the cover \( \{U_{\alpha}\}_{\alpha \in \Lambda} \). That is, \( \eta_\alpha \in C^\infty(\mathbb{C}^d) \),
\( \sum_\alpha \eta_\alpha \equiv 1 \), with only finitely many \( \eta_\alpha \) being nonzero at a point of \( X' \) and \(\textnormal{Supp}(\eta_\alpha) \subseteq U_{\alpha} \). Therefore
\begin{equation*} \label{anal-cont0}
\begin{split}
  \langle f^s, \varphi \rangle & = \int_{X'} |\pi^* f|^{2s} (\pi^* \varphi)  (\pi^\ast dx )(\pi^\ast d\overline{x}) \\
  & = \sum_{\alpha \in \Lambda} \int_{U_{\alpha}} |z_{1}|^{2(N_{1,\alpha} s + k_{1,\alpha})} \cdots |z_{d}|^{2(N_{d,\alpha} s + k_{d,\alpha})} |u_{\alpha}(z)|^{2s} |v_{\alpha}(z)|^2 \varphi_\alpha(z, \bar{z}) dz d\bar{z},
\end{split}
\end{equation*}
where \( \varphi_\alpha := \eta_\alpha \pi^* \varphi \) for each \( \alpha \in \Lambda \). Notice that  \( \pi^{-1}(\textnormal{Supp}(\varphi)) \) is a compact set because  \(\pi\) is a proper morphism.

\vskip 2mm

Once we reduced the problem to the monomial case, we can use the work of Gel'fand and Shilov \cite{GS64} on \emph{regularization} to generate a set of candidate poles of~\( f^s \).

\begin{theorem}\label{poles_resolution} The complex zeta function $f^s$ admits a meromorphic continuation to \(\CC\)  and the set of poles is included in 
\[
  \left\{ - \frac{k_i + 1 + \ell}{N_i} \ |\ \ell \in \mathbb{Z}_{\geq 0}\right\} \cup \left\{- \frac{\ell + 1}{N'_j} \ |\ \ell \in \mathbb{Z}_{\geq 0} \right\}.
\]
\end{theorem}

The fundamental result of Kashiwara \cite{KashiwaraRationality} and Malgrange \cite{BernsteinRationalMalgrange} on the rationality of the roots of the Bernstein-Sato mentioned in \Cref{rationality_BS} was refined later on by Lichtin \cite{Lic89}.  He provides the same set of candidates for the roots of the Bernstein-Sato polynomial in terms of the numerical data of the log-resolution of \(f\).

\begin{theorem}[{\cite{Lic89}}]\label{Kashiwara_Lichtin} 
Let $f\in A$ be a polynomial. Then, the  roots of the Bernstein-Sato polynomial of $f$ are included in  the set 
\[
  \left\{ - \frac{k_i + 1 + \ell}{N_i} \ |\ \ell \in \mathbb{Z}_{\geq 0}\right\} \cup \left\{- \frac{\ell + 1}{N'_j} \ |\ \ell \in \mathbb{Z}_{\geq 0} \right\}.
\] 
In particular, the roots of the Bernstein-Sato polynomial of $f$ are negative rational numbers.
\end{theorem}

This result has recently been extended by Dirks and Musta\c{t}\u{a} \cite{DirksMustata}.

We also have a bound for the roots given by Saito \cite{Sai94} in terms of the \emph{log-canonical threshold} of $f$, 
\[
 \lct(f):=\min_{i,j} \left\{  \frac{k_i + 1}{N_i}, \frac{ 1}{N'_j} \right\}.
\]

\begin{theorem} [{\cite{Sai94}}]
Let $f\in A$ be a polynomial. Then, the  roots of the Bernstein-Sato polynomial of $f$ are contained in the interval \(  [-d+\emph{\textnormal{lct}}(f), -\emph{\textnormal{lct}}(f)]\).
\end{theorem}

In general the set of candidates that we have for the poles of the complex zeta function or the roots of the Bernstein-Sato polynomial is too big. In order to separate the wheat from the chaff we consider the notion of \emph{contributing divisors}.

\begin{definition}
We say that a divisor $E_i$ or $S_j$ contributes to a pole $\lambda$ of the complex zeta function $f^s$ or to a root $\lambda$ of the Bernstein-Sato polynomial of $f$, if we have \(\lambda= - \frac{k_i + 1 + \ell}{N_i} \) or \(\lambda= - \frac{\ell + 1}{N'_j} \) for some $\ell \in \mathbb{Z}_{\geq 0}$.
\end{definition}

It is an open question to determine the contributing divisors (see \cite{Kollar1997}). Also we point out that, in general, the divisors contributing to poles are different from the divisors contributing to roots. This is not the case for reduced plane curves and isolated quasi-homogeneous singularities by work of Loeser \cite[Theorem~1.9]{Loe85}. In the case of reduced plane curves, Blanco \cite{Bla19} determined the contributing divisors.

\vskip 2mm

Although we have a set of candidate poles of the complex zeta function one has to ensure that a candidate is indeed a pole by checking the corresponding \emph{residue}. This can be quite challenging and was already posed as a question by I.~M.~Gel'fand \cite{Gelf54}. In the case of plane curves we have a complete description given by Blanco \cite{Bla19}. 
Moreover, it is not straightforward to relate poles of the complex zeta function to roots of the Bernstein-Sato polynomial. We have that a pole  \( \lambda \in [-d+\emph{\textnormal{lct}}(f), -\emph{\textnormal{lct}}(f)]\)
such that $\lambda + \ell$ is not a root of \( b_f(s) \) for all \( \ell \in \mathbb{Z}_{>0} \) is a root of \( b_{f}(s) \) but this is not enough to recover all the roots of the Bernstein-Sato polynomial even if we know all the poles of the complex zeta function.

\section{Multiplier ideals} \label{Multipliers}

Let \(f\in A=\mathbb{C}[x_1, \dots, x_d]\) be a polynomial. As we mentioned in \Cref{Log-resolutions}, the family of \emph{multiplier ideals} of $f$ is an important object in birational geometry that is described using a log-resolution of $f$ and comes with a discrete set of rational numbers, the \emph{jumping numbers}, that are also related to the roots of the Bernstein-Sato polynomial.

\vskip 2mm

We start with an analytic approach to multiplier ideals that has its origin in the work of Kohn \cite{Kohn}, Nadel \cite{Nad90}, and Siu \cite{Siu01}. The idea behind the construction is to measure the singularity of \(f\) at a point \( p \in Z(f)  \subseteq \mathbb{C}^d \) using the convergence of certain integrals. 

\begin{definition}
Let \( f \in A \) and \( p \in Z(f) \). Let ${\overline{B}_{\epsilon}(p)} $ be a closed ball of radius $\epsilon$ and center $p$. The multiplier ideal of \( f \) at \( p \) associated with a rational number \( \lambda \in \mathbb{Q}_{>0} \) is
\begin{equation*}
\mathcal{J}(f^\lambda)_p = \big\{ g \in A\ \big|\ \exists\, \epsilon \ll 1\ \textnormal{such that}\ \int_{\overline{B}_{\epsilon}(p)} \frac{|g|^2}{|f|^{2\lambda}} dxd\overline{x} < \infty \big\}.
\end{equation*}
More generally we consider \(\mathcal{J}(f^\lambda) =\cap_{p\in Z(f) }\mathcal{J}(f^\lambda)_p \).
\end{definition}

Similarly to the case of the complex zeta function we may use a log-resolution $\pi: X' \rightarrow \CC^d$ of $f$ to reduce the above integral to a monomial case where we can easily check its convergence.  
\begin{equation*}
\int_{\overline{B}_\epsilon(p)} \frac{|g|^2}{|f|^{2\lambda}} dxd\overline{x} = \int_{\pi^{-1}\big(\overline{B}_\epsilon(p)\big)} \frac{|\pi^\ast g|^2}{|\pi^* f|^{2\lambda}} |d\pi| ^2, 
\end{equation*}

Consider a finite affine open cover \( \{U_{\alpha}\}_{\alpha \in \Lambda} \) of \( \pi^{-1}\big(\overline{B}_\epsilon(p) \big)  \) which is still a compact set since \( \pi \) is proper. We have to check the convergence of the integral at each $U_{\alpha}$ so let $z_1,\dots , z_d$ be a set of local coordinates in such an open set. Taking local equations for $ \pi^\ast f$,  $\pi^\ast g$ we get

\begin{align*}
& \int_{U_\alpha}  \frac{|u(z)\,z_1^{L_{1,\alpha}} \cdots z_d^{L_{d,\alpha}}|^2}{|z_1^{N_{1,\alpha}} \cdots z_d^{N_{d,\alpha}}|^{2 \lambda}}  |z_1^{k_{1,\alpha}} \cdots z_d^{k_{d,\alpha}}| ^2 dz d\overline{z} \\ &  = \int_{U_\alpha}  |u(z)| \, |z_1|^{2(L_{1,\alpha}+ k_{1,\alpha}- \lambda N_{1,\alpha}) } \cdots |z_d|^{2(L_{d,\alpha}+ k_{d,\alpha}- \lambda N_{d,\alpha})}   dz d\overline{z} .
\end{align*}
where $u(z)$ is a unit. Using Fubini's theorem we have that the integral converges if and only if 
\[ L_{i}+ k_{i}- \lambda N_{i} > -1, \quad  L'_{j} - \lambda N'_{j} > -1\]  for all $i,j$. Here we use that the total transform divisors of $f$ and $g$ are respectively \[ F_\pi := \sum_{i=1}^r N_i E_i + \sum_{j=1}^s N'_j S_j, \quad G_\pi := \sum_{i=1}^r L_i E_i + \sum_{j=1}^t L'_j S'_j\] and the components of the strict transform of $g$ must contain the components of $f$. Equivalently, we require 
\[  L_{i}  \geq -\lceil k_{i}- \lambda N_{i} \rceil , \quad  L'_{j} \geq \lceil  \lambda N'_{j} \rceil\] 
so we are saying that $\pi^\ast g$ is a section of $\mathcal{O}_{X'}(\lceil K_\pi - \lambda F_\pi\rceil)$. This fact leads to the algebraic geometry definition of multiplier ideals given in \Cref{Mult_ideal} that we refine to the local case.

\begin{definition}
Let  $\pi: X' \rightarrow \CC^d$ be a \emph{log-resolution} of \( f \in A\) and let $F_\pi$ be the total transform divisor. The multiplier ideal of \( f \)  at  \( p \in Z(f) \) associated with a real number \( \lambda \in \mathbb{R}_{>0} \) is the stalk at $p$ of
\[\mathcal{J}(f^\lambda) = \pi_*\mathcal{O}_{X'}\left(\left\lceil K_{\pi} - \lambda F_\pi \right\rceil\right).\]
\end{definition}

 We omit the reference to the point $p$ if it is clear from the context. 
Recall that the multiplier ideals form a discrete filtration
$$A\supsetneqq\mathcal{J}(f^{\lambda_1})\supsetneqq \mathcal{J}(f^{\lambda_2})\supsetneqq\cdots\supsetneqq \mathcal{J}(f^{\lambda_i}))\supsetneqq\cdots$$
and the  $\lambda_i$ where  we have a strict inclusion of ideals are the {\em jumping numbers} of $f$ and $\lambda_1={\rm lct}(f)$ is the log-canonical threshold.

\vskip 2mm

There is a way to describe a set of candidate jumping numbers in a reasonable time. However,   contrary to the case of roots of the Bernstein-Sato polynomial, the jumping numbers are not bounded. However they satisfy some periodicity given by the following version of Skoda's theorem, which for principal ideals reads as $\mathcal{J}(f^\lambda) = (f) \cdot \mathcal{J}(f^{\lambda-1}) $ for all $\lambda \geqslant 1$.

\begin{theorem}\label{Kashiwara_Lichtin_2} 
Let $f\in A$ be a polynomial. Then, the  jumping numbers of $f$ are included in the set 
\[
  \left\{  \frac{k_i + 1 + \ell}{N_i} \ |\ \ell \in \mathbb{Z}_{\geq 0}\right\} \cup \left\{ \frac{\ell + 1}{N'_j} \ |\ \ell \in \mathbb{Z}_{\geq 0} \right\}.
\] 
In particular, the jumping numbers of $f$ form a discrete set of positive rational numbers.
\end{theorem}

We see that we have the same set of candidates for the roots of the Bernstein-Sato polynomial and the jumping numbers so it is natural to ask how these invariants of singularities are related. The result that we are going to present is due to  Ein, Lazarsfeld, Smith, and Varolin \cite{ELSV2004}. A different proof of the same result can be found in the work of Budur and Saito \cite{BudurSaito05} that relies on the theory of $V$-filtrations.

\begin{theorem}[{\cite{ELSV2004,BudurSaito05}}]\label{ELSV}
 Let $\lambda\in (0,1]$ be a jumping number of a polynomial $f\in A$. Then $-\lambda$ is a root of the Bernstein-Sato polynomial $b_f(s)$.
\end{theorem}

\begin{proof}
 Let $\lambda\in (0,1]$ be a jumping number  and take $g\in \mathcal{J}(f^{\lambda -\varepsilon}) \smallsetminus \mathcal{J}(f^{\lambda})$
 for $\varepsilon > 0$ small enough.  Therefore \( \frac{|g(x)|^2}{|f(x)|^{2(\lambda-\varepsilon)}}\) is integrable but when we take the limit $\varepsilon \rightarrow 0$   we end up with \( \frac{|g(x)|^2}{|f(x)|^{2\lambda}}\) that is not integrable.

Consider Bernstein-Sato functional equation $\delta(s) \cdot f^{s+1} = b_f(s) \cdot f^s$ and its application to the analytic continuation of the complex zeta function
\[
b^2_{f}(s) \int_{\mathbb{C}^d} \varphi(x, \bar{x})|f(x)|^{2s} dxd\bar{x}  = 
 \int_{\mathbb{C}^d} \bar{\delta}^*\delta^*(s)\big(\varphi(x, \bar{x})\big)|f(x)|^{2(s+1)} dx d\bar{x}.\]
Notice that $|g(x)|^2\varphi(x, \bar{x})$ is still a test function so
\[
b^2_{f}(s) \int_{\mathbb{C}^d} |g|^2 \varphi(x, \bar{x})|f(x)|^{2s} dxd\bar{x}  = 
 \int_{\mathbb{C}^d} \bar{\delta}^*\delta^*(s)\big(|g|^2\varphi(x, \bar{x})\big)|f(x)|^{2(s+1)} dx d\bar{x}.\]
Now we take a test function $\varphi$ which is zero outside the ball $\overline{B}_{\epsilon}(p)$ and identically one on a smaller ball  $\overline{B}_{\epsilon'}(p)\subseteq \overline{B}_{\epsilon}(p)$ and thus we get
 \[
b^2_{f}(s) \int_{\overline{B}_{\epsilon'}(p)} |g|^2 |f(x)|^{2s} dxd\bar{x}  = 
 \int_{\overline{B}_{\epsilon'}(p)} \bar{\delta}^*\delta^*(s)\big(|g|^2\big)|f(x)|^{2(s+1)} dx d\bar{x}.\]
Taking $s=-(\lambda - \varepsilon)$ we get
 \[
b^2_{f}(-\lambda + \varepsilon) \int_{\overline{B}_{\epsilon'}(p)} \frac{|g|^2} {|f(x)|^{2(\lambda - \varepsilon)}} dxd\bar{x}  = 
 \int_{\overline{B}_{\epsilon'}(p)} \bar{\delta}^*\delta^*(-\lambda + \varepsilon)\big(|g|^2\big)|f(x)|^{2(1-\lambda + \varepsilon)} dx d\bar{x}\]
but the right-hand side is uniformly bounded for all $\varepsilon >0$. Thus we have 
\[
b^2_{f}(-\lambda + \varepsilon) \int_{\overline{B}_{\epsilon'}(p)} \frac{|g|^2} {|f(x)|^{2(\lambda - \varepsilon)}} dxd\bar{x}  \leq  M < \infty\]
for some positive number $M$ that depends on $g$. Then, by the monotone convergence theorem we have to have \(b^2_{f}(-\lambda) = 0 \).
\end{proof}

So far we have been dealing with the case of an hypersurface $f\in A$ for the sake of clarity but everything works just fine for any ideal  $\mathfrak{a}=\langle f_1, \dots , f_m \rangle\subseteq A$. The analytical definition of multiplier ideal at a point \( p \in Z(\mathfrak{a}) \) associated with a rational number \( \lambda \in \mathbb{Q}_{>0} \) is
\begin{equation*}
\mathcal{J}(\mathfrak{a}^\lambda)_p = \big\{ g \in A\ \big|\ \exists\, \epsilon \ll 1\ \textnormal{such that}\ \int_{\overline{B}_{\epsilon}(p)} \frac{|g|^2}{(|f_1|^{2}+ \cdots + |f_m|^{2})^\lambda} dxd\overline{x} < \infty \big\}.
\end{equation*}
and  \(\mathcal{J}(\mathfrak{a}^\lambda) =\cap_{p\in Z(\mathfrak{a}) }\mathcal{J}(\mathfrak{a}^\lambda)_p \).
One can show that the ideal that we obtain is independent of the set of generators of the ideal $\mathfrak{a}$.

\vskip 2mm

For the  algebraic geometry version we consider the stalk at $p$ of  the multiplier ideal
\[\mathcal{J}(\mathfrak{a}^\lambda) = \pi_*\mathcal{O}_{X'}\left(\left\lceil K_{\pi} - \lambda F_\pi \right\rceil\right),\] given in \Cref{Mult_ideal}.  The extension of \Cref{ELSV} to this setting was proved by Budur, Musta\c{t}\u{a}, and Saito \cite{BMS2006a}    using the theory of $V$-filtrations.

\begin{theorem}[{\cite{BMS2006a}}]\label{BMS}
 Let $\lambda\in ({\rm lct}(\mathfrak{a}),{\rm lct}(\mathfrak{a})+1]$ be a jumping number of  $\mathfrak{a} \subseteq  A$. Then $-\lambda$ is a root of the Bernstein-Sato polynomial $b_{\mathfrak{a}}(s)$.
\end{theorem}

Finally we want to mention that multiplier ideals can be characterized completely in terms of relative Bernstein-Sato polynomials. Namely:

\begin{theorem}[{\cite{BMS2006a}}]\label{BMS+}
	For all ideals $\fa\subseteq A$ and all $\lambda$ we have the equality
	\[  \cJ(\fa^\lambda) = \{g \in A \ | \ \gamma > \lambda \text{ if } b_{\fa,g}(-\gamma) = 0\}.\]
\end{theorem}

This theorem is due to Budur and Saito \cite{BudurSaito05} in the case $\fa$ is principal, and due to Budur, Musta\c{t}\u{a}, and Saito \cite{BMS2006a} as stated. The proofs rely on the theory of mixed Hodge modules. Recent work of Dirks and Musta\c{t}\u{a} \cite{DirksMustata} provides a proof of this result that does not use the theory of mixed Hodge modules.

The analogues of Theorems~\ref{BMS} and~\ref{BMS+} have been shown to hold for certain singular rings.

To illustrate Theorem~\ref{BMS+}, we use this description of multiplier ideals to give a quick proof of Skoda's Theorem in the principal ideal case.

\begin{proposition}[Skoda's theorem for principal ideals]
	For all $f\in A\smallsetminus\{0\}$ and all $\lambda$, we have $\cJ(f^{\lambda+1}) = (f) \cJ(f^{\lambda})$.
	\end{proposition}
\begin{proof}
	Let $g\in \cJ(f^{\lambda})$, so every root of $b_{f,g}(s)$ is less than $-\lambda$. Then, by Lemma~\ref{shift}, every root of $b_{f,fg}(s)$ is less than $-\lambda-1$, and hence $fg\in \cJ(f^{\lambda+1})$. This shows the containment $\cJ(f^{\lambda+1}) \supseteq (f) \cJ(f^{\lambda})$.
	
	Now, if $g\notin (f)$, then $s=-1$ is a root of $b_{f,g}(s)$ by Lemma~\ref{s+n}. Thus, ${\cJ(f^{\lambda+1})\subseteq (f)}$. In particular, we can write $h\in \cJ(f^{\lambda+1})$ as $h=fg$ for $g\in A$; since the largest root of $b_{f,g}(s)$ is one greater than the largest root of $b_{f,h}(s)$ by Lemma~\ref{shift}, we have that $h\in \cJ(f^{\lambda})$, and the equality follows.
	\end{proof}

\begin{theorem}[{\cite{Vfilt}}] Let $R$ be either
a ring of invariants of an action of a finite group on a polynomial ring, or
	 an affine normal toric ring.
	Then, for every ideal $\fa\subseteq R$, we have the log canonical threshold of $\fa$ in $R$ coincides with the smallest
	root $\alpha$ of $b^R_{\fa}(-s)$, and every jumping number of $\fa$ in  $[\alpha, \alpha + 1)$ is a root of $b^R_{\fa}(-s)$. Moreover,
	\[ \cJ_R(\fa^\lambda)= \{g \in R \ | \ \gamma > \lambda \text{ if } b^R_{\fa,g}(-\gamma) = 0\}.\]
\end{theorem}

The idea behind the proof of this theorem is based on reduction modulo $p$ and a positive characteristic analogue of the notion of differentially extensibility direct summand as in Definition~\ref{deds}. We refer the reader to \cite{Vfilt} for details.

\section{Computations via F-thresholds}

The notion of Bernstein-Sato root in positive characteristic discussed in Section~\ref{sec:positivechar} is closely related to $F$-jumping numbers. In this section, we discuss a relationship between the classical Bernstein-Sato polynomial in characteristic zero and similar numerical invariants in characteristic $p$. This connection was first established by Musta\c{t}\u{a}, Takagi, and Watanabe \cite{MTW2005}, and extended to the singular setting by \`Alvarez Montaner, Huneke, and N\'u\~nez-Betancourt \cite{AMHNB}.

\begin{definition}[\cite{MTW2005}]
	Let $R$ be a ring of characteristic $p>0$. Let $\fa, J$ be ideals of $R$ such that $\fa\subseteq \sqrt{J}$. We set 
	\[\nu_{\fa}^J(p^e)= \max\{ n\in \NN \ | \ \fa^n \not\subseteq J^{[p^e]}\}.\]
\end{definition}

We point out that the limit of $\lim\limits_{e\to\infty}\frac{\nu_{\fa}^J(p^e)}{p^e}$ exists \cite{DSNBP}.

\begin{theorem}[{\cite{AMHNB}, see also \cite{MTW2005}}]
	Let $R$ be a finitely generated flat $\ZZ[1/a]$-algebra for some nonzero $a\in \ZZ$, and $\fa\subseteq \sqrt{J}$ ideals of $R$. Write $R_0$ for $R\otimes_{\ZZ} \QQ$, and $R_p$ for $R/pR$; likewise, write $\fa_0$ for the extension of $\fa$ to $R_0$, and similarly for $\fa_p$, $J_0$, $J_p$, etc. If $\fa_0$ has a Bernstein-Sato polynomial in $R_0$, then we have
	\[((s+1) b_{\fa_0}^{A_0})(\nu_{\fa_p}^{J_p}(p^e)) \equiv 0 \ \mathrm{mod} \ p\]
	for all $p\gg 0$.
	\end{theorem}
\begin{proof}[Sketch of proof]
	First, if $\fa = (f_1,\dots,f_\ell)$, set $g=\sum_i f_i y_i \in R'=R[y_1,\dots,y_{\ell}]$. Then, one checks easily that for $p\nmid a$,  we have $\nu_{\fa_p}^{J_p}(p^e)=\nu_{g_p}^{J R'_p}(p^e)$. Thus, we can reduce to the principal case, where $\fa=(f)$.
	
	Let $\delta(s) f^{s+1} = b_f(s) f^s$ be a functional equation for $f$ in. If we replace $a$ by a nonzero multiple, we can assume that $\delta(s)$ is contained in the image of $D_R[s]$ in $D_{R_0}[s]$ (see \cite[Lemma~4.18]{AMHNB}) and that $b_f(s)\in \ZZ[1/a][s]$. Pick $n$ such that $\delta(s)\in D^n_R[s]$ and $n$ is greater than any prime dividing a denominator of a coefficient of $b_f(s)$. Then, for every $p\geq n$, we may take the functional equation modulo $p$ in $R_p$:
	\[ \overline{\delta(s)} f^{s+1} = \overline{b_f(s)} f^s. \]
	Since $n<p$, we have $\overline{\delta(s)}\in D^{(1)}_{R_p|\FF_p}$. In particular, $\overline{\delta(s)}$ is linear over each subring $R^{[p^e]}$, so it stabilizes every ideal expanded from such a subring, namely the Frobenius powers $J^{[p^e]}$ of $J$. For $s=\nu_{f_p}^{J_p}$, we have $f^s\notin J^{[p^e]}$, and $f^{s+1} \in J^{[p^e]}$, so $\overline{\delta(s)} f^{s+1} \in J^{[p^e]}$; we conclude that $\overline{b_f(s)}=0$ in $\FF_p$, as claimed.
	\end{proof}

The previous theorem can be applied to find roots of $b^{A_0}_{\fa_0}(s)$ in $\QQ$ when there are sufficiently nice formulas for $\nu_{\fa_p}^{J_p}(p^e)$ for $e$ fixed as $p$ varies. 

\begin{proposition}[{\cite{MTW2005}}]\label{prop-roots-nu}
	Let $R$ be a finitely generated flat $\ZZ[1/a]$-algebra for some nonzero $a\in \ZZ$, and $\fa\subseteq \sqrt{J}$ ideals of $R$. Write $R_0$ for $R\otimes_{\ZZ} \QQ$, and $R_p$ for $R/pR$; likewise, write $\fa_0$ for the extension of $\fa$ to $R_0$, and similarly for $\fa_p$, $J_0$, $J_p$, etc. Suppose that $\fa_0$ has a Bernstein-Sato polynomial in $R_0$. 
	
	Let $e>0$. Suppose that there is an integer $N$ and polynomials $Q_{[i]}$ for each $[i]\in (\ZZ/N\ZZ)^\times$ such that $\nu_{\fa_p}^{J_p}(p^e)=Q_{[i]}(p^e)$ for all $p\gg 0$ with $p\in [i]$. Then $Q_{[i]}(0)$ is a root of $b_{\fa_0}^{A_0}(s)$ for each $[i]\in (\ZZ/N\ZZ)^\times$.
\end{proposition}
\begin{proof}
	We can consider $b_{\fa_0}^{A_0}(s)$ as a polynomial over $\ZZ[1/aa']$ for some $a'$. Fix ${[i]\in (\ZZ/N\ZZ)^\times}$. For any $p\in [i]$ with $p\nmid (aa')$, we have \[(s+1)b_{\fa_0}^{A_0}(Q_{[i]}(0)) \equiv b_{\fa_0}^{A_0}(Q_{[i]}(p^e)) \equiv 0 \mod  p,\] so  $p \,|\, b_{\fa_0}^{A_0}(Q_{[i]}(0))$. As there are infinitely many primes $p\in [i]$, we must have $b_{\fa_0}^{A_0}(Q_{[i]}(0))=0$.
\end{proof}

\begin{example}[{\cite{MTW2005}}]
	Let $f=x^2+y^3\in \ZZ[x,y]$, and $\mathfrak{m}=(x,y)$. One has
	\[\nu_{f_p}^{\mathfrak{m}}(p^e) =  \begin{cases} \frac{5}{6}p^e -\frac{5}{6}& \text{if } p \equiv 1 \mod 3 \\
	\nu_{f_p}^{\mathfrak{m}}(p^e) = \frac{5}{6}p-\frac{7}{6} & \text{if } p \equiv 2 \mod 3, \ e=1 \\
	\nu_{f_p}^{\mathfrak{m}}(p^e) = \frac{5}{6}p^e-\frac{1}{6}p^{e-1} -1 & \text{if } p \equiv 2 \mod 3, \ e\geq 2.
	  \end{cases} \]
	  By the previous proposition, $-5/6, -1$ and $-7/6$ are roots of $b_f(s)$, considering $f$ as an element of $\QQ[x,y]$. In fact, $b_f(s)=(s+\frac{5}{6})(s+1)(s+\frac{7}{6})$. 
\end{example}

We note that the method of Proposition~\ref{prop-roots-nu} does not yield any information about the multiplicities of the roots. There are also examples given in \cite{MTW2005} of Bernstein-Sato polynomials with roots that cannot be recovered by this method. Nonetheless, we note that this method was successfully employed by Budur, Musta\c{t}\u{a}, and Saito \cite{BMS2006} to compute the Bernstein-Sato polynomials of monomial ideals.

\begin{remark}
	In the case of a regular ring $A=\FF_p[x_1,\dots,x_d]$, and ideals $\fa,J$ of $A$ with $\fa\subseteq \sqrt{J}$, the numbers $\nu_{\fa}^J(p^e)$ are closely related to the $F$-jumping numbers discussed in the introduction. In particular, combining \cite[Propositions~1.9~\&~2.7]{MTW2005} for $\fa$ and $e$ fixed, we have
	\[ \{ \nu_{\fa}^J(p^e) \ | \ \sqrt{J} \supseteq \fa\} = \{ \lceil p^e \lambda \rceil - 1 \ | \ \lambda \text{ is an $F$-jumping number of } \fa\}. \]
\end{remark}

\section*{Acknowledgments}
We thank Guillem Blanco, Manuel Gonz\'alez Villa, and Luis Narv\'aez-Macarro for helpful comments.

\newcommand{\etalchar}[1]{$^{#1}$}


\begin{thebibliography}{ABCNLMH17}

\bibitem[{\`A}HNB17]{AMHNB}
Josep {{\`A}lvarez Montaner}, Craig Huneke, and Luis N{\'u}{\~n}ez-Betancourt.
\newblock {$D$}-modules, {B}ernstein-{S}ato polynomials and {$F$}-invariants of
  direct summands.
\newblock {\em Adv. Math.}, 321:298--325, 2017.

\bibitem[{\`A}HJ{\etalchar{+}}19]{Vfilt}
Josep {{\`A}lvarez Montaner}, Daniel~J. Hern\'andez, Jack Jeffries, Luis
  N{\'u}{\~n}ez-Betancourt, Pedro Teixeira, and Emily~E. Witt.
\newblock Bernstein-{S}ato functional equations, {V}-filtrations, and
  multiplier ideals of direct summands.
\newblock 2019.
\newblock Preprint, \href{https://arxiv.org/abs/1907.10017}{arXiv:1907.10017}. To appear in {\em Commun. Contemp. Math}.



\bibitem[ALM09]{AndresLevMM}
Daniel Andres, Viktor Levandovskyy, and Jorge~Mart\'{\i}n Morales.
\newblock Principal intersection and {B}ernstein-{S}ato polynomial of an affine
  variety.
\newblock In {\em I{SSAC} 2009---{P}roceedings of the 2009 {I}nternational
  {S}ymposium on {S}ymbolic and {A}lgebraic {C}omputation}, pages 231--238.
  ACM, New York, 2009.
  
  \bibitem[ACNLM17]{ACNLM17}
Enrique Artal~Bartolo, Pierrette Cassou-Nogu\`es, Ignacio Luengo, and Alejandro
  Melle-Hern\'{a}ndez.
\newblock Yano's conjecture for two-{P}uiseux-pair irreducible plane curve
  singularities.
\newblock {\em Publ. Res. Inst. Math. Sci.}, 53(1):211--239, 2017.

\bibitem[Ati70]{Ati70}
Michael~F. Atiyah.
\newblock Resolution of singularities and division of distributions.
\newblock {\em Comm. Pure Appl. Math.}, 23:145--150, 1970.

\bibitem[Bah01]{Bahloul}
Rouchdi Bahloul.
\newblock Algorithm for computing {B}ernstein-{S}ato ideals associated with a
  polynomial mapping.
\newblock {\em J. Symbolic Comput.}, 32(6):643--662, 2001.

\bibitem[BO10]{BahloulOaku}
Rouchdi Bahloul and Toshinori Oaku.
\newblock Local {B}ernstein-{S}ato ideals: algorithm and examples.
\newblock {\em J. Symbolic Comput.}, 45(1):46--59, 2010.

\bibitem[Bat20]{Bath}
Daniel Bath.
\newblock Combinatorially determined zeroes of Bernstein-Sato ideals for tame and free arrangements.
\newblock {\em J. Sing.}, 20:165--204, 2020.


\bibitem[Ber71]{Ber71}
Joseph~N. Bern\v{s}te\u{\i}n.
\newblock Modules over a ring of differential operators. {A}n investigation of
  the fundamental solutions of equations with constant coefficients.
\newblock {\em Funkcional. Anal. i Prilo\v{z}en.}, 5(2):1--16, 1971.

\bibitem[Ber72]{Ber72}
Joseph.~N. Bern{\v{s}}te{\u\i}n.
\newblock Analytic continuation of generalized functions with respect to a
  parameter.
\newblock {\em Funkcional. Anal. i Prilo\v zen.}, 6(4):26--40, 1972.

\bibitem[BG69]{BG69}
Joseph~N. Bern\v{s}te\u{\i}n and Sergei~I. Gel'fand.
\newblock Meromorphy of the function {$P^{\lambda }$}.
\newblock {\em Funkcional. Anal. i Prilo\v{z}en.}, 3(1):84--85, 1969.

\bibitem[BGG72]{DiffNonNoeth}
Joseph.~N. Bern{\v{s}}te{\u\i}n, Israel~M. Gel'fand, and Sergei~I. Gel'fand.
\newblock Differential operators on a cubic cone.
\newblock {\em Uspehi Mat. Nauk}, 27(1(163)):185--190, 1972.

\bibitem[BL10]{BLAlg}
Christine Berkesch and Anton Leykin.
\newblock Algorithms for {B}ernstein-{S}ato polynomials and multiplier ideals.
\newblock In {\em I{SSAC} 2010---{P}roceedings of the 2010 {I}nternational
  {S}ymposium on {S}ymbolic and {A}lgebraic {C}omputation}, pages 99--106. ACM,
  New York, 2010.






\bibitem[Bit18]{BitounBSpos}
Thomas Bitoun.
\newblock On a theory of the {$b$}-function in positive characteristic.
\newblock {\em Selecta Math. (N.S.)}, 24(4):3501--3528, 2018.



\bibitem[BS18]{BitSch}
Thomas Bitoun and Travis Schedler.
\newblock On {D}-modules related to the {$b$}-function and {H}amiltonian flow.
\newblock {\em Compos. Math.}, 154(11):2426--2440, 2018.

\bibitem[Bj{\"{o}}79]{Bjork79}
Jan-Erik Bj{\"{o}}rk.
\newblock {\em Rings of differential operators}, volume~21 of {\em
  North-Holland Mathematical Library}.
\newblock North-Holland Publishing Co., Amsterdam-New York, 1979.




\bibitem[Bla19]{Bla19}
Guillem Blanco.
\newblock Poles of the complex zeta function of a plane curve.
\newblock {\em Adv. Math.}, 350:396--439, 2019.

\bibitem[Bla21]{Blanco_yano}
Guillem Blanco.
\newblock Yano's conjecture.
\newblock {\em Invent. Math.}, 2021. 
\newblock \href{https://doi.org/10.1007/s00222-021-01052-2}{doi.org/10.1007/s00222-021-01052-2}.


\bibitem[Bli13]{BlickleP-1maps}
Manuel Blickle.
\newblock Test ideals via algebras of {$p^{-e}$}-linear maps.
\newblock {\em J. Algebraic Geom.}, 22(1):49--83, 2013.



\bibitem[BB11]{BB-CartierMod}
Manuel Blickle and Gebhard B{\"o}ckle.
\newblock Cartier modules: finiteness results.
\newblock {\em J. Reine Angew. Math.}, 661:85--123, 2011.


\bibitem[BMS08]{BMS2008}
Manuel Blickle, Mircea Musta{\c{t}}{\v{a}}, and Karen~E. Smith.
\newblock Discreteness and rationality of {$F$}-thresholds.
\newblock {\em Michigan Math. J.}, 57:43--61, 2008.

\bibitem[BMS09]{BMS2009}
Manuel Blickle, Mircea Musta{\c{t}}{\u{a}}, and Karen~E. Smith.
\newblock {$F$}-thresholds of hypersurfaces.
\newblock {\em Trans. Amer. Math. Soc.}, 361(12):6549--6565, 2009.



\bibitem[BS16]{BliSta}
Manuel Blickle and Axel St\"{a}bler.
\newblock Bernstein-{S}ato polynomials and test modules in positive
  characteristic.
\newblock {\em Nagoya Math. J.}, 222(1):74--99, 2016.

\bibitem[BJNnB19]{BJNB}
Holger Brenner, Jack Jeffries, and Luis N\'{u}\~{n}ez Betancourt.
\newblock Quantifying singularities with differential operators.
\newblock {\em Adv. Math.}, 358:106843, 89, 2019.

\bibitem[BGM86]{BGM_pre}
J\"oel Brian{\c{c}}on, Michel Granger, and Philippe Maisonobe.
\newblock Sur le polyn\^ome de {B}ernstein des singularit\'es semi
  quasi-homog\`enes, 1986.
\newblock Pr\'epublication de l'Universit\'e de Nice,.

\bibitem[BGMM89]{BrianconAlgorithm}
J\"oel Brian{\c{c}}on, Michel Granger, Philippe Maisonobe, and Michel Miniconi.
\newblock Algorithme de calcul du polyn\^ome de {B}ernstein: cas non
  d{\'e}g{\'e}n{\'e}r{\'e}.
\newblock {\em Ann. Inst. Fourier (Grenoble)}, 39(3):553--610, 1989.



\bibitem[BM96]{BM_Rabida}
J\"oel Brian\c{c}on and Philippe Maisonobe.
\newblock Caract\'{e}risation g\'{e}om\'{e}trique de l'existence du
  polyn\^{o}me de {B}ernstein relatif.
\newblock In {\em Algebraic geometry and singularities ({L}a {R}\'{a}bida,
  1991)}, volume 134 of {\em Progr. Math.}, pages 215--236. Birkh\"{a}user,
  Basel, 1996.



\bibitem[BM02]{BriMai02}
Jo\"el Brian\c{c}on and Philippe Maisonobe.
\newblock Bernstein-{S}ato ideals associated to polynomials {I}, 2002.
\newblock Unpublished notes.

\bibitem[BMT07]{Briancon2007}
Jo{{\"e}}l Brian{\c{c}}on, Philippe Maisonobe, and Tristan Torrelli.
\newblock Matrice magique associ{\'e}e {\`a} un germe de courbe plane et
  division par l'id{\'e}al jacobien.
\newblock {\em Ann. Inst. Fourier (Grenoble)}, 57(3):919--953, 2007.


\bibitem[BM99]{BriMay}
Jo\"el Brian\c{c}on and H\'el\`ene Maynadier.
\newblock \'{E}quations fonctionnelles g\'{e}n\'{e}ralis\'{e}es:
  transversalit\'{e} et principalit\'{e} de l'id\'{e}al de {B}ernstein-{S}ato.
\newblock {\em J. Math. Kyoto Univ.}, 39(2):215--232, 1999.





\bibitem[Bud05]{SurveyBudur}
Nero Budur.
\newblock On the {$V$}-filtration of {$D$}-modules.
\newblock In {\em Geometric methods in algebra and number theory}, volume 235
  of {\em Progr. Math.}, pages 59--70. Birkh\"auser, Boston, MA, 2005.
  
  

\bibitem[Bud15a]{Bud15}
Nero Budur.
\newblock Bernstein-{S}ato ideals and local systems.
\newblock {\em Ann. Inst. Fourier (Grenoble)}, 65(2):549--603, 2015.


\bibitem[Bud15b]{BudurNotes}
Nero Budur.
\newblock Bernstein-{S}ato polynomials.
\newblock Lecture notes for the summer school \emph{Multiplier Ideals, Test
  Ideals, and Bernstein-Sato Polynomials}, at UPC Barcelona, available at
  \href{https://perswww.kuleuven.be/~u0089821/Barcelona/BarcelonaNotes.pdf}{https://perswww.kuleuven.be/~u0089821/Barcelona/BarcelonaNotes.pdf},
  2015.


\bibitem[BLSW17]{BLSW}
Nero Budur, Yongqiang Liu, Luis Saumell, and Botong Wang.
\newblock Cohomology support loci of local systems.
\newblock {\em Michigan Math. J.}, 66(2):295--307, 2017.




\bibitem[BMS06a]{BMS2006c}
Nero Budur, Mircea Musta\c{t}\v{a}, and Morihiko Saito.
\newblock Combinatorial description of the roots of the {B}ernstein-{S}ato
  polynomials for monomial ideals.
\newblock {\em Comm. Algebra}, 34(11):4103--4117, 2006.

\bibitem[BMS06b]{BMS2006a}
Nero Budur, Mircea Musta{\c t}{\v a}, and Morihiko Saito.
\newblock Bernstein-{S}ato polynomials of arbitrary varieties.
\newblock {\em Compos. Math.}, 142(3):779--797, 2006.

\bibitem[BMS06c]{BMS2006}
Nero Budur, Mircea Musta{\c{t}}{\v{a}}, and Morihiko Saito.
\newblock Roots of {B}ernstein-{S}ato polynomials for monomial ideals: a
  positive characteristic approach.
\newblock {\em Math. Res. Lett.}, 13(1):125--142, 2006.


\bibitem[BS05]{BudurSaito05}
Nero Budur and Morihiko Saito.
\newblock Multiplier ideals, {$V$}-filtration, and spectrum.
\newblock {\em J. Algebraic Geom.}, 14(2):269--282, 2005.






\bibitem[BvdVWZ21]{BvVWZ}
Nero Budur, Robin van~der Veer, Lei Wu, and Peng Zhou.
\newblock Zero loci of {B}ernstein-{S}ato ideals.
\newblock {\em Invent. Math.}, 225(1):45--72, 2021. 




\bibitem[BW17]{BudWang}
Nero Budur and Botong Wang.
\newblock Local systems on analytic germ complements.
\newblock {\em Adv. Math.}, 306:905--928, 2017.


\bibitem[CSS13]{CSS_Cayley}
Sergio Caracciolo, Alan~D. Sokal, and Andrea Sportiello.
\newblock Algebraic/combinatorial proofs of {C}ayley-type identities for
  derivatives of determinants and {P}faffians.
\newblock {\em Adv. in Appl. Math.}, 50(4):474--594, 2013.

\bibitem[CA00]{Cas00}
Eduardo Casas-Alvero.
\newblock {\em Singularities of plane curves}, volume 276 of {\em London
  Mathematical Society Lecture Note Series}.
\newblock Cambridge University Press, Cambridge, 2000.




\bibitem[CN86]{Cassou1986}
Pierrette Cassou-Nogu{{\`e}}s.
\newblock Racines de poly\^omes de {B}ernstein.
\newblock {\em Ann. Inst. Fourier (Grenoble)}, 36(4):1--30, 1986.

\bibitem[CN87]{Cassou1987}
Pierrette Cassou-Nogu{{\`e}}s.
\newblock \'{E}tude du comportement du polyn\^ome de {B}ernstein lors d'une
  d{\'e}formation {\`a} {$\mu$}-constant de {$X^a+Y^b$} avec {$(a,b)=1$}.
\newblock {\em Compositio Math.}, 63(3):291--313, 1987.

\bibitem[CN88]{Cassou1988}
P.~Cassou-Nogu\`es.
\newblock Polyn\^{o}me de {B}ernstein g\'{e}n\'{e}rique.
\newblock {\em Abh. Math. Sem. Univ. Hamburg}, 58:103--123, 1988.

\bibitem[Cou95]{Coutinho}
S.~C. Coutinho.
\newblock {\em A primer of algebraic {$D$}-modules}, volume~33 of {\em London
  Mathematical Society Student Texts}.
\newblock Cambridge University Press, Cambridge, 1995.



\bibitem[dFH09]{dFH}
Tommaso de~Fernex and Christopher~D. Hacon.
\newblock Singularities on normal varieties.
\newblock {\em Compos. Math.}, 145(2):393--414, 2009.

\bibitem[DL92]{Denef_Loeser}
Jan Denef and Fran\c{c}ois Loeser.
\newblock Caract\'{e}ristiques d'{E}uler-{P}oincar\'{e}, fonctions z\^{e}ta
  locales et modifications analytiques.
\newblock {\em J. Amer. Math. Soc.}, 5(4):705--720, 1992.

\bibitem[DM20]{DirksMustata}
Bradley Dirks and Mircea Musta{\c{t}}{{a}}.
\newblock Upper bounds for roots of {B}-functions, following {K}ashiwara and
  {L}ichtin, 2020.
\newblock  Preprint, \href{https://arxiv.org/abs/2003.03842 }{arXiv:2003.03842}. To appear in {\em Publ. Res. Inst. Math. Sci.} 

\bibitem[DSNBP18]{DSNBP}
Alessandro De~Stefani, Luis N{\'u}{\~n}ez-Betancourt, and Felipe P{\'e}rez.
\newblock On the existence of {$F$}-thresholds and related limits.
\newblock {\em Trans. Amer. Math. Soc.}, 370(9):6629--6650, 2018.

\bibitem[ELSV04]{ELSV2004}
Lawrence Ein, Robert Lazarsfeld, Karen~E. Smith, and Dror Varolin.
\newblock Jumping coefficients of multiplier ideals.
\newblock {\em Duke Math. J.}, 123(3):469--506, 2004.

\bibitem[EGSS02]{UliAlg}
David Eisenbud, Daniel~R. Grayson, Michael Stillman, and Bernd Sturmfels,
  editors.
\newblock {\em Computations in algebraic geometry with {M}acaulay 2}, volume~8
  of {\em Algorithms and Computation in Mathematics}.
\newblock Springer-Verlag, Berlin, 2002.

\bibitem[EN85]{EN85}
David Eisenbud and Walter Neumann.
\newblock {\em Three-dimensional link theory and invariants of plane curve
  singularities}, volume 110 of {\em Annals of Mathematics Studies}.
\newblock Princeton University Press, Princeton, NJ, 1985.






\bibitem[EC85]{EC15}
Federigo Enriques and Oscar Chisini.
\newblock {\em Lezioni sulla teoria geometrica delle equazioni e delle funzioni
  algebriche. 1. {V}ol. {I}, {II}}, volume~5 of {\em Collana di Matematica
  [Mathematics Collection]}.
\newblock Zanichelli Editore S.p.A., Bologna, 1985.
\newblock Reprint of the 1915 and 1918 editions.

\bibitem[Gel57]{Gelf54}
Israel~M. Gel'fand.
\newblock Some aspects of functional analysis and algebra.
\newblock In {\em Proceedings of the {I}nternational {C}ongress of
  {M}athematicians, {A}msterdam, 1954, {V}ol. 1}, pages 253--276. Erven P.
  Noordhoff N.V., Groningen; North-Holland Publishing Co., Amsterdam, 1957.
  
  \bibitem[GS64]{GS64}
Israel~M. Gel'fand and Georgiy~E. Shilov.
\newblock {\em Generalized functions. {V}ol. {I}: {P}roperties and operations}.
\newblock Translated by Eugene Saletan. Academic Press, New York-London, 1964.


\bibitem[Gra10]{Granger2010}
Michel Granger.
\newblock Bernstein-{S}ato polynomials and functional equations.
\newblock In {\em Algebraic approach to differential equations}, pages
  225--291. World Sci. Publ., Hackensack, NJ, 2010.

\bibitem[Gro67]{EGA}
Alexander Grothendieck.
\newblock \'{E}l{\'e}ments de g{\'e}om{\'e}trie alg{\'e}brique: {IV}. \'{E}tude
  locale des sch{\'e}mas et des morphismes de sch{\'e}mas ({Q}uatri\`eme
  partie).
\newblock {\em Inst. Hautes {\'E}tudes Sci. Publ. Math.}, (32):361, 1967.


\bibitem[Gyo93]{Gyo}
Akihiko Gyoja.
\newblock Bernstein-{S}ato's polynomial for several analytic functions.
\newblock {\em J. Math. Kyoto Univ.}, 33(2):399--411, 1993.




\bibitem[HY03]{HY2003}
Nobuo Hara and Ken-{i}chi Yoshida.
\newblock A generalization of tight closure and multiplier ideals.
\newblock {\em Trans. Amer. Math. Soc.}, 355(8):3143--3174, 2003.


\bibitem[HS99]{HerSta}
Claus Hertling and Colin Stahlke.
\newblock Bernstein polynomial and {T}jurina number.
\newblock {\em Geom. Dedicata}, 75(2):137--176, 1999.





\bibitem[{Hsi}15]{HsiaoD}
Jen-Chieh {Hsiao}.
\newblock {A remark on bigness of the tangent bundle of a smooth projective
  variety and $D$-simplicity of its section rings.}
\newblock {\em {J. Algebra Appl.}}, 14(7):10, 2015.

\bibitem[HM18]{HsiaoMatusevich}
Jen-Chieh Hsiao and Laura~Felicia Matusevich.
\newblock Bernstein-{S}ato polynomials on normal toric varieties.
\newblock {\em Michigan Math. J.}, 67(1):117--132, 2018.


\bibitem[HH89]{HoHuStrong}
Melvin Hochster and Craig Huneke.
\newblock Tight closure and strong {$F$}-regularity.
\newblock {\em M\'em. Soc. Math. France (N.S.)}, (38):119--133, 1989.
\newblock Colloque en l'honneur de Pierre Samuel (Orsay, 1987).

\bibitem[HH90]{HoHu1}
Melvin Hochster and Craig Huneke.
\newblock Tight closure, invariant theory, and the {B}rian\c con-{S}koda
  theorem.
\newblock {\em J. Amer. Math. Soc.}, 3(1):31--116, 1990.

\bibitem[HH94a]{HoHu2}
Melvin Hochster and Craig Huneke.
\newblock {$F$}-regularity, test elements, and smooth base change.
\newblock {\em Trans. Amer. Math. Soc.}, 346(1):1--62, 1994.

\bibitem[HH94b]{HoHu3}
Melvin Hochster and Craig Huneke.
\newblock Tight closure of parameter ideals and splitting in module-finite
  extensions.
\newblock {\em J. Algebraic Geom.}, 3(4):599--670, 1994.



\bibitem[Igu00]{Igusa_book}
Jun-ichi Igusa.
\newblock {\em An introduction to the theory of local zeta functions},
  volume~14 of {\em AMS/IP Studies in Advanced Mathematics}.
\newblock American Mathematical Society, Providence, RI; International Press,
  Cambridge, MA, 2000.

\bibitem[Kan77]{Kantor}
Jean-Michel Kantor.
\newblock Formes et op\'erateurs diff{\'e}rentiels sur les espaces analytiques
  complexes.
\newblock {\em Bull. Soc. Math. France M{\'e}m.}, (53):5--80, 1977.



\bibitem[Kas70]{Kas70}
Masaki Kashiwara.
\newblock Algebraic study of systems of partial differential equations. Master thesis, Univ. of Tokyo, Dec 1970. English translation 
\newblock{\em M\'em. Soc. Math. France}, (65), 1995. xiv+72 pp. 2



\bibitem[Kas77]{KashiwaraRationality}
Masaki Kashiwara.
\newblock {$B$}-functions and holonomic systems. {R}ationality of roots of
  {$B$}-functions.
\newblock {\em Invent. Math.}, 38(1):33--53, 1976/77.



\bibitem[Kas83]{KashiwaraVfil}
Masaki Kashiwara.
\newblock Vanishing cycle sheaves and holonomic systems of differential
  equations.
\newblock In {\em Algebraic geometry ({T}okyo/{K}yoto, 1982)}, volume 1016 of
  {\em Lecture Notes in Math.}, pages 134--142. Springer, Berlin, 1983.


\bibitem[Kat81]{Kato1981}
Mitsuo Kato.
\newblock The {$b$}-function of a {$\mu $}-constant deformation of
  {$x^{7}+y^{5}$}.
\newblock {\em Bull. College Sci. Univ. Ryukyus}, (32):5--10, 1981.

\bibitem[Kat82]{Kato1982}
Mitsuo Kato.
\newblock The {$b$}-function of {$\mu $}-constant deformation of
  {$x^{9}+y^{4}$}.
\newblock {\em Bull. College Sci. Univ. Ryukyus}, (33):5--8, 1982.

\bibitem[Koh79]{Kohn}
John~J. Kohn.
\newblock Subellipticity of the {$\bar \partial $}-{N}eumann problem on
  pseudo-convex domains: sufficient conditions.
\newblock {\em Acta Math.}, 142(1-2):79--122, 1979.

\bibitem[Kol97]{Kollar1997}
J{{\'a}}nos Koll{{\'a}}r.
\newblock Singularities of pairs.
\newblock In {\em Algebraic geometry---{S}anta {C}ruz 1995}, volume~62 of {\em
  Proc. Sympos. Pure Math.}, pages 221--287. Amer. Math. Soc., Providence, RI,
  1997.

\bibitem[Laz04]{Laz2004}
R.~Lazarsfeld.
\newblock {\em Positivity in {A}lgebraic {G}eometry {II}}.
\newblock Springer-Verlag, Berlin, 2004.


\bibitem[LMM12]{LevandovskyyMartinMorales2012}
Viktor Levandovskyy and Jorge Mart{\'{\i}}n-Morales.
\newblock Algorithms for checking rational roots of {$b$}-functions and their
  applications.
\newblock {\em J. Algebra}, 352:408--429, 2012.

\bibitem[LS89]{LS}
Thierry Levasseur and J.~Toby Stafford.
\newblock Rings of differential operators on classical rings of invariants.
\newblock {\em Mem. Amer. Math. Soc.}, 81(412):vi+117, 1989.

\bibitem[Lic89]{Lic89}
Ben Lichtin.
\newblock Poles of {$|f(z,w)|^{2s}$} and roots of the {$b$}-function.
\newblock {\em Ark. Mat.}, 27(2):283--304, 1989.


\bibitem[Loe85]{Loe85}
Fran\c{c}ois Loeser.
\newblock Quelques cons\'{e}quences locales de la th\'{e}orie de {H}odge.
\newblock {\em Ann. Inst. Fourier (Grenoble)}, 35(1):75--92, 1985.

\bibitem[Loe88]{Loeser88}
Fran\c{c}ois Loeser.
\newblock Fonctions d'{I}gusa {$p$}-adiques et polyn\^{o}mes de {B}ernstein.
\newblock {\em Amer. J. Math.}, 110(1):1--21, 1988.

\bibitem[L{\H{o}}r20]{LA}
Andr\'{a}s~Cristian L{\H{o}}rincz.
\newblock Decompositions of {B}ernstein-{S}ato polynomials and slices.
\newblock {\em Transform. Groups}, 25(2):577--607, 2020.

\bibitem[LRWW17]{LRWW}
Andr\'{a}s~C. L\H{o}rincz, Claudiu Raicu, Uli Walther, and Jerzy Weyman.
\newblock Bernstein-{S}ato polynomials for maximal minors and sub-maximal
  {P}faffians.
\newblock {\em Adv. Math.}, 307:224--252, 2017.



\bibitem[Lyu93]{Lyubeznik93}
Gennady Lyubeznik.
\newblock Finiteness properties of local cohomology modules (an application of
  {$D$}-modules to commutative algebra).
\newblock {\em Invent. Math.}, 113(1):41--55, 1993.

\bibitem[Lyu00]{LyuMixChar}
Gennady Lyubeznik.
\newblock Finiteness properties of local cohomology modules for regular local
  rings of mixed characteristic: the unramified case.
\newblock volume~28, pages 5867--5882. 2000.
\newblock Special issue in honor of Robin Hartshorne.

\bibitem[Mai16a]{Mai16a}
Philippe Maisonobe.
\newblock Filtration relative, l'id\'eal de {B}ernstein et ses pentes, 2016.
\newblock Preprint, \href{https://arxiv.org/abs/1610.03354 }{arXiv:1610.03354}.

\bibitem[Mai16b]{Mai16b}
Philippe Maisonobe.
\newblock Id\'eal de {B}ernstein d'un arrangement central g\'en\'erique
  d'hyperplans, 2016.
\newblock Preprint, \href{https://arxiv.org/abs/1610.03357 }{arXiv:1610.03357}.

\bibitem[Mal]{Mallory}
Devlin Mallory.
\newblock Bigness of the tangent bundle of del Pezzo surfaces and {$ D $}-simplicity, 2020.
\newblock Preprint, \href{https://arxiv.org/abs/2002.11010 }{arXiv:2002.11010}. To appear in {\em Algebra Number Theory}.

\bibitem[Mal74a]{Mal74b}
Bernard Malgrange.
\newblock Int\'{e}grales asymptotiques et monodromie.
\newblock {\em Ann. Sci. \'{E}cole Norm. Sup. (4)}, 7:405--430 (1975), 1974.

\bibitem[Mal74b]{Mal74}
Bernard Malgrange.
\newblock Sur les polyn\^{o}mes de {I}. {N}. {B}ernstein.
\newblock In {\em S\'{e}minaire {G}oulaouic-{S}chwartz 1973--1974:
  \'{E}quations aux d\'{e}riv\'{e}es partielles et analyse fonctionnelle,
  {E}xp. {N}o. 20}, page~10. Centre de Math., \'{E}cole Polytech., Paris, 1974.

\bibitem[Mal75]{BernsteinRationalMalgrange}
Bernard Malgrange.
\newblock Le polyn\^ome de {B}ernstein d'une singularit{\'e} isol{\'e}e.
\newblock In {\em Fourier integral operators and partial differential equations
  ({C}olloq. {I}nternat., {U}niv. {N}ice, {N}ice, 1974)}, pages 98--119.
  Lecture Notes in Math., Vol. 459. Springer, Berlin, 1975.

\bibitem[Mal83]{MalgrangeVfil}
Bernard Malgrange.
\newblock Polyn\^omes de {B}ernstein-{S}ato et cohomologie \'evanescente.
\newblock In {\em Analysis and topology on singular spaces, {II},
  {III}({L}uminy, 1981)}, volume 101 of {\em Ast\'erisque}, pages 243--267.
  Soc. Math. France, Paris, 1983.

\bibitem[Mat80]{Matsumura}
Hideyuki Matsumura.
\newblock {\em Commutative algebra}, volume~56 of {\em Mathematics Lecture Note
  Series}.
\newblock Benjamin/Cummings Publishing Co., Inc., Reading, Mass., second
  edition, 1980.

\bibitem[May97]{May}
H\'{e}l\`ene Maynadier.
\newblock Polyn\^{o}mes de {B}ernstein-{S}ato associ\'{e}s \`a une intersection
  compl\`ete quasi-homog\`ene \`a singularit\'{e} isol\'{e}e.
\newblock {\em Bull. Soc. Math. France}, 125(4):547--571, 1997.

\bibitem[MR87]{McC_Rob}
John~Coulter McConnell and J.~Chris Robson.
\newblock {\em Noncommutative {N}oetherian rings}.
\newblock Pure and Applied Mathematics (New York). John Wiley \& Sons, Ltd.,
  Chichester, 1987.
\newblock With the cooperation of L. W. Small, A Wiley-Interscience
  Publication.

\bibitem[Meb89]{Mebkhout_book}
Zoghman Mebkhout.
\newblock {\em Le formalisme des six op\'{e}rations de {G}rothendieck pour les
  {$\mathscr{D}_X$}-modules coh\'{e}rents}, volume~35 of {\em Travaux en
  Cours}.
\newblock Hermann, Paris, 1989.
\newblock With supplementary material by the author and L. Narv\'{a}ez Macarro.

\bibitem[MNM91]{MNM}
Zoghman Mebkhout and Luis Narv\'{a}ez-Macarro.
\newblock La th\'{e}orie du polyn\^{o}me de {B}ernstein-{S}ato pour les
  alg\`ebres de {T}ate et de {D}work-{M}onsky-{W}ashnitzer.
\newblock {\em Ann. Sci. \'{E}cole Norm. Sup. (4)}, 24(2):227--256, 1991.

\bibitem[Mus87]{Musson}
Ian~M. Musson.
\newblock Rings of diferential operators on invariant rings of tori.
\newblock {\em Trans. Amer. Math. Soc.}, 303(2):805--827, 1987.


\bibitem[Mus09]{MustataBSprime}
Mircea Musta{\c t}{\u a}.
\newblock Bernstein-{S}ato polynomials in positive characteristic.
\newblock {\em J. Algebra}, 321(1):128--151, 2009.

\bibitem[Mus]{Mustata2019}
Mircea Musta{\c{t}}{\v{a}},
\newblock Bernstein-{S}ato polynomials for general ideals vs principal ideals.
\newblock Preprint, \href{https://arxiv.org/abs/1906.03086 }{arXiv:1906.03086}, 2019.




\bibitem[MTW05]{MTW2005}
Mircea Musta{\c{t}}{\v{a}}, Shunsuke Takagi, and Kei-ichi Watanabe.
\newblock F-thresholds and {B}ernstein-{S}ato polynomials.
\newblock In {\em European {C}ongress of {M}athematics}, pages 341--364. Eur.
  Math. Soc., Z{\"u}rich, 2005.




\bibitem[Nad90]{Nad90}
Alan~Michael Nadel.
\newblock Multiplier ideal sheaves and {K}\"{a}hler-{E}instein metrics of
  positive scalar curvature.
\newblock {\em Ann. of Math. (2)}, 132(3):549--596, 1990.

\bibitem[Nak70]{Nakai}
Yoshikazu Nakai.
\newblock High order derivations. {I}.
\newblock {\em Osaka Math. J.}, 7:1--27, 1970.



\bibitem[Nic10]{Nicaise}
Johannes Nicaise.
\newblock An introduction to {$p$}-adic and motivic zeta functions and the
  monodromy conjecture.
\newblock In {\em Algebraic and analytic aspects of zeta functions and
  {$L$}-functions}, volume~21 of {\em MSJ Mem.}, pages 141--166. Math. Soc.
  Japan, Tokyo, 2010.
  
  \bibitem[NB13]{NBDT}
Luis N{\'u}{\~n}ez-Betancourt.
\newblock On certain rings of differentiable type and finiteness properties of
  local cohomology.
\newblock {\em J. Algebra}, 379:1--10, 2013.

\bibitem[Oak97a]{Oaku97}
Toshinori Oaku.
\newblock An algorithm of computing {$b$}-functions.
\newblock {\em Duke Math. J.}, 87(1):115--132, 1997.

\bibitem[Oak97b]{OakuLC}
Toshinori Oaku.
\newblock Algorithms for {$b$}-functions, restrictions, and algebraic local
  cohomology groups of {$D$}-modules.
\newblock {\em Adv. in Appl. Math.}, 19(1):61--105, 1997.

\bibitem[Oak97c]{OakuLC2}
Toshinori Oaku.
\newblock Algorithms for the {$b$}-function and {$D$}-modules associated with a
  polynomial.
\newblock volume 117/118, pages 495--518. 1997.
\newblock Algorithms for algebra (Eindhoven, 1996).

\bibitem[Oak97d]{Oaku1997}
Toshinori Oaku.
\newblock Algorithms for the {$b$}-function and {$D$}-modules associated with a
  polynomial.
\newblock {\em J. Pure Appl. Algebra}, 117/118:495--518, 1997.

\bibitem[Oak18]{OakuLC3}
Toshinori Oaku.
\newblock Localization, local cohomology, and the {$b$}-function of a
  {$D$}-module with respect to a polynomial.
\newblock In {\em The 50th anniversary of {G}r\"{o}bner bases}, volume~77 of
  {\em Adv. Stud. Pure Math.}, pages 353--398. Math. Soc. Japan, Tokyo, 2018.

\bibitem[Put18]{Put}
Tony~J. Puthenpurakal.
\newblock On the ring of differential operators of certain regular domains.
\newblock {\em Proc. Amer. Math. Soc.}, 146(8):3333--3343, 2018.

\bibitem[QG20]{EamonBSmon}
Eamon Quinlan-Gallego.
\newblock Bernstein-{S}ato roots for monomial ideals in prime characteristic.
\newblock {\em Nagoya Math. J.}, 2020.
\newblock \href{
  https://doi.org/10.1017/nmj.2020.3}{doi.org/10.1017/nmj.2020.3}.

\bibitem[QG21]{EamonBSpos}
Eamon Quinlan-Gallego.
\newblock Bernstein-{S}ato theory for arbitrary ideals in positive
  characteristic.
  \newblock {\em Trans. Amer. Math. Soc.}, 374(3):1623--1660, 2021.



\bibitem[RSW18]{RSW18}
Thomas Reichelt, Christian Sevenheck, and Uli Walther.
\newblock On the {$b$}-functions of hypergeometric systems.
\newblock {\em Int. Math. Res. Not. IMRN}, (21):6535--6555, 2018.


\bibitem[Rot09]{Rot}
Joseph~J. Rotman.
\newblock {\em An introduction to homological algebra}.
\newblock Universitext. Springer, New York, second edition, 2009.

\bibitem[Sab87a]{Sabbah87}
Claude Sabbah.
\newblock {$D$}-modules et cycles \'{e}vanescents (d'apr\`es {B}. {M}algrange
  et {M}. {K}ashiwara).
\newblock In {\em G\'{e}om\'{e}trie alg\'{e}brique et applications, {III} ({L}a
  {R}\'{a}bida, 1984)}, volume~24 of {\em Travaux en Cours}, pages 53--98.
  Hermann, Paris, 1987.

\bibitem[Sab87b]{Sabbah_ideal}
Claude Sabbah.
\newblock Proximit\'{e} \'{e}vanescente. {II}. \'{E}quations fonctionnelles
  pour plusieurs fonctions analytiques.
\newblock {\em Compositio Math.}, 64(2):213--241, 1987.

\bibitem[Sai86]{Saitomhm}
Morihiko Saito.
\newblock Mixed {H}odge modules.
\newblock {\em Proc. Japan Acad. Ser. A Math. Sci.}, 62(9):360--363, 1986.

\bibitem[Sai89]{saito89}
Morihiko Saito.
\newblock On the structure of {B}rieskorn lattice.
\newblock {\em Ann. Inst. Fourier (Grenoble)}, 39(1):27--72, 1989.

\bibitem[Sai94]{Sai94}
Morihiko Saito.
\newblock On microlocal {$b$}-function.
\newblock {\em Bull. Soc. Math. France}, 122(2):163--184, 1994.

\bibitem[Sai09]{Saito_survey}
Morihiko Saito.
\newblock On {$b$}-function, spectrum and multiplier ideals.
\newblock In {\em Algebraic analysis and around}, volume~54 of {\em Adv. Stud.
  Pure Math.}, pages 355--379. Math. Soc. Japan, Tokyo, 2009.

\bibitem[Sai15]{Saitorootnojump}
Morihiko Saito.
\newblock D-modules generated by rational powers of holomorphic functions.
\newblock {\em arXiv preprint arXiv:1507.01877}, 2015.

\bibitem[Sai16]{Saito_BS_arrangements}
Morihiko Saito.
\newblock Bernstein-{S}ato polynomials of hyperplane arrangements.
\newblock {\em Selecta Math. (N.S.)}, 22(4):2017--2057, 2016.


\bibitem[Sat90]{MR1086566}
Mikio Sato.
\newblock Theory of prehomogeneous vector spaces (algebraic part)---the
  {E}nglish translation of {S}ato's lecture from {S}hintani's note.
\newblock {\em Nagoya Math. J.}, 120:1--34, 1990.
\newblock Notes by Takuro Shintani, Translated from the Japanese by Masakazu
  Muro.

\bibitem[SKKO81]{MR595585}
Mikio Sato, Masaki Kashiwara, Tatsuo Kimura, and Toshio Oshima.
\newblock Microlocal analysis of prehomogeneous vector spaces.
\newblock {\em Invent. Math.}, 62(1):117--179, 1980/81.



\bibitem[Sch95]{Schwarz}
Gerald~W. Schwarz.
\newblock Lifting differential operators from orbit spaces.
\newblock {\em Ann. Sci. \'Ecole Norm. Sup. (4)}, 28(3):253--305, 1995.

\bibitem[Sch11]{TestQGor}
Karl Schwede.
\newblock Test ideals in non-{$\mathbb{Q}$}-{G}orenstein rings.
\newblock {\em Trans. Amer. Math. Soc.}, 363(11):5925--5941, 2011.

\bibitem[Siu01]{Siu01}
Yum-Tong Siu.
\newblock Very ampleness part of {F}ujita's conjecture and multiplier ideal
  sheaves of {K}ohn and {N}adel.
\newblock In {\em Complex analysis and geometry ({C}olumbus, {OH}, 1999)},
  volume~9 of {\em Ohio State Univ. Math. Res. Inst. Publ.}, pages 171--191. de
  Gruyter, Berlin, 2001.



\bibitem[Smi81]{PaulSmith}
S.~Paul Smith.
\newblock An example of a ring {M}orita equivalent to the {W}eyl algebra
  {$A_{1}$}.
\newblock {\em J. Algebra}, 73(2):552--555, 1981.

\bibitem[Smi87]{SmithSP}
S.~Paul Smith.
\newblock The global homological dimension of the ring of differential
  operators on a nonsingular variety over a field of positive characteristic.
\newblock {\em J. Algebra}, 107(1):98--105, 1987.



\bibitem[SS88]{SmithStafford}
S.~Paul Smith and J.~Toby Stafford.
\newblock Differential operators on an affine curve.
\newblock {\em Proc. London Math. Soc. (3)}, 56(2):229--259, 1988.

\bibitem[Smi95]{DModFSplit}
Karen~E. Smith.
\newblock The {$D$}-module structure of {$F$}-split rings.
\newblock {\em Math. Res. Lett.}, 2(4):377--386, 1995.

\bibitem[Sta14]{Stadnik}
Theodore~J. Stadnik, Jr.
\newblock The {V}-filtration for tame unit {$F$}-crystals.
\newblock {\em Selecta Math. (N.S.)}, 20(3):855--883, 2014.

\bibitem[Tor02]{Torrelli02}
Tristan Torrelli.
\newblock Polyn\^{o}mes de {B}ernstein associ\'{e}s \`a une fonction sur une
  intersection compl\`ete \`a singularit\'{e} isol\'{e}e.
\newblock {\em Ann. Inst. Fourier (Grenoble)}, 52(1):221--244, 2002.

\bibitem[Tor03]{Torrelli03}
Tristan Torrelli.
\newblock Bernstein polynomials of a smooth function restricted to an isolated
  hypersurface singularity.
\newblock {\em Publ. Res. Inst. Math. Sci.}, 39(4):797--822, 2003.

\bibitem[Tra99]{Travesmono}
William~N. Traves.
\newblock Differential operators on monomial rings.
\newblock {\em J. Pure Appl. Algebra}, 136(2):183--197, 1999.

\bibitem[Tri97]{Tripp}
J.~Raymond Tripp.
\newblock Differential operators on {S}tanley-{R}eisner rings.
\newblock {\em Trans. Amer. Math. Soc.}, 349(6):2507--2523, 1997.

\bibitem[UCJ04]{UchaCastro}
Jos\'e~M. Ucha and Francisco~J. Castro-Jim\'{e}nez.
\newblock On the computation of {B}ernstein-{S}ato ideals.
\newblock {\em J. Symbolic Comput.}, 37(5):629--639, 2004.

\bibitem[Var80]{Var80}
Alexandre~N. Varchenko.
\newblock Gauss-{M}anin connection of isolated singular point and {B}ernstein
  polynomial.
\newblock {\em Bull. Sci. Math. (2)}, 104(2):205--223, 1980.

\bibitem[Var81]{Var81}
Alexandre~N. Varchenko.
\newblock Asymptotic {H}odge structure on vanishing cohomology.
\newblock {\em Izv. Akad. Nauk SSSR Ser. Mat.}, 45(3):540--591, 688, 1981.

\bibitem[Wal04]{Wall04}
C.~T.~C. Wall.
\newblock {\em Singular points of plane curves}, volume~63 of {\em London
  Mathematical Society Student Texts}.
\newblock Cambridge University Press, Cambridge, 2004.

\bibitem[Wal05]{WaltherBS}
Uli Walther.
\newblock Bernstein-{S}ato polynomial versus cohomology of the {M}ilnor fiber
  for generic hyperplane arrangements.
\newblock {\em Compos. Math.}, 141(1):121--145, 2005.

\bibitem[Wal15]{WSurvey}
Uli Walther.
\newblock Survey on the {$D$}-module {$f^s$}.
\newblock In {\em Commutative algebra and noncommutative algebraic geometry.
  {V}ol. {I}}, volume~67 of {\em Math. Sci. Res. Inst. Publ.}, pages 391--430.
  Cambridge Univ. Press, New York, 2015.
\newblock With an appendix by Anton Leykin.

\bibitem[Wal17]{Walther_inv}
Uli Walther.
\newblock The {J}acobian module, the {M}ilnor fiber, and the {$D$}-module
  generated by {$f^s$}.
\newblock {\em Invent. Math.}, 207(3):1239--1287, 2017.

\bibitem[Yan78]{Yano1978}
Tamaki Yano.
\newblock On the theory of {$b$}-functions.
\newblock {\em Publ. Res. Inst. Math. Sci.}, 14(1):111--202, 1978.

\bibitem[Yan82]{Yano1982}
Tamaki Yano.
\newblock Exponents of singularities of plane irreducible curves.
\newblock {\em Sci. Rep. Saitama Univ. Ser. A}, 10(2):21--28, 1982.

\bibitem[Yek92]{yekutieli.explicit_construction}
Amnon Yekutieli.
\newblock An explicit construction of the {G}rothendieck residue complex.
\newblock {\em Ast\'{e}risque}, (208):127, 1992.
\newblock With an appendix by Pramathanath Sastry.

\bibitem[Zar06]{Zariski_Appendix}
Oscar Zariski.
\newblock {\em The moduli problem for plane branches}, volume~39 of {\em
  University Lecture Series}.
\newblock American Mathematical Society, Providence, RI, 2006.
\newblock With an appendix by Bernard Teissier, Translated from the 1973 French
  original by Ben Lichtin.

\end{thebibliography}
\end{document}